\documentclass[a4paper,11pt,reqno]{amsart}
\usepackage{tikz}

\usepackage{amssymb,upref,mathtools}
\usepackage[mathcal]{euscript}
\usepackage[cmtip,all]{xy}

\newcommand{\bydef}{:=}

\newcommand{\wh}[1]{\widehat{#1}}

\newcommand{\id}{\mathrm{id}}

\DeclareMathOperator*{\ot}{\otimes}
\newcommand{\tr}{\mathrm{tr}}

\newcommand{\cP}{\mathcal{P}}

\newcommand{\cR}{\mathcal{R}}

\newcommand{\cT}{\mathcal{T}}

\newcommand{\cV}{\mathcal{V}}

\newcommand{\frV}{\mathfrak{V}}

\newcommand{\NN}{\mathbb{N}}

\newcommand{\FF}{\mathbb{F}}

\DeclareMathOperator{\Hom}{\mathrm{Hom}}
\DeclareMathOperator{\End}{\mathrm{End}}

\DeclareMathOperator{\Aut}{\mathrm{Aut}}

\DeclareMathOperator{\val}{\mathrm{val}}
\DeclareMathOperator{\Mor}{\mathrm{Mor}}
\DeclareMathOperator{\R}{\mathrm{R}}
\DeclareMathOperator{\skews}{\mathrm{skew}}

\newcommand{\frsl}{{\mathfrak{sl}}}

\newcommand{\frso}{{\mathfrak{so}}}

\newcommand{\frosp}{{\mathfrak{osp}}}

\newcommand{\Ort}{\mathrm{O}}
\newcommand{\SO}{\mathrm{SO}}

\newcommand{\subo}{_{\bar 0}}
\newcommand{\subuno}{_{\bar 1}}

\newcommand{\Cm}{\mathsf{C}}

\newcommand{\GG}{\mathsf{G}}
\newcommand{\mm}{\mathsf{m}}
\newcommand{\VV}{\mathsf{V}} 
\renewcommand{\Hom}{\mathsf{Hom}}
\renewcommand{\Aut}{\mathsf{Aut}}
\renewcommand{\End}{\mathsf{End}}
\renewcommand{\SO}{\mathsf{SO}}
\newcommand{\dimm}{\mathsf{dim\,}}
\renewcommand{\id}{\mathsf{id}}
\newcommand{\II}{\mathsf{I}}
\renewcommand{\tr}{\mathsf{tr}}
\newcommand{\norm}{\mathsf{n}}
\newcommand{\overbar}[1]{\mkern 1.2mu\overline{\mkern-1mu#1\mkern-1mu}\mkern 1mu}

\newtheorem{theorem}[equation]{Theorem}
\newtheorem{proposition}[equation]{Proposition}
\newtheorem{lemma}[equation]{Lemma}
\newtheorem{corollary}[equation]{Corollary}

\theoremstyle{definition}
\newtheorem{df}[equation]{Definition}
\newtheorem{example}[equation]{Example}
\newtheorem{remark}[equation]{Remark}

\numberwithin{equation}{section}

\newenvironment{romanenumerate}
 {\begin{enumerate}
 
 }{\end{enumerate}}

\tikzstyle cross=[preaction={draw=white,-,line width=4pt},thick]

\tikzstyle bolito=[draw,circle,fill=black,minimum size=3pt,inner sep=0pt]
                            
\tikzstyle punto=[draw,circle,fill=black,minimum size=1pt,inner sep=0pt]
                            
\tikzstyle fino=[line width=0.5pt]

\newcommand{\tik}[2]{%
 \begin{tikzpicture}[scale=#1,baseline={([yshift=-3pt]current bounding box.center)},label distance=-3pt] #2 \end{tikzpicture}}

\begin{document}

\title{Cross products, invariants, and centralizers}

\author[Georgia Benkart]{Georgia Benkart$^{\star}$}
\address{Department of Mathematics, University of Wisconsin-Madison, Madison, WI 53706, USA}
\email{benkart@math.wisc.edu}

\author[Alberto Elduque]{Alberto Elduque$^{\star}$}
\thanks{$^{\star}$ Both authors have been supported by the Spanish Ministerio de Econom{\'\i}a y Competitividad and 
Fondo Europeo de Desarrollo Regional (FEDER) MTM 2013-45588-C3-2-P. The second author also acknowledges support 
 by the Diputaci\'on General de Arag\'on -- Fondo Social Europeo (Grupo de Investigaci\'on de \'Algebra). He is also grateful for the hospitality of the Department of Mathematics of the University of Wisconsin-Madison in November 2015. Most of the results of this paper were obtained during that visit.}
\address{Departamento de Matem\'aticas e Instituto Universitario de Matem\'aticas y Aplicaciones,
Universidad de Zaragoza, 50009 Zaragoza, Spain
}
\email{elduque@unizar.es}




\begin{abstract}  An algebra $\mathsf{V}$ with a cross product $\times$ has  dimension 3 or  7.  In this work,  we use 
$3$-tangles to describe, and provide a basis  for, the space of  homomorphisms from $\mathsf{V}^{\otimes n}$ to $\mathsf{V}^{\otimes m}$ that are invariant
under the action of the automorphism group $\mathsf{Aut(V,\times)}$ of $\mathsf{V}$, which is a special orthogonal group when $\mathsf{dim\, V} = 3$,
and a simple algebraic group of type $\mathsf{G}_2$ when $\mathsf{dim \, V} = 7$.   When $m = n$, this gives
a graphical description of the centralizer algebra $\mathsf{End_{Aut(V,\times)}}(\mathsf{V}^{\otimes n})$, and therefore, also 
a graphical realization of the $\mathsf{Aut(V,\times)}$-invariants in $\mathsf{V}^{\otimes 2n}$ equivalent to the 
First Fundamental Theorem of Invariant Theory.   We show how the  $3$-dimensional simple Kaplansky Jordan
superalgebra can be interpreted as a cross product (super)algebra and use $3$-tangles to obtain a 
graphical description of  the centralizers  and invariants of the Kaplansky superalgebra
relative to the action of the special orthosymplectic group.    
\end{abstract}

\dedicatory{Dedicated to Efim Zelmanov on the occasion of his $60$th birthday.}

\maketitle  
\section{Introduction}  A cross product algebra $(\VV,\mathsf{b},\times)$  is  a finite-dimensional  vector space $\VV$ over a field $\FF$
(assumed to have characteristic different from 2) with  a nondegenerate symmetric bilinear form $\mathsf{b}$  and a bilinear multiplication $\VV\times \VV\rightarrow \VV$, $(x,y)\mapsto x\times y$,  that satisfies
\begin{align*}
& \mathsf{b}( x\times y,x)=0,\\
& x\times x=0,\\
& \mathsf{b}( x\times y,x\times y)= \mathsf{b}( x,x)\mathsf{b}( y,y)- \mathsf{b}( x,y)\mathsf{b}( y,x), 
\end{align*}
for any $x,y\in \VV$.    
Nonzero cross products exist only if $\dimm_\FF \VV=3$ or $7$ (see \cite{BrownGray}, \cite{Rost04}),
and when $\FF$ is the field of real numbers, the cross product 
is the familiar one from calculus in dimension $3$ if $\mathsf{b}$ is positive definite.
We relate cross product algebras to certain $3$-tangle categories and give a graphical
realization of the invariants and centralizer algebras of tensor powers of $\VV$ under the action of 
its automorphism group $\Aut(\VV, \times)$. 

The \emph{$3$-tangle category}  $\mathcal T$ has as objects the finite sets $[n] = \{1,2,\dots, n\}$ for $n\in\NN = \{0,1,2,\ldots\}$, where 
$[0] = \emptyset$.  For any $n,m\in\NN$,  the morphisms $\Mor_{\cT}([n],[m])$ are $\FF$-linear combinations of 
 $3$-tangles, and they are generated through composition and disjoint union from the basic morphisms (basic $3$-tangles) in \eqref{eq:basic1}  and 
 \eqref{eq:basic2}. This gives a graphical calculus that enables us to describe the space $\Hom_{\Aut(\VV,\times)}(\VV^{\otimes n},\VV^{\otimes m})$
 of $\Aut(\VV,\times)$-homomorphisms on tensor powers of $\VV$.
  When $\dimm_\FF \VV = 7$,  the group  $\Aut(\VV,\times)$ is a simple algebraic group of type $\GG_2$, and $\VV$ is its
 natural $7$-dimensional module (its smallest nontrivial irreducible module).    When $\dimm_\FF \VV = 3$,   
 $\Aut(\VV,\times)$ is the special orthogonal group $\SO(\VV,\mathsf{b})$.   
 
From  the properties of the cross product, we construct three homomorphisms (when $\dimm_\FF \VV = 7$, they are given in 
\eqref{eq:c0}, \eqref{eq:c1_c2}, and when $\dimm_\FF \VV = 3$, in \eqref{eq:c4},\eqref{eq:c3}). 
Applying  methods similar to those in (\cite{Kup2}, \cite{Kup1}, \cite{CKR}), 
we show in each case (see Theorems \ref{th:G2} and \ref{th:sl2} for the precise statements) that these homomorphisms correspond to a set $\Gamma^*$ consisting of  three $3$-tangle relations,  and  the following result holds.  In the statement, the $3$-tangles
 must satisfy some additional constraints.  When $\dimm_\FF\VV=3$, these constraints are 
  incorporated in the definition of ``normalized" $3$-tangles.   
 
 \begin{theorem}\label{th:general} 
Let $n,m\in\NN$ and assume that the characteristic of $\FF$ is $0$.   Let $(\VV,\mathsf{b},\times)$ be a vector space $\VV$ endowed
with a nonzero cross product $x\times y$ relative to the nondegenerate symmetric bilinear form $\mathsf{b}$.
\begin{itemize}
\item[{\rm (a)}]  The classes modulo $\Gamma^*$ of (normalized) $3$-tangles $[n]\rightarrow [m]$ without crossings and without any of the subgraphs in \eqref{eq:avoided} form a basis of $\Mor_{\cT_{\Gamma^*}}([n],[m])$.

\item[{\rm (b)}]   There is a functor  $\cR_{\Gamma^*}$  giving a linear isomorphism 
\[
\Mor_{\cT_{\Gamma^*}}([n],[m])\rightarrow \Hom_{\Aut(\VV,\times)}(\VV^{\otimes n},\VV^{\otimes m}).
\]

\item[{\rm (c)}]  The $3$-tangles $[n]\rightarrow [n]$ as in part \textrm{(a)} give a basis of the centralizer algebra $$\End_{\Aut(\VV,\times)}\bigl(\VV^{\otimes n}\bigr)\simeq \Mor_{\cT_{\Gamma^*}}([n],[n]).$$ 
\end{itemize}
\end{theorem}
The centralizer algebra $\End_{\Aut(\VV,\times)}\bigl(\VV^{\otimes n}\bigr)$ is isomorphic to the space of $\Aut(\VV,\times)$-invariants
in $\VV^{\otimes 2n}$, and part (c) gives a graphical realization of that space equivalent to the First Fundamental Theorem of Invariant Theory
results in (\cite{Weyl}, \cite{Schwarz}, \cite{ZSh}).

Section \ref{se:pr} establishes basic facts about bilinear forms on tensor powers and Section \ref{se:3tangles} about $3$-tangles.
In Sections \ref{se:7dim} and  \ref{se:3_dim},  we prove the above theorem for dimensions $7$ and $3$ respectively. 
The final section  extends these considerations to the Kaplansky superalgebra.   This $3$-dimensional simple Jordan superalgebra
has a multiplication that can be interpreted as a cross product relative to a nondegenerate supersymmetric bilinear form $\mathsf{b}$.    With modifications to the arguments that take into
account the fact that the cross product is supercommutative,   $x \times y = (-1)^{xy} y \times x$, and so depends on the parity of the homogeneous elements $x,y$, we are able in 
Theorem \ref{th:Kap} to prove an analogue of Theorem \ref{th:general}, where the automorphism group $\Aut(\VV,\times)$ is the special orthosymplectic group 
$\mathsf{SOSp}(\VV,\mathsf{b})$ (or equivalently, the group could be replaced by the orthosymplectic Lie superalgebra $\frosp(\VV,\mathsf{b})$ of derivations on $\VV$). 
This gives a graphical version of the  First Fundamental Theorem of Invariant Theory for $\mathsf{SOSp}(\VV,\mathsf{b})$  (see \cite{Ser92}, \cite{LZ15}).    

In all the $3$-dimensional cases,   the centralizer algebra $\End_{\Aut(\VV,\times)}\bigl(\VV^{\otimes n}\bigr)$ has dimension equal to
the number $a(2n)$ of Catalan partitions  i.e., partitions of $2n$ points on a circle into a nonintersecting union of subsets, each  of size at least 2, whose convex hulls are disjoint (sequence $\#$A099251 in \cite{OEIS}).    

The cross product and certain properties that  it satisfies are key to deriving the $3$-tangle results in this paper. 
This approach differs from that of Kuperberg \cite{Kup2},  which derives quantum $\mathsf{G}_2$ link invariants
from the Jones polynomial starting from its simplest formulation in terms of the Kauffman bracket.  Kuperberg's quantum $\mathsf{G}_2$
invariant  is defined not only  
on knots and links, but also on regular isotopy classes of certain graphs with endpoints and all other nodes
trivalent, akin to the $3$-tangles here,  and is a particular case of the Reshetikhin-Turaev invariant of ribbon graphs.


\section{Preliminaries}\label{se:pr}

Let $\VV$ be a finite-dimensional nonzero vector space over a field $\FF$ endowed with a nondegenerate bilinear form $\mathsf{b}:\VV\times \VV\rightarrow \FF$, which we will identify with the linear map $\mathsf{b}:\VV^{\otimes 2}\rightarrow \FF$
given by $\mathsf{b}( u \ot v) = \mathsf{b}( u,v)$. For $n\in \NN=\{0,1,2,\ldots\}$, consider the linear map
\[
\begin{split}
\mathsf{b}_n:\VV^{\otimes 2n}&\longrightarrow \FF\\
v_1\otimes\cdots\otimes v_{2n}&\mapsto \prod_{i=1}^n \mathsf{b}( v_i,v_{2n+1-i}),
\end{split}
\]
with $\mathsf{b}_0=\id$ (the identity map on $\VV^{\otimes 0}\bydef \FF$). For $n\geq 1$, $\mathsf{b}_n$ will be represented graphically as follows:
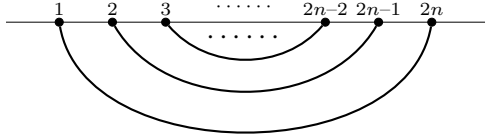
\[
\tik{0.7}{%
\draw[very thin] (-1,0) -- (8,0);
\node[bolito] (a1) at (0,0) [label=above: {\tiny $1$}] {};
\node[bolito] (a2) at (1,0) [label=above: {\tiny $2$}] {};
\node[bolito] (a3) at (2,0) [label=above: {\tiny $3$}] {};
\node[bolito] (a4) at (5,0) [label=above: {\tiny $2n$--$2$}] {};
\node[bolito] (a5) at (6,0) [label=above: {\tiny $2n$--$1$}] {};
\node[bolito] (a6) at (7,0) [label=above: {\tiny $2n$}] {};
\node at (3.5,0.3) {\tiny $\cdots\cdots$};
\node at (3.5,-0.3) {$\cdots\cdots$};
\draw[thick] (a1) to [bend right=80] (a6);
\draw[thick] (a2) to [bend right=60] (a5);
\draw[thick] (a3) to [bend right=50] (a4);}
\]
where no crossings appear. (This will be important later on.)

Also, $\mathsf{b}_n$ will be identified with the nondegenerate bilinear form \\
$\VV^{\otimes n}\times \VV^{\otimes n}\rightarrow \FF$ such that
\[
\mathsf{b}_n(v_1\otimes \cdots\otimes v_n,w_1\otimes\cdots w_n)=\prod_{i=1}^n \mathsf{b}( v_i,w_{n+1-i}),
\]
for any $v_1,\ldots,v_n,w_1,\ldots,w_n\in \VV$. If $\mathsf{b}$ is symmetric, so is $\mathsf{b}_n$ for any $n\ge 1$.

\begin{df}\label{df:transpose}
Given $n,m\in\NN$ and a linear map $f\in\Hom_\FF(\VV^{\otimes n},\VV^{\otimes m})$, the \emph{transpose} of $f$ is the linear map $f^t\in\Hom_\FF(\VV^{\otimes m},\VV^{\otimes n})$ such that
\[
\mathsf{b}_m\bigl(f(x),y\bigr)=\mathsf{b}_n\bigl(x,f^t(y)\bigr),
\]
for any $x\in \VV^{\otimes n}$ and $y\in \VV^{\otimes m}$.
\end{df}
 
In other words, 
\begin{equation}\label{eq:bmfbnft}
\mathsf{b}_m\circ (f\otimes \id_m)=\mathsf{b}_n\circ(\id_n\otimes f^t)
\end{equation}
in $\Hom_\FF(\VV^{\otimes (n+m)},\FF)$, where $\id_n$ denotes the identity map on $\VV^{\otimes n}$.
 
Equation \eqref{eq:bmfbnft} will be represented graphically as follows:   
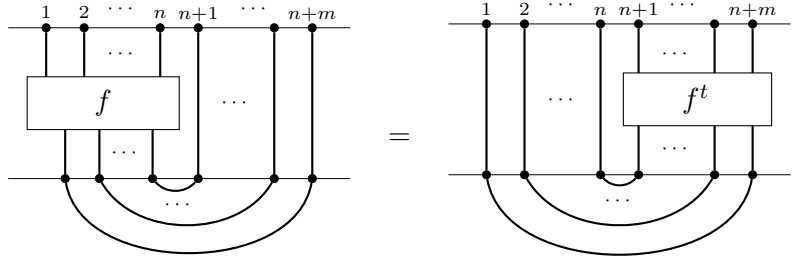
\[
\tik{0.5}{%
\draw[very thin] (-1,0) -- (8,0);
\draw[very thin] (-1,-4) --(8,-4);
\node[bolito] (a1) at (0,0) [label=above: \raisebox{1pt}{\tiny $1$}] {};
\node[bolito] (a2) at (1,0) [label=above: \raisebox{1pt}{\tiny $2$}] {};
\node[bolito] (a3) at (3,0) [label=above: \raisebox{1pt}{\tiny $n$}] {};
\node[bolito] (a4) at (4,0) [label=above: {\tiny $n$+$1$}] {};
\node[bolito] (a5) at (6,0)  {};
\node[bolito] (a6) at (7,0) [label=above: {\tiny $n$+$m$}] {};
\node at (2,0.5) {\tiny $\cdots$};
\node at (5.5,0.5) {\tiny $\cdots$};
\node at (2,-0.7) {\tiny $\cdots$};
\node at (2.1,-3.3) {\tiny $\cdots$};
\node at (5,-2) {\tiny $\cdots$};
\node[bolito] (b1) at (0.5,-4) {};
\node[bolito] (b2) at (1.4,-4) {};
\node[bolito] (b3) at (2.8,-4) {};
\node[bolito] (b4) at (4,-4) {};
\node[bolito] (b5) at (6,-4) {};
\node[bolito] (b6) at (7,-4) {};
\draw (-0.5,-2.7) rectangle (3.5,-1.3);
\node at (1.5,-2) {$f$};
\draw[thick] (a1) to (0,-1.3);
\draw[thick] (a2) to (1,-1.3);
\draw[thick] (a3) to (3,-1.3);
\draw[thick] (b1) to (0.5,-2.7);
\draw[thick] (b2) to (1.4,-2.7);
\draw[thick] (b3) to (2.8,-2.7);
\draw[thick] (a4) to (b4);
\draw[thick] (a5) to (b5);
\draw[thick] (a6) to (b6);
\draw[thick] (b1) to [bend right=80] (b6);
\draw[thick] (b2) to [bend right=60] (b5);
\draw[thick] (b3) to [bend right=50] (b4);
\node at (3.5,-4.7) {\tiny $\cdots$};
}   
\quad
=
\quad
\tik{0.5}{%
\draw[very thin] (-1,0) -- (8,0);
\draw[very thin] (-1,-4) --(8,-4);
\node[bolito] (a1) at (0,0) [label=above: \raisebox{1pt}{\tiny $1$}] {};
\node[bolito] (a2) at (1,0) [label=above: \raisebox{1pt}{\tiny $2$}] {};
\node[bolito] (a3) at (3,0) [label=above: \raisebox{1pt}{\tiny $n$}] {};
\node[bolito] (a4) at (4,0) [label=above: {\tiny $n$+$1$}] {};
\node[bolito] (a5) at (6,0)  {};
\node[bolito] (a6) at (7,0) [label=above: {\tiny $n$+$m$}] {};
\node at (2,0.5) {\tiny $\cdots$};
\node at (5.2,0.5) {\tiny $\cdots$};
\node at (5,-0.8) {\tiny $\cdots$};
\node at (5,-3.3) {\tiny $\cdots$};
\node at (2,-2) {\tiny $\cdots$};     
\node[bolito] (b1) at (0,-4) {};
\node[bolito] (b2) at (1,-4) {};
\node[bolito] (b3) at (3,-4) {};
\node[bolito] (b4) at (4,-4) {};
\node[bolito] (b5) at (6,-4) {};
\node[bolito] (b6) at (7,-4) {};
\draw (3.6,-2.7) rectangle (7.5,-1.3);
\node at (5.5,-2) {$f^t$};
\draw[thick] (a4) to (4,-1.3);
\draw[thick] (a5) to (6,-1.3);
\draw[thick] (a6) to (7,-1.3);
%
\draw[thick] (b4) to (4,-2.7);
\draw[thick] (b5) to (6,-2.7);
\draw[thick] (b6) to (7,-2.7);
\draw[thick] (a1) to (b1);
\draw[thick] (a2) to (b2);
\draw[thick] (a3) to (b3);
\draw[thick] (b1) to [bend right=80] (b6);
\draw[thick] (b2) to [bend right=60] (b5);
\draw[thick] (b3) to [bend right=50] (b4);
\node at (3.5,-4.7) {\tiny $\cdots$};
}
\]      
Given linear maps $f\in\Hom_\FF(\VV^{\otimes n},\VV^{\otimes q})$, $g\in\Hom_\FF(\VV^{\otimes m},\VV^{\otimes p})$, the linear map $f\otimes g\in \Hom_\FF(\VV^{\otimes (n+m)},\VV^{\otimes(q+p)})$ will be represented as follows:
\[
\tik{0.5}{%
\draw[very thin] (-1,0) -- (8,0);
\draw[very thin] (-1,-4) --(8,-4);
\node[bolito] (a1) at (0,0) [label=above: \raisebox{1pt}{\tiny $1$}] {};
\node[bolito] (a2) at (1,0) [label=above: \raisebox{1pt}{\tiny $2$}] {};
\node[bolito] (a3) at (3,0) [label=above: \raisebox{1pt}{\tiny $n$}] {};
\node[bolito] (a4) at (4,0) [label=above: {\tiny $n$+$1$}] {};
\node[bolito] (a5) at (7,0) [label=above: {\tiny $n$+$m$}] {};
\node at (2,0.5) {\tiny $\cdots$};
\node at (5.5,0.5) {\tiny $\cdots$};
\node at (2,-0.8) {\tiny $\cdots$};
\node at (5.5,-0.8) {\tiny $\cdots$};
\node[bolito] (b1) at (0.5,-4) {};
\node[bolito] (b2) at (2.5,-4) {};
\node[bolito] (b3) at (4.5,-4) {};
\node[bolito] (b4) at (6.5,-4) {};
\node at (1.5,-3.3) {\tiny $\cdots$};
\node at (5.5,-3.3) {\tiny $\cdots$};
\draw (-0.3,-2.7) rectangle (3.3,-1.3);
\node at (1.5,-2) {$f$};
\draw (3.7,-2.7) rectangle (7.3,-1.3);
\node at (5.5,-2) {$g$};
\draw[thick] (a1) to (0,-1.3);
\draw[thick] (a2) to (1,-1.3);
\draw[thick] (a3) to (3,-1.3);
\draw[thick] (a4) to (4,-1.3);
\draw[thick] (a5) to (7,-1.3);
%
\draw[thick] (b1) to (0.5,-2.7);
\draw[thick] (b2) to (2.5,-2.7);
\draw[thick] (b3) to (4.5,-2.7);
\draw[thick] (b4) to (6.5,-2.7);
}
\] 
The next result is clear.

\begin{proposition}
\null\qquad
\begin{itemize}
\item $(\id_n)^t=\id_n$ for any $n\in\NN$.
\item $(f\circ g)^t=g^t\circ f^t$, for any $n,m,p\in\NN$, $f\in\Hom_\FF(\VV^{\otimes m},\VV^{\otimes p})$, and $g\in\Hom_\FF(\VV^{\otimes n},\VV^{\otimes m})$.
\item $(f\otimes g)^t=g^t\otimes f^t$ for any $n,m,p,q\in\NN$, $f\in\Hom_\FF(\VV^{\otimes n},\VV^{\otimes q})$, and $g\in\Hom_\FF(\VV^{\otimes m},\VV^{\otimes p})$.
\item If $\mathsf{b}$ is symmetric, $(f^t)^t=f$ for any $n\in \NN$ and $f\in\Hom_\FF(\VV^{\otimes n},\VV^{\otimes m})$.
\end{itemize}
\end{proposition}

Assume  $\dimm_\FF \VV=d$,  and let $\{u_i\}_{i=1}^d$ and $\{v_i\}_{i=1}^d$ be dual bases relative to $\mathsf{b}$ so that $\mathsf{b}( u_i,v_j)=\delta_{i,j}$ ($1$ if $i=j$ and $0$ otherwise).   

For any $n\in \NN$, let $[d]^n\bydef \{1,\ldots,d\}^n$ ($[d]^0=\{\emptyset\}$), and for any $\alpha=\bigl(\alpha_1,\ldots,\alpha_n\bigr)\in [d]^n$, let $\overbar\alpha\bydef \bigl(\alpha_n,\ldots,\alpha_1\bigr)$, $u_\alpha\bydef u_{\alpha_1}\otimes\cdots\otimes u_{\alpha_n}$, and $v_\alpha\bydef v_{\alpha_1}\otimes\cdots\otimes v_{\alpha_n}$. For $n=0$, set $u_\emptyset=1=v_\emptyset$ and $\overbar\emptyset=\emptyset$. Then $\{u_\alpha\}_{\alpha\in [d]^n}$ and $\{v_{\overbar\alpha}\}_{\alpha\in [d]^n}$ are dual bases of $\VV^{\otimes n}$ relative to $\mathsf{b}_n$. Hence, for any $x\in \VV^{\otimes n}$,
\[
x=\sum_{\alpha\in [d]^n}\mathsf{b}_n(u_\alpha\otimes x)v_{\overbar\alpha}
=\sum_{\alpha\in [d]^n} \mathsf{b}_n(x\otimes v_{\overbar \alpha})u_\alpha.
\]

\noindent For any $n,m\in\NN$, $f\in\Hom_\FF(\VV^{\otimes n},\VV^{\otimes m})$, and $\alpha\in [d]^n$, $f(u_\alpha)=\sum_{\beta\in [d]^m}a_{\alpha,\beta}u_\beta$, with $a_{\alpha,\beta}\in \FF$ for any $\beta$. Then, for any $\beta\in [d]^m$:
\begin{equation}\label{eq:ftv}
f^t(v_{\overbar\beta})=\sum_{\alpha\in [d]^n}a_{\alpha,\beta}v_{\overbar\alpha},
\end{equation}
because
\[
a_{\alpha,\beta}=\mathsf{b}_m\bigl(f(u_\alpha)\otimes v_{\overbar\beta}\bigr)
=\mathsf{b}_n\bigl(u_\alpha\otimes f^t(v_{\overbar\beta})\bigr).
\]

\begin{proposition}\label{pr:tPhiPsi} 
\null\qquad
\begin{romanenumerate}
\item For any $n\in\NN$, $\mathsf{b}_n^t(1)=\sum_{\alpha\in [d]^n}v_{\overbar\alpha}\otimes u_\alpha$.
\pagebreak[2]
\item For any $n,m\in\NN$, the linear map
\[
\begin{split}
\Phi_{n,m}:\Hom_\FF(\VV^{\otimes n},\VV^{\otimes m})&\longrightarrow \Hom_\FF(\VV^{\otimes(n+m)},\FF)\\
 f\quad&\mapsto\quad \mathsf{b}_m\circ(f\otimes \id_m),
\end{split}
\]
is a linear isomorphism with inverse given by
\[
\begin{split}
\Psi_{n,m}:\Hom_\FF(\VV^{\otimes(n+m)},\FF)&\longrightarrow \Hom_\FF(\VV^{\otimes n},\VV^{\otimes m})\\
h\quad&\mapsto\quad (h\otimes \id_m)\circ(\id_n\otimes \mathsf{b}_m^t).
\end{split}
\]

\item For any $n,m\in\NN$ and $f\in \Hom_\FF(\VV^{\otimes n},\VV^{\otimes m})$,
\[
f^t=(\id_n\otimes \mathsf{b}_m)\circ(\id_n\otimes f\otimes \id_m)\circ (\mathsf{b}_n^t\otimes \id_m).
\]
\end{romanenumerate}
\end{proposition}
\begin{proof}
For (i),  we must prove that $\mathsf{b}_n(x\otimes y)
=\mathsf{b}_{2n}(x\otimes y\otimes \mathsf{b}_n^t(1))$ for any $x,y\in \VV^{\otimes n}$.  But
\[
\begin{split}
\mathsf{b}_n(x\otimes y)&=\mathsf{b}_n\Bigl(x\otimes 
\Big(\sum_{\alpha\in [d]^n}\mathsf{b}_n(y\otimes v_{\overbar\alpha})u_\alpha\Big)\Bigr)\\
   &=\sum_{\alpha\in [d]^n}\mathsf{b}_n(x\otimes u_\alpha)\, \mathsf{b}_n(y\otimes v_{\overbar\alpha})\\
   &=\mathsf{b}_{2n}\Bigl(x\otimes y\otimes\Big(\sum_{\alpha\in [d]^n} v_{\overbar\alpha}\otimes u_\alpha\Big)\Bigr).
\end{split}
\]

(ii) Now, for $n,m\in\NN$ and $f\in \Hom_\FF(\VV^{\otimes n},\VV^{\otimes m})$,
\[
\Bigl(\bigl(\mathsf{b}_m\circ(f\otimes\id_m)\bigr)\otimes \id_m\Bigr)\circ(\id_n\otimes \mathsf{b}_m^t)
  =(\mathsf{b}_m\otimes \id_m)\circ(f\otimes \mathsf{b}_m^t),
\]
and for any $x\in \VV^{\otimes n}$,
\[
x \xmapsto{f\otimes \mathsf{b}_m^t} \sum_{\alpha\in [d]^m}f(x)\otimes v_{\overbar\alpha}\otimes u_\alpha
 \xmapsto{\mathsf{b}_m\otimes \id_m} \sum_{\alpha\in [d]^m}\mathsf{b}_m\bigl(f(x)\otimes v_{\bar\alpha}\bigr)u_\alpha
   =f(x).
\]
Also, for $h\in \Hom_\FF(\VV^{\otimes (n+m)},\FF)$,
\[
\mathsf{b}_m\circ\Bigl(\bigl((h\otimes\id_m)\circ(\id_n\otimes \mathsf{b}_m^t)\bigr)\otimes\id_m\Bigr)
  =\mathsf{b}_m\circ(h\otimes\id_m\otimes\id_m)\circ(\id_n\otimes \mathsf{b}_m^t\otimes\id_m),
\]
and for $x\in \VV^{\otimes n}$ and $y\in \VV^{\otimes m}$,
\begin{multline*}
x\otimes y\xmapsto{\id_n\otimes \mathsf{b}_m^t\otimes\id_m}
  \sum_{\alpha\in [d]^n} x\otimes v_{\overbar\alpha}\otimes u_\alpha\otimes y
  \xmapsto{h\otimes\id_m\otimes\id_m}
  \sum_{\alpha\in[d]^m}h(x\otimes v_{\overbar\alpha})u_\alpha\otimes y \\
  \xmapsto{\mathsf{b}_m}\sum_{\alpha\in [d]^n}h(x\otimes v_{\overbar\alpha})\mathsf{b}_m(u_\alpha\otimes y)
  =\sum_{\alpha\in [d]^m}h(x\otimes \mathsf{b}_m(u_\alpha\otimes y)v_{\bar\alpha})=h(x\otimes y).
\end{multline*}

(iii) Finally, for $n,m\in\NN$, $f\in\Hom_\FF(\VV^{\otimes n},\VV^{\otimes m})$, and $x\in \VV^{\otimes m}$,
\begin{multline*}
x\xmapsto{\mathsf{b}_m^t\otimes \id_m}\sum_{\alpha\in [d]^n}v_{\overbar\alpha}\otimes u_\alpha\otimes x
 \xmapsto{\id_n\otimes f\otimes\id_m}\sum_{\alpha\in [d]^n}v_{\overbar\alpha}\otimes f(u_\alpha)\otimes x\\
 \xmapsto{\id_n\otimes \mathsf{b}_m}\sum_{\alpha\in [d]^n}\mathsf{b}_m\bigl(f(u_\alpha)\otimes x)v_{\overbar\alpha}
 =\sum_{\alpha\in [d]^n}\mathsf{b}_n\bigl(u_\alpha\otimes f^t(x)\bigr)v_{\overbar\alpha}=f^t(x).
\end{multline*}
\end{proof}

The map $\mathsf{b}_n^t:\VV^{\otimes 0}=\FF\mapsto \VV^{\otimes 2n}$ will be represented as follows:
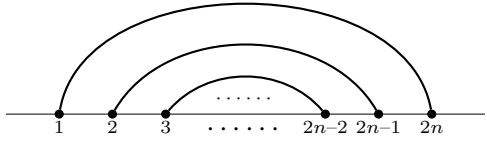
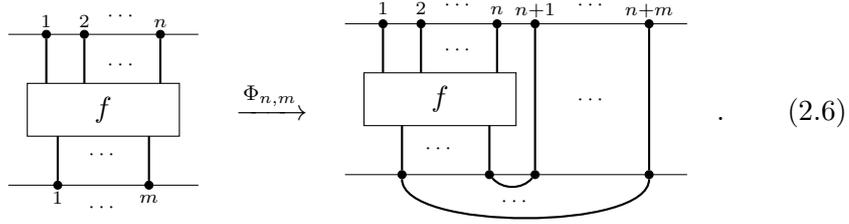
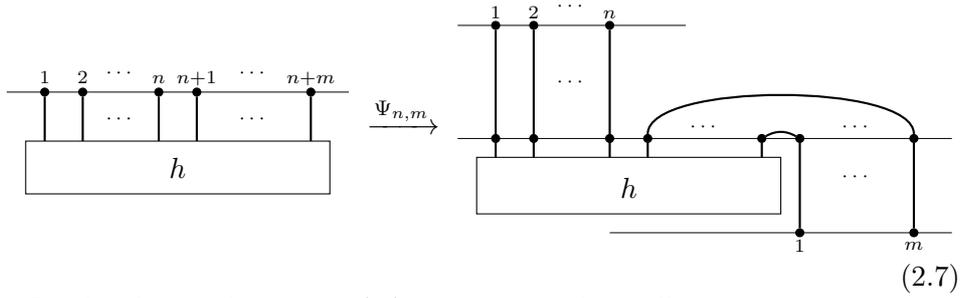
\[
\tik{0.7}{%
\draw[very thin] (-1,0) -- (8,0);
\node[bolito] (a1) at (0,0) [label=below: {\tiny $1$}] {};
\node[bolito] (a2) at (1,0) [label=below: {\tiny $2$}] {};
\node[bolito] (a3) at (2,0) [label=below: {\tiny $3$}] {};
\node[bolito] (a4) at (5,0) [label=below: {\tiny $2n$--$2$}] {};
\node[bolito] (a5) at (6,0) [label=below: {\tiny $2n$--$1$}] {};
\node[bolito] (a6) at (7,0) [label=below: {\tiny $2n$}] {};
\node at (3.5,0.3) {\tiny $\cdots\cdots$};
\node at (3.5,-0.3) {$\cdots\cdots$};
\draw[thick] (a1) to [bend left=80] (a6);
\draw[thick] (a2) to [bend left=60] (a5);
\draw[thick] (a3) to [bend left=50] (a4);}
\]
Then $\Phi_{n,m}$ corresponds to:
\begin{equation}\label{eq:Phinm}
\tik{0.5}{%
\draw[very thin] (-1,0) -- (4,0);
\draw[very thin] (-1,-4) --(4,-4);
\node[bolito] (a1) at (0,0) [label=above: {\tiny $1$}] {};
\node[bolito] (a2) at (1,0) [label=above: {\tiny $2$}] {};
\node[bolito] (a3) at (3,0) [label=above: {\tiny $n$}] {};
\node at (2,0.5) {\tiny $\cdots$};
\node at (2,-0.8) {\tiny $\cdots$};
\node at (1.5,-3.2) {\tiny $\cdots$};
\node at (1.5,-4.6) {\tiny $\cdots$};
\node[bolito] (b1) at (0.3,-4) [label=below: {\tiny $1$}] {};
\node[bolito] (b2) at (2.7,-4) [label=below: {{\vrule width 0pt height 4pt}\tiny $m$}] {};
\draw (-0.5,-2.7) rectangle (3.5,-1.3);
\node at (1.5,-2) {$f$};
%
\draw[thick] (a1) to (0,-1.3);
\draw[thick] (a2) to (1,-1.3);
\draw[thick] (a3) to (3,-1.3);
%
\draw[thick] (b1) to (0.3,-2.7);
\draw[thick] (b2) to (2.7,-2.7);
} 
\quad
\xrightarrow{\Phi_{n,m}}
\quad
\tik{0.5}{%
\draw[very thin] (-1,0) -- (8,0);
\draw[very thin] (-1,-4) --(8,-4);
\node[bolito] (a1) at (0,0) [label=above: \raisebox{1pt}{\tiny $1$}] {};
\node[bolito] (a2) at (1,0) [label=above: \raisebox{1pt}{\tiny $2$}] {};
\node[bolito] (a3) at (3,0) [label=above: \raisebox{1pt}{\tiny $n$}] {};
\node[bolito] (a4) at (4,0) [label=above: {\tiny $n$+$1$}] {};
\node[bolito] (a6) at (7,0) [label=above: {\tiny $n$+$m$}] {};
\node at (2,0.5) {\tiny $\cdots$};
\node at (5.5,0.5) {\tiny $\cdots$};
\node at (2,-0.7) {\tiny $\cdots$};
\node at (1.5,-3.3) {\tiny $\cdots$};
\node at (5.5,-2) {\tiny $\cdots$};
\node[bolito] (b1) at (0.5,-4) {};
\node[bolito] (b3) at (2.8,-4) {};
\node[bolito] (b4) at (4,-4) {};
\node[bolito] (b6) at (7,-4) {};
\draw (-0.5,-2.7) rectangle (3.5,-1.3);
\node at (1.5,-2) {$f$};
\draw[thick] (a1) to (0,-1.3);
\draw[thick] (a2) to (1,-1.3);
\draw[thick] (a3) to (3,-1.3);
\draw[thick] (b1) to (0.5,-2.7);
\draw[thick] (b3) to (2.8,-2.7);
\draw[thick] (a4) to (b4);
\draw[thick] (a6) to (b6);
\draw[thick] (b1) .. controls (0.5,-5.5) and (7,-5.5) .. (b6);
\draw[thick] (b3) to [bend right=50] (b4);
\node at (3.5,-4.7) {\tiny $\cdots$};
}\quad .
\end{equation}
and $\Psi_{n,m}$ to
\begin{equation}\label{eq:Psinm}
\tik{0.5}{%
\draw[very thin] (-1,0) -- (8,0);
\node[bolito] (a1) at (0,0) [label=above: \raisebox{1pt}{\tiny $1$}] {};
\node[bolito] (a2) at (1,0) [label=above: \raisebox{1pt}{\tiny $2$}] {};
\node[bolito] (a3) at (3,0) [label=above: \raisebox{1pt}{\tiny $n$}] {};
\node[bolito] (a4) at (4,0) [label=above: {\tiny $n$+$1$}] {};
\node[bolito] (a5) at (7,0) [label=above: {\tiny $n$+$m$}] {};
\node at (2,0.5) {\tiny $\cdots$};
\node at (5.5,0.5) {\tiny $\cdots$};
\node at (2,-0.7) {\tiny $\cdots$};
\node at (5.5,-0.7) {\tiny $\cdots$};
\draw (-0.5,-2.7) rectangle (7.5,-1.3);
\node at (3.5,-2) {$h$};
\draw[thick] (a1) to (0,-1.3);
\draw[thick] (a2) to (1,-1.3);
\draw[thick] (a3) to (3,-1.3);
\draw[thick] (a4) to (4,-1.3);
\draw[thick] (a5) to (7,-1.3);
} 
\
\xrightarrow{\Psi_{n,m}}
\
\tik{0.5}{%
\draw[very thin] (-1,0) -- (5,0);
\draw[very thin] (-1,-3) --(12,-3);
\draw[very thin] (3,-5.5) --(12,-5.5);
\node[bolito] (a1) at (0,0) [label=above: {\tiny $1$}] {};
\node[bolito] (a2) at (1,0) [label=above: {\tiny $2$}] {};
\node[bolito] (a3) at (3,0) [label=above: {\tiny $n$}] {};
\node at (2,0.5) {\tiny $\cdots$};
\node at (2,-1.5) {\tiny $\cdots$};
\node[bolito] (b1) at (0,-3)  {};
\node[bolito] (b2) at (1,-3)  {};
\node[bolito] (b3) at (3,-3)  {};
\node[bolito] (b4) at (4,-3)  {};
\node[bolito] (b5) at (7,-3)  {};
\node[bolito] (b6) at (8,-3)  {};
\node[bolito] (b7) at (11,-3)  {};
\draw (-0.5,-5) rectangle (7.5,-3.5);
\node at (3.5,-4.3) {$h$};
\draw[thick] (a1) to (0,-3.5);
\draw[thick] (a2) to (1,-3.5);
\draw[thick] (a3) to (3,-3.5);
%
\draw[thick] (b4) to (4,-3.5);
\draw[thick] (b5) to (7,-3.5);
\draw[thick] (b4) .. controls (4,-1.5) and (11,-1.5) .. (b7);
\draw[thick] (b5) to [bend left=30pt] (b6);
\node[bolito] (c1) at (8,-5.5) [label=below: {\tiny $1$}] {};
\node[bolito] (c2) at (11,-5.5) [label=below: {{\vrule width 0pt height 4pt}\tiny $m$}] {};
\draw[thick] (b6) to (c1);
\draw[thick] (b7) to (c2);
\node at (5.5,-2.7) {\tiny $\cdots$};
\node at (9.5,-2.7) {\tiny $\cdots$};
\node at (9.5,-4) {\tiny $\cdots$};
}
\end{equation}  

Lastly, the result in item (iii) is represented as follows:
\begin{equation}\label{eq:transpose}
\tik{0.5}{%
\draw[very thin] (-1,0) -- (4,0);
\draw[very thin] (-1,-4) --(4,-4);
\node[bolito] (a1) at (0,-4) [label=below: {\tiny $1$}] {};
\node[bolito] (a2) at (1,-4) [label=below: {\tiny $2$}] {};
\node[bolito] (a3) at (3,-4) [label=below: {{\vrule width 0pt height 4pt}\tiny $n$}] {};
\node at (2,-4.5) {\tiny $\cdots$};
\node at (2,-3.2) {\tiny $\cdots$};
\node at (1.5,-0.8) {\tiny $\cdots$};
\node at (1.5,0.5) {\tiny $\cdots$};
\node[bolito] (b1) at (0.3,0) [label=above: {\tiny $1$}] {};
\node[bolito] (b2) at (2.7,0) [label=above: {\tiny $m$}] {};
\draw (-0.5,-2.7) rectangle (3.5,-1.3);
\node at (1.5,-2) {$f^t$};
\draw[thick] (a1) to (0,-2.7);
\draw[thick] (a2) to (1,-2.7);
\draw[thick] (a3) to (3,-2.7);
%
\draw[thick] (b1) to (0.3,-1.3);
\draw[thick] (b2) to (2.7,-1.3);
}
\qquad
=
\qquad
\tik{0.5}{%
\draw[very thin] (-5,3) --  (8,3);
\draw[very thin] (-5,0) -- (8,0);
\draw[very thin] (-5,-4) --(8,-4);
\draw[very thin] (-5,-7) -- (8,-7);
\node[bolito] (c1) at (4,3) [label=above: {\tiny $1$}] {};
\node[bolito] (c2) at (7,3) [label=above: {\tiny $m$}] {};
\draw[thick] (c1) to (4,-4);
\draw[thick] (c2) to (7,-4);
\node[bolito] (a1) at (0,0)  {};
\node[bolito] (a2) at (1,0)  {};
\node[bolito] (a3) at (3,0)  {};
\node at (2,0.5) {\tiny $\cdots$};
\node at (2,-0.8) {\tiny $\cdots$};
\node at (1.5,-3.2) {\tiny $\cdots$};
\node at (1.5,-4.6) {\tiny $\cdots$};
\node[bolito] (b1) at (0.3,-4)  {};
\node[bolito] (b2) at (2.7,-4)  {};
\draw (-0.5,-2.7) rectangle (3.5,-1.3);
\node at (1.5,-2) {$f$};
\draw[thick] (a1) to (0,-1.3);
\draw[thick] (a2) to (1,-1.3);
\draw[thick] (a3) to (3,-1.3);
%
\draw[thick] (b1) to (0.3,-2.7);
\draw[thick] (b2) to (2.7,-2.7);
\node[bolito] (d1) at (-4,-7) [label=below: {\tiny $1$}] {};
\node[bolito] (d2) at (-2,-7) [label=below: {\tiny $n$--$1$}] {};
\node[bolito] (d3) at (-1,-7) [label=below: {{\vrule width 0pt height 4pt}\tiny $n$}] {};
\draw[thick] (d1) to (-4,0);
\draw[thick] (d2) to (-2,0);
\draw[thick] (d3) to (-1,0);
\draw[thick] (-4,0) .. controls (-4,2) and (3,2) .. (a3);
\draw[thick] (-2,0) .. controls (-2,1) and (1,1) .. (a2);
\draw[thick] (-1,0) to [bend left=40pt] (a1);
\draw[thick] (b1) .. controls (0.3,-5.5) and (7,-5.5) .. (7,-4);
\draw[thick] (b2) to [bend right=40pt] (4,-4);
\node at (-3,-2) {\tiny $\cdots$};
\node at (5.5,-2) {\tiny $\cdots$};
\node at (-3,-7.5) {\tiny $\cdots$};
\node at (5.5,3.5) {\tiny $\cdots$};
}
\end{equation}

\begin{remark}\label{re:bbt} 
  From item (i) in Proposition \ref{pr:tPhiPsi} we get:
\[
(\mathsf{b}\circ \mathsf{b}^t)(1)=\sum_{i=1}^d \mathsf{b}( v_i\otimes u_i)=
\begin{cases} \dimm_\FF \VV&\text{if $\mathsf{b}$ is symmetric,}\\
   -\dimm_\FF \VV&\text{if $\mathsf{b}$ is skew symmetric.}
   \end{cases}
\]
\end{remark}  

%
%

\bigskip   

\section{$3$-tangles}\label{se:3tangles}

In this section, we establish the facts needed about
$3$-tangles and certain $3$-tangle categories,
following for the most part the definitions and conventions of \cite{CKR},
 but with some slight modifications.

\subsection{3-tangles, composition and disjoint union} \ \  Let $G=(V,E)$ be a finite, undirected graph, where $V$ is the set of nodes (vertices) and $E$ is the set of edges, which is a subset of the set of unordered pairs $uv$ of different nodes, $u,v \in V$, $u \ne v$. 

The \emph{valency} of a node $v$ is the number $\val(v)$ of edges incident with $v$. The cyclic group $C_3$ acts on the Cartesian power $E^3$ by cyclic permutations. Denote by $E^3/C_3$ the set of orbits, and by $[e_1,e_2,e_3]$ the orbit of $(e_1,e_2,e_3)$.

We will consider only graphs $G=(V,E)$ with $1\leq \val(v)\leq 3$ for all $v\in V$, and which are endowed with a map (\emph{orientation})
\[
\nu: \{v\in V : \val(v)=3\}\rightarrow E^3/C_3,
\]
such that, for any $v\in V$ with $\val(v)=3$, $\nu(v)=[e_1e_2e_3]$, where $e_1,e_2,e_3$ are the edges incident with $\nu$. That is, $\nu$ chooses one of the two possible orientations of the edges incident with $v$. Orientations are indicated by a coloring,
black for positive (clockwise) and white for negative (counterclockwise):
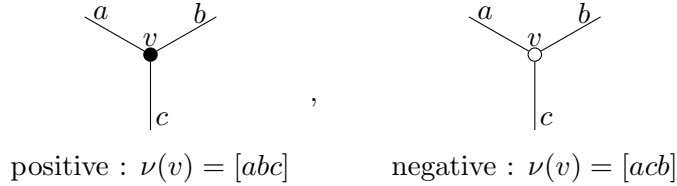
\[
\tik{0.5}{%
\node[shape=circle,draw,fill=black,minimum size=5pt,inner sep=0pt] (origin) at (0,0) {};
\path (150:2cm) coordinate (a);
\path (30:2cm) coordinate (b);
\path (270:2cm) coordinate (c);
\draw (origin) -- (a) (origin) -- (b) (origin) -- (c);
\node at (-1.3,1.1) {$a$};
\node at (1.3,1.1) {$b$};
\node at (.3,-1.7) {$c$};
\node at (0,0.4) {$v$};
\node at (0,-3) {positive : $\nu(v)=[abc]$};
}\ ,\qquad
\tik{0.5}{%
\node[shape=circle,draw,minimum size=5pt,inner sep=0pt] (origin) at (0,0) {};
\path (150:2cm) coordinate (a);
\path (30:2cm) coordinate (b);
\path (270:2cm) coordinate (c);
\draw (origin) -- (a) (origin) -- (b) (origin) -- (c);
\node at (-1.3,1.1) {$a$};
\node at (1.3,1.1) {$b$};
\node at (.3,-1.7) {$c$};
\node at (0,0.4) {$v$};
\node at (0,-3) {negative : $\nu(v)=[acb]$};
}\ .
\]   
The set of nodes of valency $1$ will be called the \emph{boundary} of $G$ and denoted by $\partial G$. A \emph{path} in $G=(V,E)$ is a subgraph $P$ with nodes $\{v_0,v_1,\ldots,v_k\}$ and edges $\{v_0v_1,v_1v_2,\ldots,v_{k-1}v_k\}$, where the $v_i$'s  are pairwise distinct. We write  $P=v_0v_1\cdots v_k$ and call $k$ the \emph{length} of $P$. A \emph{refinement} of a graph $G=(V,E)$ is a graph obtained from $G$ by dividing some of the edges, i.e., by replacing these edges with paths having the same endpoints and so that intermediate nodes are not nodes of $G$, and different edges involve different new nodes:
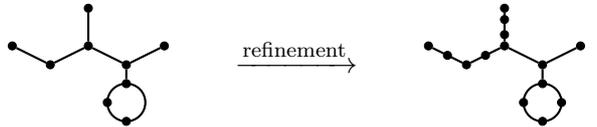
\[
\tik{0.5}{%
\node[bolito] (a1) at (0,0) {};
\node[bolito] (a2) at (0,-1) {};
\node[bolito] (a3) at (-2,-1) {};
\node[bolito] (a4) at (-1,-1.5) {};
\node[bolito] (a5) at (2,-1) {};
\node[bolito] (a6) at (1,-1.5) {};
\node[bolito] (a7) at (1,-2) {};
\node[bolito] (a8) at (0.5,-2.5) {};
\node[bolito] (a9) at (1,-3) {};
\draw[thick] (a3) -- (a4) -- (a2) -- (a6) -- (a5);
\draw[thick] (a1) to (a2);
\draw[thick] (a6) to (a7);
\draw[thick] (1,-2.5) circle (0.5cm);
}
\qquad
\xrightarrow{\text{refinement}}
\qquad
\tik{0.5}{%
\node[bolito] (a1) at (0,0) {};
\node[bolito] (a2) at (0,-1) {};
\node[bolito] (a3) at (-2,-1) {};
\node[bolito] (a4) at (-1,-1.5) {};
\node[bolito] (a5) at (2,-1) {};
\node[bolito] (a6) at (1,-1.5) {};
\node[bolito] (a7) at (1,-2) {};
\node[bolito] (a8) at (0.5,-2.5) {};
\node[bolito] (a9) at (1,-3) {};
\node[bolito] (b1) at (-1.5,-1.25) {};
\node[bolito] (b2) at (-0.5,-1.25) {};
\node[bolito] (b3) at (1.5,-2.5) {};
\node[bolito] (b4) at (0,-0.3) {};
\node[bolito] (b5) at (0,-0.7) {};
\draw[thick] (a3) -- (a4) -- (a2) -- (a6) -- (a5);
\draw[thick] (a1) to (a2);
\draw[thick] (a6) to (a7);
\draw[thick] (1,-2.5) circle (0.5cm);
}\ .
\] 
Two such graphs $G=(V,E)$ and $G'=(V',E')$ are said to be equivalent if $\partial G=\partial G'$,  and they admit identical refinements. Each equivalence class has a representative having no nodes of valency $2$ except for the nodes on circles.

For any $n\in\NN$, let $[n]=\{1,\ldots,n\}$ with $[0]=\emptyset$. Then, for $n,m\in\NN$, a \emph{$3$-tangle} $\gamma: [n]\rightarrow [m]$ is an equivalence class of graphs $G=(V,E)$ with $\partial G=[n]\sqcup [m]$ (disjoint union, which may thought of as $\{1,\ldots,n,1',\ldots,m'\}$). This is represented graphically by taking a representative with no nodes of valency $2$ (except for the circles), drawn between two lines,  as in the next example $\gamma:[7]\rightarrow [4]$
\begin{equation}\label{eq:gamma_example}
\gamma\quad =\qquad
\tik{0.5}{%
\draw[very thin] (0,0) -- (8,0);
\node[bolito] (a1) at (1,0) [label=above: {\tiny $1$}] {};
\node[bolito] (a2) at (2,0) [label=above: {\tiny $2$}] {};
\node[bolito] (a3) at (3,0) [label=above: {\tiny $3$}] {};
\node[bolito] (a4) at (4,0) [label=above: {\tiny $4$}] {};
\node[bolito] (a5) at (5,0) [label=above: {\tiny $5$}] {};
\node[bolito] (a6) at (6,0) [label=above: {\tiny $6$}] {};
\node[bolito] (a7) at (7,0) [label=above: {\tiny $7$}] {};
%
\node[bolito] (b1) at (1.5,-4) [label=below: {\tiny $1'$}] {};
\node[bolito] (b2) at (3,-4) [label=below: {\tiny $2'$}] {};
\node[bolito] (b3) at (4.5,-4) [label=below: {\tiny $3'$}] {};
\node[bolito] (b4) at (6,-4) [label=below: {\tiny $4'$}] {};
\node[bolito] (c1) at (1.5,-1.5)  {}; 
\node[bolito] (c2) at (2,-2.5) {};
\node[bolito] (c3) at (3.2,-1.5) {};
\node (x) at (1.2,-1.6) {\tiny $x$};
\node (y) at (1.7,-2.5) {\tiny $y$};
\draw[thick] (4.5,-2) circle (0.5cm);
\node[bolito] (c4) at (4,-2) {};
\node[bolito] (c5) at (5,-2) {};
\node[bolito] (c6) at (4.5,-1.5) {};
\draw[thick] (a3) to (c3);
\draw[thick] (a4) to (c3);
\draw[thick] (c3) to (b2);
\draw[thick] (a5) to [bend right=60] (a6);
\draw[thick] (a7) to (b4);
\draw[thick] (c1) to (c2);
\draw[thick] (c2) to (b1);
\draw[cross] (c2) to (b3);
\draw[thick] (a2) .. controls (1,-1) and (1,-1) .. (c1);
\draw[cross] (a1) .. controls (2,-1) and (2,-1) .. (c1);
\draw[very thin] (0,-4) -- (8,-4);
}
\end{equation}  
where the orientation of a valency $3$ node (like the node labeled by $x$ in \eqref{eq:gamma_example}) is given by clockwise order: $\nu(x)=[2x\, 1x\, yx]$, that is, we have:
\[
\tik{0.5}{%
\draw[very thin] (0,0) -- (3,0);
\node[bolito] (a1) at (1,0)  {};
\node[bolito] (a2) at (2,0)  {};
\draw[very thin] (0,-2) -- (3,-2);
\node[bolito] (b1) at (1.5,-2)  {};
\node[bolito] (c1) at (1.5,-1.2)  {};
\draw[thick] (c1) to (b1);
\draw[thick] (a2) .. controls (1,-0.6) and (1,-0.6) .. (c1);
\draw[cross] (a1) .. controls (2,-0.6) and (2,-0.6) .. (c1);
}
\quad = \quad
\tik{0.5}{%
\node[bolito] (a1) at (1,0)  {};
\node[bolito] (a2) at (2,0)  {};
\draw[very thin] (0,-2) -- (3,-2);
\node[bolito] (b1) at (1.5,-2)  {};
\node[shape=circle,draw,minimum size=3pt,inner sep=0pt ] (c1) at (1.5,-1)  {};
\draw[thick] (c1) to (b1);
\draw[thick] (a2) to (c1);
\draw[cross] (a1) to (c1);
\draw[very thin] (0,0) -- (3,0);
}
\]
and hence only positively oriented nodes will be used. Moreover,  
we will usually suppress the $\bullet$'s in our figures:
\begin{equation}\label{eq:gamma_example_bis}
\gamma\quad =\qquad
\tik{0.5}{%
\node[punto] (a1) at (1,0) [label=above: {\tiny $1$}] {};
\node[punto] (a2) at (2,0) [label=above: {\tiny $2$}] {};
\node[punto] (a3) at (3,0) [label=above: {\tiny $3$}] {};
\node[punto] (a4) at (4,0) [label=above: {\tiny $4$}] {};
\node[punto] (a5) at (5,0) [label=above: {\tiny $5$}] {};
\node[punto] (a6) at (6,0) [label=above: {\tiny $6$}] {};
\node[punto] (a7) at (7,0) [label=above: {\tiny $7$}] {};
%
\node[punto] (b1) at (1.5,-4) [label=below: {\tiny $1'$}] {};
\node[punto] (b2) at (3,-4) [label=below: {\tiny $2'$}] {};
\node[punto] (b3) at (4.5,-4) [label=below: {\tiny $3'$}] {};
\node[punto] (b4) at (6,-4) [label=below: {\tiny $4'$}] {};
\node[punto] (c1) at (1.5,-1.5)  {}; 
\node[punto] (c2) at (2,-2.5) {};
\node[punto] (c3) at (3.2,-1.5) {};
\node (x) at (1.2,-1.6) {\tiny $x$};
\node (y) at (1.7,-2.5) {\tiny $y$};
\draw[thick] (4.5,-2) circle (0.5cm);
\draw[thick] (a3) to (c3);
\draw[thick] (a4) to (c3);
\draw[thick] (c3) to (b2);
\draw[thick] (a5) to [bend right=60] (a6);
\draw[thick] (a7) to (b4);
\draw[thick] (c1) to (c2);
\draw[thick] (c2) to (b1);
\draw[cross] (c2) to (b3);
\draw[thick] (a2) .. controls (1,-1) and (1,-1) .. (c1);
\draw[cross] (a1) .. controls (2,-1) and (2,-1) .. (c1);
\draw[very thin] (0,0) -- (8,0);
\draw[very thin] (0,-4) -- (8,-4);
}
\end{equation}  
The \emph{disjoint union} of two $3$-tangles $\gamma:[n]\rightarrow [m]$ and $\gamma^*:[n^*]\rightarrow [m^*]$ is the $3$-tangle $\gamma \sqcup\gamma^*:[n+n^*]\rightarrow [m+m^*]$ obtained by  juxtaposing $\gamma^*$ to the right of $\gamma$
and shifting the numbering on $\gamma$ as illustrated below. The composition of two $3$-tangles $\gamma:[n]\rightarrow [m]$ and $\gamma^*:[m]\rightarrow [r]$ is the tangle $\gamma^* \circ\gamma:[n]\rightarrow [r]$ obtained by bordism: 
\[
\begin{split}
\gamma&=\ 
\tik{0.5}{%
%
\node[punto] (a1) at (1,0) [label=above: {\tiny $1$}] {};
\node[punto] (a2) at (2,0) [label=above: {\tiny $2$}] {};
\node[punto] (a3) at (3,0) [label=above: {\tiny $3$}] {};
\node[punto] (a4) at (4,0) [label=above: {\tiny $4$}] {};
\draw[very thin] (0,-2) -- (5,-2);
\node[punto] (b1) at (1,-2) [label=below left: {\tiny $1'$}] {};
\node[punto] (b2) at (2,-2) [label=below right: {\tiny $2'$}] {};
\node[punto] (c1) at (1.5,-1) {};
\node[punto] (c2) at (2.5,-1.2) {};
\draw[thick] (a3) to [bend right=60] (a4);
\draw[thick] (2.5,-0.8) circle (0.4cm);
\draw[thick] (3.5,-1) circle (0.4cm);
\draw[thick] (c1) to (b1);
\draw[thick] (c2) to (b2);
\draw[thick] (a2) .. controls (1,-0.5) and (1,-0.5) .. (c1);
\draw[cross] (a1) .. controls (2,-0.5) and (2,-0.5) .. (c1);
\draw[very thin] (0,0) -- (5,0);
}
\\[-15pt]
\gamma^*&=\ 
\tik{0.5}{%
\draw[very thin] (0,-4) --(5,-4);
\node[punto] (aa1) at (1,-4) [label=above left: {\tiny $1$}] {};
\node[punto] (aa2) at (2,-4) [label=above right: {\tiny $2$}] {};
\draw[thin,dotted] (b1) to (aa1);
\draw[thin,dotted] (b2) to (aa2);
\draw[very thin] (0,-6) to (5,-6);
\node[punto] (bb1) at (1,-6) [label=below: {\tiny $1'$}] {};
\node[punto] (bb2) at (2,-6) [label=below: {\tiny $2'$}] {};
\node[punto] (bb3) at (3,-6) [label=below: {\tiny $3'$}] {};
\node[punto] (cc1) at (1.5,-5) {};
\draw[thick] (aa1) to (cc1);
\draw[thick] (aa2) to (cc1);
\draw[thick] (cc1) to (bb1);
\draw[thick] (bb2) to [bend left=60] (bb3);
}\\[5pt]
\gamma \sqcup\gamma^*&=\ 
\tik{0.5}{%
%
\node[punto] (a1) at (1,0) [label=above: {\tiny $1$}] {};
\node[punto] (a2) at (2,0) [label=above: {\tiny $2$}] {};
\node[punto] (a3) at (3,0) [label=above: {\tiny $3$}] {};
\node[punto] (a4) at (4,0) [label=above: {\tiny $4$}] {};
\node[punto] (aa1) at (5,0) [label=above: {\tiny $5$}] {};
\node[punto] (aa2) at (6,0) [label=above: {\tiny $6$}] {};
\draw[very thin] (0,-2) -- (7,-2);
\node[punto] (b1) at (1,-2) [label=below: {\tiny $1'$}] {};
\node[punto] (b2) at (2,-2) [label=below: {\tiny $2'$}] {};
\node[punto] (bb1) at (3,-2) [label=below: {\tiny $3'$}] {};
\node[punto] (bb2) at (4,-2) [label=below: {\tiny $4'$}] {};
\node[punto] (bb3) at (5,-2) [label=below: {\tiny $5'$}] {};
\node[punto] (c1) at (1.5,-1) {};
\node[punto] (c2) at (2.5,-1.2) {};
\node[punto] (cc1) at (5.5,-1) {};
\draw[thick] (a3) to [bend right=60] (a4);
\draw[thick] (2.5,-0.8) circle (0.4cm);
\draw[thick] (3.5,-1) circle (0.4cm);
\draw[thick] (c1) to (b1);
\draw[thick] (c2) to (b2);
\draw[thick] (a2) .. controls (1,-0.5) and (1,-0.5) .. (c1);
\draw[cross] (a1) .. controls (2,-0.5) and (2,-0.5) .. (c1);
\draw[very thin] (0,0) -- (7,0);
\draw[thick] (aa1) to (cc1);
\draw[thick] (aa2) to (cc1);
\draw[thick] (cc1) to (bb1);
\draw[thick] (bb2) to [bend left=60] (bb3);
}
\\[5pt]
\gamma^*\circ\gamma&=\ 
\tik{0.5}{%
\draw[very thin] (0,-9) to (5,-9);
\node[punto] (aaa1) at (1,-9) [label=above: {\tiny $1$}] {};
\node[punto] (aaa2) at (2,-9) [label=above: {\tiny $2$}] {};
\node[punto] (aaa3) at (3,-9) [label=above: {\tiny $3$}] {};
\node[punto] (aaa4) at (4,-9) [label=above: {\tiny $4$}] {};
\draw[very thin] (0,-12) to (5,-12);
\node[punto] (bbb1) at (1,-12) [label=below: {\tiny $1'$}] {};
\node[punto] (bbb2) at (2,-12) [label=below: {\tiny $2'$}] {};
\node[punto] (bbb3) at (3,-12) [label=below: {\tiny $3'$}] {};
\node[punto] (ccc1) at (1,-10) {};
\node[punto] (ccc2) at (3,-10.5) {};
\node[punto] (ccc3) at (2,-11) {};
\draw[thick] (aaa3) to [bend right=60] (aaa4);
\draw[thick] (bbb2) to [bend left=60] (bbb3);
\draw[thick] (4,-10.5) circle (0.4cm);
\draw[thick] (3,-10.1) circle (0.4cm);
\draw[thick] (ccc1) to (ccc3) ;
\draw[thick] (ccc2) to (ccc3);
\draw[thick] (ccc3) to (bbb1);
\draw[thick] (aaa2) .. controls (1,-9.5) and (0.5,-9.5) .. (ccc1);
\draw[cross] (aaa1) .. controls (1.2,-9.5) and (1.2,-9.8) .. (ccc1);
}
\end{split}
\]

The \emph{transpose} of a $3$-tangle is obtained rotating it $180^{\circ}$. In \cite{CKR},  the transpose is gotten  by using a reflection, however, this changes the orientation of the nodes of valency $3$, and we want to avoid that.  For $\gamma$ in \eqref{eq:gamma_example_bis}, its transpose is
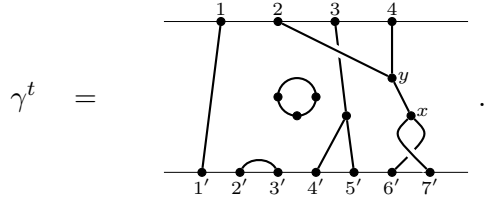
\[
\gamma^t\quad =\qquad
\tik{0.5}{%
\draw[very thin] (0,-4) -- (8,-4);
\node[bolito] (a1) at (7,-4) [label=below: {\tiny $7'$}] {};
\node[bolito] (a2) at (6,-4) [label=below: {\tiny $6'$}] {};
\node[bolito] (a3) at (5,-4) [label=below: {\tiny $5'$}] {};
\node[bolito] (a4) at (4,-4) [label=below: {\tiny $4'$}] {};
\node[bolito] (a5) at (3,-4) [label=below: {\tiny $3'$}] {};
\node[bolito] (a6) at (2,-4) [label=below: {\tiny $2'$}] {};
\node[bolito] (a7) at (1,-4) [label=below: {\tiny $1'$}] {};
%
\node[bolito] (b1) at (6,0)  [label=above: {\tiny $4$}] {};
\node[bolito] (b2) at (4.5,0) [label=above: {\tiny $3$}] {};
\node[bolito] (b3) at (3,0) [label=above: {\tiny $2$}] {};
\node[bolito] (b4) at (1.5,0) [label=above: {\tiny $1$}] {};
\node[bolito] (c1) at (6.5,-2.5)  {}; 
\node[bolito] (c2) at (6,-1.5) {};
\node[bolito] (c3) at (4.8,-2.5) {};
\node (x) at (6.8,-2.4) {\tiny $x$};
\node (y) at (6.3,-1.5) {\tiny $y$};
\draw[thick] (3.5,-2) circle (0.5cm);
\node[bolito] (c4) at (4,-2) {};
\node[bolito] (c5) at (3,-2) {};
\node[bolito] (c6) at (3.5,-2.5) {};
\draw[thick] (a3) to (c3);
\draw[thick] (a4) to (c3);
\draw[thick] (c3) to (b2);
\draw[thick] (a5) to [bend right=60] (a6);
\draw[thick] (a7) to (b4);
\draw[thick] (c1) to (c2);
\draw[thick] (c2) to (b1);
\draw[cross] (c2) to (b3);
\draw[thick] (a2) .. controls (7,-3) and (7,-3) .. (c1);
\draw[cross] (a1) .. controls (6,-3) and (6,-3) .. (c1);
\draw[very thin] (0,0) -- (8,0);
}\ .
\]
Alternatively, the transpose of $\gamma$ can be obtained as in \eqref{eq:transpose} by disjoint union and composition:
\begin{equation}\label{eq:gamma_transpose}
\gamma^t=(\II_n\sqcup\beta_m)\circ(\II_n\sqcup\gamma\sqcup \II_m)\circ(\beta_n^t\sqcup \II_m),
\end{equation}
where
\[
\II_n=
\tik{0.5}{%
\draw[very thin] (0,0) --(5,0);
\node[punto] (a1) at (1,0) [label=above :{\tiny $1$}] {};
\node[punto] (a2) at (2,0) [label=above :{\tiny $2$}] {};
\node[punto] (a3) at (4,0) [label=above :{\tiny $n$}] {};
\node at (3,0.3) {$\cdots$};
\draw[very thin] (0,-2) -- (5,-2);
\node[punto] (b1) at (1,-2) [label=below:{\tiny $1'$}] {};
\node[punto] (b2) at (2,-2) [label=below:{\tiny $2'$}] {};
\node[punto] (b3) at (4,-2) [label=below:{\tiny $n'$}] {};
\node at (3,-2.2) {$\cdots$};
\draw[thick] (a1) to (b1);
\draw[thick] (a2) to (b2);
\draw[thick] (a3) to (b3);
\node at (3,-1) {$\cdots$};
}
,\quad 
\beta_n=
\tik{0.5}{%
\draw[very thin] (0,0) -- (6,0);
\node[punto] (a1) at (1,0) [label=above: {\tiny $1$}] {};
\node[punto] (a2) at (2,0) [label=above: {\tiny $2$}] {};
\node[punto] (a3) at (4,0) [label=above: {\tiny $2n$--$1$}] {};
\node[punto] (a4) at (5,0) [label=above: {\tiny $2n$}] {};
\node at (2.9,0.3) {\tiny$\cdots$};
\node at (3,-0.3) {$\cdots$};
\draw[thick] (a1) to [bend right=80] (a4);
\draw[thick] (a2) to [bend right=80] (a3);
}
,\quad
\beta_n^t=
\tik{0.5}{%
\draw[very thin] (0,-2) -- (6,-2);
\node[punto] (a1) at (1,-2) [label=below: {\tiny $1'$}] {};
\node[punto] (a2) at (2,-2) [label=below: {\tiny $2'$}] {};
\node[punto] (a3) at (4,-2) [label=below: {\tiny $(2n$--$1)'$}] {};
\node[punto] (a4) at (5,-2) [label=below: {\tiny\qquad $(2n)'$}] {};
\node at (2.8,-2.3) {\tiny$\cdots$};
\node at (3,-1.8) {$\cdots$};
\draw[thick] (a1) to [bend left=80] (a4);
\draw[thick] (a2) to [bend left=80] (a3);
}.
\]
 \subsection{The 3-tangle categories $\cT$ and $\cT_\Gamma$} \ \ 
Assume $\FF$ is a field.  The \emph{$3$-tangle category}  $\cT$ has as objects the finite sets $[n]$, $n\in\NN$, and as morphisms
$\Mor_{\cT}([n],[m])$, the 
$\FF$-linear combinations of $3$-tangles. The disjoint union $\sqcup$ induces a tensor product $\sqcup:\cT \times\cT \rightarrow \cT$. The tensor product is associative and has the unit object given by the empty $3$-tangle $[0]\rightarrow [0]$.

The transpose induces a map $\Mor_{\cT}([n],[m])\rightarrow\Mor_{\cT}([m],[n])$, $\gamma\mapsto \gamma^t$,  such that $(\gamma^*\circ\gamma)^t=\gamma^t\circ(\gamma^*)^t$ whenever this makes sense.

Also, as in Proposition \ref{pr:tPhiPsi} (see equations \eqref{eq:Phinm} and \eqref{eq:Psinm}), there is a  bijection
\begin{equation}\label{eq:Phinm_bis}
\begin{split}
\Phi_{n,m}:\Mor_{\cT}([n],[m])&\longrightarrow \Mor_{\cT}([n+m],[0])\\
\gamma\quad &\mapsto\quad \beta_m\circ(\gamma\sqcup  \II_n),
\end{split}
\end{equation}
with inverse
\begin{equation}\label{eq:Psinm_bis}
\begin{split}
\Psi_{n,m}:\Mor_{\cT}([n+m],[0])&\longrightarrow \Mor_{\cT}([n],[m])\\
\gamma'\quad &\mapsto (\gamma'\sqcup  \II_m)\circ(\II_n\sqcup  \beta_m^t).
\end{split}
\end{equation}

The morphisms 
\begin{align}\label{eq:basic1}
\II_1&=
\tik{0.5}{%
\draw[very thin] (0,0) -- (2,0);
\draw[very thin] (0,-2) -- (2,-2);
\draw[thick] (1,0) -- (1,-2);
}
&
\beta=\beta_1&=
\tik{0.5}{%
\draw[very thin] (0,0) -- (3,0);
\draw[thick] (1,0) to [bend right=80] (2,0);
\node at (1,-1) {};
}
&
\mu&=
\tik{0.5}{%
\draw[very thin] (0,0) -- (3,0);
\draw[very thin] (0,-2) -- (3,-2);
\draw[thick] (1,0) to (1.5,-1);
\draw[thick] (2,0) to (1.5,-1);
\draw[thick] (1.5,-1) to (1.5,-2);
}
\\[10pt]   \label{eq:basic2} 
\tau&=
\tik{0.5}{%
\draw[thick] (2,0) to (1,-2);
\draw[cross] (1,0) to (2,-2);
\draw[very thin] (0,0) -- (3,0);
\draw[very thin] (0,-2) -- (3,-2);
}
&
\beta^t&=
\tik{0.5}{%
\draw[very thin] (0,-2) -- (3,-2);
\draw[thick] (1,-2) to [bend left=80] (2,-2);
\node at (2,-1.5) {};
}
&
\mu^t&=
\tik{0.5}{%
\draw[very thin] (0,0) -- (3,0);
\draw[very thin] (0,-2) -- (3,-2);
\draw[thick] (1,-2) to (1.5,-1);
\draw[thick] (2,-2) to (1.5,-1);
\draw[thick] (1.5,-1) to (1.5,0);
}
\end{align} 
are called \emph{basic}. They constitute the \emph{alphabet} of $\cT$. Note that $\mu^t$ can be obtained from the other basic morphisms by composition and tensor product  as in \eqref{eq:gamma_transpose}. Also, composing $\mu\circ\tau$,  we get
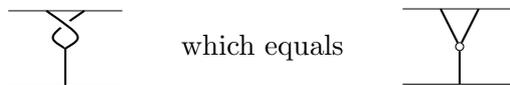
\[
\tik{0.5}{%
\draw[thick] (2,0) .. controls (1,-0.7) and (1,-0.7) .. (1.5,-1);
\draw[cross] (1,0) .. controls (2,-0.7) and (2,-0.7) .. (1.5,-1);
\draw[thick] (1.5,-1) to (1.5,-2);
\draw[very thin] (0,0) -- (3,0);
\draw[very thin] (0,-2) -- (3,-2);
}
\qquad \text{which equals}\qquad
\tik{0.5}{%
\node[shape=circle,draw, inner sep=0pt, minimum size=3pt] (c1) at (1.5,-1) {};
\draw[thick] (2,0) to (c1);
\draw[thick] (1,0) to (c1);
\draw[thick] (c1) to (1.5,-2);
\draw[very thin] (0,0) -- (3,0);
\draw[very thin] (0,-2) -- (3,-2);
}
\]
The morphisms in $\cT$ are generated, through composition and disjoint union and taking 
linear combinations, from basic morphisms  \cite[Proposition 2.3]{CKR}.

In addition to generators,  some relations can be imposed in the category $\cT$. Let
\[
\Gamma=\{\gamma_i\in\Mor_{\cT}([n_i],[m_i]) : i=1,\ldots,k\}
\]
be a finite set of morphisms in $\cT$. For each $n,m\in\NN$, the set $\Gamma$ generates, through compositions and tensor products with arbitrary $3$-tangles, a subspace $\R_\Gamma([n],[m])$ of $\Mor_{\cT}([n],[m])$, and we define a new category $\cT_\Gamma$ with the same objects and with $\Mor_{\cT_\Gamma}([n],[m])=\Mor_{\cT}([n],[m])/\R_\Gamma([n],[m])$. 
$\cT_\Gamma$ is the \emph{$3$-tangle category associated with the set of relations $\Gamma$}.

\subsection{The functor $\cR_{\frV}$} \ \ 
Assume that $\frV=(\VV,\mathsf{b},\mm)$ is a finite-dimensional nonassociative (i.e. not necessarily associative)  algebra  with multiplication $\mm:\VV\times \VV\rightarrow \VV$ (i.e., $\mm$ is an $\FF$-bilinear map), and endowed with an associative, nondegenerate, symmetric bilinear form $\mathsf{b}:\VV\times \VV\rightarrow \FF$. Associativity of the form means $\mathsf{b}\bigl(\mm(x,y),z\bigr)=\mathsf{b}\bigl(x,\mm(y,z)\bigr)$ for any $x,y,z\in \VV$. Let $\cV$ be the category of finite-dimensional vector spaces over $\FF$ with linear maps as morphisms. Denote by $\tau$ the switch map $\tau:\VV^{\otimes 2}\rightarrow \VV^{\otimes 2}$, $x\otimes y\mapsto y\otimes x$. Recall that we may identify $\mathsf{b}$ with a linear map $\VV^{\otimes 2}\rightarrow \FF$ and $\mm$ with a linear map $\VV^{\otimes 2}\rightarrow \VV$.   Let $1_{\VV}$ be the identity map on $\VV$. 
\medskip

\begin{theorem}[{\cite[Theorem 3.1]{CKR}, \cite{Boo98}}]\label{th:RV}
There exists a unique functor $\cR_\frV:\cT \rightarrow \cV$ such that:
\begin{enumerate}
\item $\cR_{\frV}([0])=\FF$ and $\cR_{\frV}([n])=\VV^{\otimes n}$, for any $n\geq 1$.

\item $\cR_{\frV}(\II_1)=1_\VV$ and  $\cR_{\frV}(\tau)=\tau$.

\item $\cR_{\frV}(\beta)=\mathsf{b}$, $\cR_{\frV}(\mu)=\mm$, $\cR_{\frV}(\gamma^t)=\cR_{\frV}(\gamma)^t$, and \\
$\cR_{\frV}(\gamma\sqcup \delta)=\cR_{\frV}(\gamma)\otimes\cR_{\frV}(\delta)$, for morphisms $\gamma$ and $\delta$ in $\cT$.
\end{enumerate}
\end{theorem}

The symmetry and associativity of the bilinear form $\mathsf{b}$ play  a key role here.   

The empty $3$-tangle $[0]\rightarrow [0]$ corresponds under $\cR_{\frV}$ to the identity map in $\Hom_\FF(\FF,\FF)\cong\FF$. We will denote it simply by $1$.

\begin{remark}\label{re:bbt2}
By Remark \ref{re:bbt},
\[
\cR_{\frV}(\beta\circ\beta^t)=\dimm_\FF \VV\in\FF\cong\Hom_\FF(\FF,\FF).
\]
\end{remark}

For any triple $\frV=(\VV,\mathsf{b},\mm)$ as above, we have the \emph{alphabet} of basic tensors $1_\VV=\id:\VV\rightarrow \VV$, $\mathsf{b}: \VV^{\otimes 2}\rightarrow \FF$, $\mm:\VV^{\otimes 2}\rightarrow \VV$, $\tau: \VV^{\otimes 2}\rightarrow \VV^{\otimes 2}$, $\mathsf{b}^t:\FF\rightarrow \VV^{\otimes 2}$, and $\mm^t:\VV\rightarrow \VV^{\otimes 2}$. Let   
\[
\Gamma_{\frV}=\{\mathsf{c}_i\in \Hom_\FF(\VV^{\otimes n_i},\VV^{\otimes m_i}) : i=1,\ldots,k\}
\]
be a finite set of homomorphisms. The algebra $\frV$ is said to be \emph{of tensor type $\Gamma_{\frV}$} if the $\mathsf{c}_i$'s are identities for $\frV$, i.e., if
\[
\mathsf{c}_i(x_1\otimes\cdots\otimes x_{n_i})=0
\]
for all $i=1,\ldots,k$, and all $x_1,\ldots,x_{n_i}\in \VV$.

\begin{corollary}[{\cite[Corollary 3.5]{CKR}}]\label{co:RGamma}
Let $\frV=(\VV,\mathsf{b},\mm)$ be an algebra of tensor type $\Gamma_{\frV}=\{\mathsf{c}_i: i=1,\ldots,k\}$, for tensors $\mathsf{c}_i$ expressible in terms of  the alphabet, $1_\VV$, $\mm$, $\mm^t$, $\mathsf{b}$, $\mathsf{b}^t$, and $\tau$. Let $\Gamma=\{\gamma_i: i=1,\ldots,k\}$ with $\cR_{\frV}(\gamma_i)=\mathsf{c}_i$ for all $i=1,\ldots,k$. Then there is a unique functor $\cR_\Gamma:\cT_{\Gamma}\rightarrow \cV$ such that $\cR_{\cV}=\cR_{\Gamma}\circ\cP$, where $\cP$ is the natural projection $\cT \rightarrow \cT_{\Gamma}$.
\end{corollary}  
%
%
\section{Seven-dimensional cross products}\label{se:7dim}

\subsection{Cross products} \ \ 
Let $\FF$ be a field of characteristic  $\ne 2$,  and let $(\VV,\mathsf{b})$ be a vector space $\VV$ over $\FF$ equipped with a nondegenerate symmetric bilinear form $\mathsf{b}$. A \emph{cross product} on $(\VV,\mathsf{b})$ is a bilinear multiplication $\VV\times \VV\rightarrow \VV$, $(x,y)\mapsto x\times y$, such that:
\begin{subequations}\label{eq:CP}
\begin{align}
& \mathsf{b}( x\times y,x)=0,\label{eq:CP1}\\
& x\times x=0,\label{eq:CP2}\\
& \mathsf{b}( x\times y,x\times y)=\mathsf{b}( x,x)\mathsf{b}( y,y)-\mathsf{b}( x,y)\mathsf{b}( y,x),\label{eq:CP3}
\end{align}
\end{subequations}  
for any $x,y\in \VV$. It is well known that a nonzero cross product exists only if $\dimm_\FF \VV=3$ or $7$. (See \cite{BrownGray}. For a proof using $3$-tangles,  see \cite{Rost04}.)

Since the characteristic is $\ne 2$, the equations in \eqref{eq:CP} are equivalent to:
\begin{subequations}\label{eq:CP7}
\begin{align}
&\mathsf{b}( x\times y,z)=\mathsf{b}( x,y\times z),\label{eq:CP7_1}\\
&x\times y=-y\times x,\label{eq:CP7_2}\\
&\mathsf{b}( x\times y,z\times t)+\mathsf{b}( y\times z,t\times x) \ = \label{eq:CP7_3}\\[-2pt]
&\qquad 2\mathsf{b}( x,z)\mathsf{b}( y,t)-\mathsf{b}( x,y)\mathsf{b}( z,t)-\mathsf{b}( y,z)\mathsf{b}( x,t),\nonumber
\end{align}
\end{subequations}
for any $x,y,z,t\in \VV$. Equation \eqref{eq:CP7_1} shows that $\mathsf{b}$ is associative relative to the cross product, so that $\frV=(\VV,\mathsf{b},\times)$ satisfies the conditions in Section \ref{se:3tangles}, and \eqref{eq:CP7_2} and \eqref{eq:CP7_3} can be expressed as
\[
\mathsf{c}_1(x_1\otimes x_2)=0=\mathsf{c}_2(x_1\otimes x_2\otimes x_3 \otimes x_4)
\]
for any $x_1,x_2,x_3, x_4\in \VV$, where
\begin{equation}\label{eq:c1_c2}
\begin{split}
&\mathsf{c}_1(x_1\otimes x_2)=x_1\times x_2+x_2\times x_1, \\
&\mathsf{c}_2(x_1\otimes x_2\otimes x_3\otimes x_4)= \mathsf{b}( x_1\times x_2,x_3\times x_4)+\mathsf{b}( x_2\times x_3,x_4\times x_1) \\[-2pt]
&\qquad -2\mathsf{b}( x_1,x_3)\mathsf{b}( x_2,x_4)+\mathsf{b}( x_1,x_2)\mathsf{b}( x_3,x_4) +\mathsf{b}( x_2,x_3)\mathsf{b}( x_1,x_4).
\end{split}
\end{equation}
 Thus,  $\mathsf{c}_1\in\Hom_\FF(\VV^{\otimes 2},\VV)$ and $\mathsf{c}_2\in\Hom_\FF(\VV^{\otimes 4},\FF)$, and they are the images under $\cR_\frV$ of $\gamma_1\in\Mor_{\cT}([2],[1])$ and $\gamma_2' \in\Mor_{\cT_2}([4],[0])$, with

\begin{equation}\label{eq:gamma1}
\gamma_1:\quad 
\tik{0.5}{%
\draw[very thin] (0,0) -- (3,0);
\path (1,0) coordinate (a1);
\path (2,0) coordinate (a2);
\draw[very thin] (0,-2) -- (3,-2);
\path (1.5,-2) coordinate (b1);
\path (1.5,-1) coordinate (c1);
\draw[thick] (a2) to (c1);
\draw[thick] (a1) to (c1);
\draw[thick] (c1) to (b1);
}
\quad + \quad
\tik{0.5}{%
\path (1,0) coordinate (a1);
\path (2,0) coordinate (a2);
\draw[very thin] (0,-2) -- (3,-2);
\path (1.5,-2) coordinate (b1);
\path (1.5,-1) coordinate (c1);
\draw[thick] (a2) .. controls (1,-0.6) and (1,-0.6) .. (c1);
\draw[cross] (a1) .. controls (2,-0.6) and (2,-0.6) .. (c1);
\draw[very thin] (0,0) -- (3,0);
\draw[thick] (c1) to (b1);
},
\end{equation}
($\gamma_1=\mu+\mu\circ\tau\in\Mor_{\cT_2}([2],[1])$), and
\[
\gamma_2':\quad
\tik{0.4}{%
\path (1,0) coordinate (a1);
\path (2,0) coordinate (a2);
\path (3,0) coordinate (a3);
\path (4,0) coordinate (a4);
\path (1.5,-0.5) coordinate (b1);
\path (3.5,-0.5) coordinate (b2);
\path (2.5,-1.5) coordinate (c);
\draw[thick] (a1) .. controls (1,-0.3) and (1.2,-0.5) .. (b1);
\draw[thick] (a2) .. controls (2,-0.3) and (1.8,-0.5) .. (b1);
\draw[thick] (a3) .. controls (3,-0.3) and (3.2,-0.5) .. (b2);
\draw[thick] (a4) .. controls (4,-0.3) and (3.8,-0.5) .. (b2);
\draw[thick] (b1) .. controls (1.5,-1) and (3.5,-1) .. (b2);
\draw[very thin] (0.5,0)-- (4.5,0);
}+
\tik{0.4}{%
\path (1,0) coordinate (a1);
\path (2,0) coordinate (a2);
\path (3,0) coordinate (a3);
\path (4,0) coordinate (a4);
\path (2.5,-0.5) coordinate (b1);
\path (2.5,-1) coordinate (b2);
\path (2.5,-1.5) coordinate (c);
\draw[thick] (a2) .. controls (2,-0.3) and (2.2,-0.5) .. (b1);
\draw[thick] (a3) .. controls (3,-0.3) and (2.8,-0.5) .. (b1);
\draw[thick] (a1) .. controls (1,-0.6) and (1.5,-1) .. (b2);
\draw[thick] (a4) .. controls (4,-0.6) and (3.5,-1) .. (b2);
\draw[thick] (b1) to (b2);
\draw[very thin] (0.5,0)-- (4.5,0);
}-2
\tik{0.4}{%
\path (1,0) coordinate (a1);
\path (2,0) coordinate (a2);
\path (3,0) coordinate (a3);
\path (4,0) coordinate (a4);
%
%
\path (2.5,-1.5) coordinate (c);
\draw[thick] (a2) .. controls (2,-0.8) and (4,-0.8) .. (a4);
\draw[cross] (a1) .. controls (1,-0.8) and (3,-0.8) .. (a3);
\draw[very thin] (0.5,0)-- (4.5,0);
}+
\tik{0.4}{%
\path (1,0) coordinate (a1);
\path (2,0) coordinate (a2);
\path (3,0) coordinate (a3);
\path (4,0) coordinate (a4);
\path (1.5,-0.5) coordinate (b1);
\path (3.5,-0.5) coordinate (b2);
\path (2.5,-1.5) coordinate (c);
\draw[thick] (a1) .. controls (1,-0.3) and (1.2,-0.5) .. (b1);
\draw[thick] (a2) .. controls (2,-0.3) and (1.8,-0.5) .. (b1);
\draw[thick] (a3) .. controls (3,-0.3) and (3.2,-0.5) .. (b2);
\draw[thick] (a4) .. controls (4,-0.3) and (3.8,-0.5) .. (b2);
\draw[very thin] (0.5,0)-- (4.5,0);
}+
\tik{0.4}{%
\path (1,0) coordinate (a1);
\path (2,0) coordinate (a2);
\path (3,0) coordinate (a3);
\path (4,0) coordinate (a4);
\path (2.5,-0.5) coordinate (b1);
\path (2.5,-1) coordinate (b2);
\path (2.5,-1.5) coordinate (c);
\draw[thick] (a2) .. controls (2,-0.3) and (2.2,-0.5) .. (b1);
\draw[thick] (a3) .. controls (3,-0.3) and (2.8,-0.5) .. (b1);
\draw[thick] (a1) .. controls (1,-0.6) and (1.5,-1) .. (b2);
\draw[thick] (a4) .. controls (4,-0.6) and (3.5,-1) .. (b2);
\draw[very thin] (0.5,0)-- (4.5,0);
}
\]

Using $\Psi_{2,2}$ we may substitute $\gamma_2'$ with

\begin{equation}\label{eq:gamma2}
\gamma_2:\quad
\tik{0.5}{%
\path (1,0) coordinate (a1);
\path (3,0) coordinate (a2);
\path (1,-2) coordinate (b1);
\path (3,-2) coordinate (b2);
\path (2,-0.5) coordinate (c1);
\path (2,-1.5) coordinate (c2);
\draw[thick] (a1) .. controls (1,-0.3) and (1.7,-0.5) .. (c1);
\draw[thick] (a2) .. controls (3,-0.3) and (2.3,-0.5) .. (c1);
\draw[thick] (b1) .. controls (1,-1.7) and (1.7,-1.5) .. (c2);
\draw[thick] (b2) .. controls (3,-1.7) and (2.3,-1.5) .. (c2);
\draw[thick] (c1) to (c2);
\draw[very thin] (0.5,0)-- (3.5,0);
\draw[very thin] (0.5,-2)-- (3.5,-2);
}+
\tik{0.5}{%
\path (1,0) coordinate (a1);
\path (3,0) coordinate (a2);
\path (1,-2) coordinate (b1);
\path (3,-2) coordinate (b2);
\path (1.5,-1) coordinate (c1);
\path (2.5,-1) coordinate (c2);
\draw[thick] (a1) .. controls (1.3,0) and (1.5,-0.7) .. (c1);
\draw[thick] (b1) .. controls (1.3,-2) and (1.5,-1.3) .. (c1);
\draw[thick] (a2) .. controls (2.7,0) and (2.5,-0.7) .. (c2);
\draw[thick] (b2) .. controls (2.7,-2) and (2.5,-1.3) .. (c2);
\draw[thick] (c1) to (c2);
\draw[very thin] (0.5,0)-- (3.5,0);
\draw[very thin] (0.5,-2)-- (3.5,-2);
}-2
\tik{0.5}{%
\path (1,0) coordinate (a1);
\path (3,0) coordinate (a2);
\path (1,-2) coordinate (b1);
\path (3,-2) coordinate (b2);
\draw[thick] (a2)to (b1);
\draw[cross] (a1) to (b2);
\draw[very thin] (0.5,0)-- (3.5,0);
\draw[very thin] (0.5,-2)-- (3.5,-2);
}+
\tik{0.5}{%
\path (1,0) coordinate (a1);
\path (3,0) coordinate (a2);
\path (1,-2) coordinate (b1);
\path (3,-2) coordinate (b2);
\path (2,-0.5) coordinate (c1);
\path (2,-1.5) coordinate (c2);
\draw[thick] (a1) .. controls (1,-0.3) and (1.7,-0.5) .. (c1);
\draw[thick] (a2) .. controls (3,-0.3) and (2.3,-0.5) .. (c1);
\draw[thick] (b1) .. controls (1,-1.7) and (1.7,-1.5) .. (c2);
\draw[thick] (b2) .. controls (3,-1.7) and (2.3,-1.5) .. (c2);
\draw[very thin] (0.5,0)-- (3.5,0);
\draw[very thin] (0.5,-2)-- (3.5,-2);
}+
\tik{0.5}{%
\path (1,0) coordinate (a1);
\path (3,0) coordinate (a2);
\path (1,-2) coordinate (b1);
\path (3,-2) coordinate (b2);
\path (1.5,-1) coordinate (c1);
\path (2.5,-1) coordinate (c2);
\draw[thick] (a1) .. controls (1.3,0) and (1.5,-0.7) .. (c1);
\draw[thick] (b1) .. controls (1.3,-2) and (1.5,-1.3) .. (c1);
\draw[thick] (a2) .. controls (2.7,0) and (2.5,-0.7) .. (c2);
\draw[thick] (b2) .. controls (2.7,-2) and (2.5,-1.3) .. (c2);
\draw[very thin] (0.5,0)-- (3.5,0);
\draw[very thin] (0.5,-2)-- (3.5,-2);
},
\end{equation}

\noindent where $\gamma_2\in\Mor_{\cT}([2],[2])$.  
 
 Also, Remarks \ref{re:bbt} and \ref{re:bbt2} show that $(\mathsf{b}\circ \mathsf{b}^t)(1)=\dimm_\FF \VV\ne 0$ and  
\begin{equation}\label{eq:c0}
\mathsf{c}_0 =(\mathsf{b}\circ \mathsf{b}^t)(1)- \dimm_\FF \VV\, \in \, \FF\simeq \Hom_\FF(\FF,\FF)
\end{equation}
is the image under $\cR_\frV$ of
\begin{equation}\label{eq:gamma0}
\gamma_0=\beta\circ\beta^t-(\dimm_\FF \VV)1=
\tik{0.5}{%
\draw[thick] (0,0) circle (0.5cm);
}
\ - (\dimm_\FF\VV)1,\qquad \gamma_0\in\Mor_{\cT}([0],[0]).
\end{equation}



Therefore, we have the following consequence of Corollary \ref{co:RGamma}. 
\begin{proposition}  
If  $\frV=(\VV,\mathsf{b},\times)$ for a vector space $\VV$ endowed with a nonzero cross product $\times$ relative to the nondegenerate symmetric bilinear form $\mathsf{b}$, then $\frV$   
is of tensor type  $\{\mathsf{c}_0,\mathsf{c}_1,\mathsf{c}_2\}$, for $\mathsf{c}_0$ in \eqref{eq:c0} and $\mathsf{c}_1$ and $\mathsf{c}_2$ in \eqref{eq:c1_c2}, and hence the functor $\cR_\frV$ induces a functor $\cR_\Gamma:\cT_\Gamma\rightarrow \cV$, with $\Gamma=\{\gamma_0,\gamma_1,\gamma_2\}$, for $\gamma_0$ in \eqref{eq:gamma0}, $\gamma_1$ in \eqref{eq:gamma1}, and $\gamma_2$ in \eqref{eq:gamma2}.  
\end{proposition}

Equation \eqref{eq:CP3} is equivalent, by the associativity and nondegeneracy of $\mathsf{b}$ and the anticommutativity of $\times$, to the equation:
\begin{equation}
(x\times y)\times z+x\times (y\times z)=2\mathsf{b}( x,z)y-\mathsf{b}( x,y)z-\mathsf{b}( y,z)x,
\end{equation}   
for any $x,y,z\in \VV$. From here it is easy to check that any automorphism of $(\VV,\times)$ is an isometry of $(\VV,\mathsf{b})$. Then $\mathsf{b}$, $\mathsf{b}^t$, $\tau$ and $\times$ are invariant under the action of $\Aut(\VV,\times)$, which is a simple algebraic group of type $\GG_2$ if $\dimm_\FF \VV=7$  and is the special orthogonal group $\SO(\VV,\mathsf{b})$ if $\dimm_\FF \VV = 3$. 

\subsection{Dimension 7 cross products} \ \ 
Suppose now that $\dimm_\FF \VV=7$.
Then the image $\cR_\Gamma\bigl(\Mor_{\cT_\Gamma}([n],[m])\bigr)$, which lies in $\Hom_\FF(\VV^{\otimes n},\VV^{\otimes m})$, is actually contained in $\Hom_{\Aut(\VV,\times)}(\VV^{\otimes n},\VV^{\otimes m})$, the subspace of linear maps invariant under the action of $\Aut(\VV,\times)$.

\begin{theorem}\label{th:G2}
Let $\FF$ be an infinite field of characteristic  $\ne 2,3$, 
and let $\frV=(\VV,\mathsf{b},\times)$ be a 7-dimensional vector space $\VV$ endowed with a nonzero cross product $x\times y$ relative to the nondegenerate symmetric bilinear form $\mathsf{b}$.   For any $n,m\in\NN$, 
\[
\cR_\Gamma\bigl(\Mor_{\cT_\Gamma}([n],[m])\bigr)=\Hom_{\Aut(\VV,\times)}(\VV^{\otimes n},\VV^{\otimes m}).
\]
\end{theorem}
\begin{proof}
The algebra $\Cm=\FF 1\oplus \VV$, with multiplication
\[
(\alpha 1+v)(\beta 1+w)=\bigl(\alpha\beta -\mathsf{b}( v,w)\bigr)1+(\alpha w+\beta v+v\times w)
\]
for $\alpha,\beta\in\FF$ and $v,w\in \VV$, is a Cayley algebra over $\FF$. Any $x\in \Cm$ satisfies the Cayley-Hamilton equation:
\[
x^2-\tr(x)x+\norm(x)1=0,
\]
where the \emph{trace} and \emph{norm} are given by
\[
\tr(\alpha 1+v)=2\alpha,\qquad \norm(\alpha 1+v)=\alpha^2+\mathsf{b}( v,v).
\] 
(See \cite[Chapter III]{Schafer} for  basic facts on Cayley algebras.)   

Denote by $\FF[n\VV]$ the algebra of polynomial maps on the Cartesian power $\VV^n$. According to \cite[(3.23)]{Schwarz} and \cite{ZSh}, the First Fundamental Theorem of  Invariant Theory for $\GG_2$ asserts that the algebra of invariants $\FF[n\VV]^{\Aut(\VV,\times)}$ is generated by the polynomial maps    
\[
\begin{split}
(v_1,\ldots,v_n)&\xrightarrow{\alpha_{ij}} \tr(v_iv_j),\\
(v_1,\ldots,v_n)&\xrightarrow{\beta_{ijk}} \tr\bigl(v_i(v_jv_k)\bigr),\\
(v_1,\ldots,v_n)&\xrightarrow{\gamma_{ijkl}} \skews \tr\bigl(v_i(v_j(v_kv_l))\bigr).
\end{split}
\]
Here $\skews f(x_1,x_2,x_3,x_4)=\sum_{\sigma\in \mathsf{S}_4}(-1)^\sigma f(x_{\sigma(1)},x_{\sigma(2)},x_{\sigma(3)},x_{\sigma(4)})$, denotes the complete skew symmetrization of the map $f$ relative to the symmetric group $\mathsf{S}_4$,  and $(-1)^\sigma$ is the
sign of the permutation $\sigma$.  

Let us note that for any $n\in\NN$, $n\geq 2$, and any $v_1,\ldots,v_n\in \VV$,
\[
\alpha_{ij}(v_1,\ldots,v_n)=-2\mathsf{b}( v_i,v_j).
\]
Also,
\[
\begin{split}
v_i(v_jv_k)&=-\mathsf{b}( v_j,v_k)v_i+v_i(v_j\times v_k)\\
 &=-\mathsf{b}( v_j,v_k)v_i-\mathsf{b}( v_i,v_j\times v_k)1+v_i\times (v_j\times v_k),
\end{split}
\]
and since $\tr(v)=0$ for any $v\in \VV$, it follows that for $n\geq 3$,
\[
\beta_{ijk}(v_1,\ldots,v_n)=-\mathsf{b}( v_i,v_j\times v_k).
\]
Finally,
\[
\begin{split}
\tr\bigl(v_i(v_j(v_kv_l))\bigr)&= \tr\bigl((v_iv_j)(v_kv_l)\bigr)\\
 &=\tr\bigl((-\mathsf{b}( v_i,v_j)1+v_i\times v_j)(-\mathsf{b}( v_k,v_l)1+v_k\times v_l)\bigr)\\
 &=2\mathsf{b}( v_i,v_j)\mathsf{b}( v_k,v_l)-2\mathsf{b}( v_i\times v_j,v_k\times v_l).
\end{split}
\]
The skew symmetrization of the first summand is $0$, and because $\mathsf{b}$ is symmetric and $\times$ is anticommutative, the skew symmetrization of \\  $\mathsf{b}( v_i\times v_j,v_k\times v_l)$ is 8 times
\begin{multline*}
\mathsf{b}( v_i\times v_j,v_k\times v_l)+\mathsf{b}( v_j\times v_k,v_i\times v_l)+\mathsf{b}( v_k\times v_i,v_j\times v_l)\\
 =\mathsf{b}\bigl((v_i\times v_j)\times v_l+(v_j\times v_k)\times v_i+(v_k\times v_i)\times v_j,v_l\bigr).
\end{multline*}
But because of \eqref{eq:CP},  we get:
\[
\begin{split}
(v_j\times v_k)\times v_i&=-(v_k\times v_j)\times v_i\\
 &=v_k\times (v_j\times v_i)-2\mathsf{b}( v_i,v_k)v_j+\mathsf{b}( v_j,v_i)v_k+\mathsf{b}( v_j,v_k)v_i\\
 &=(v_i\times v_j)\times v_k+\mathsf{b}( v_i,v_j)v_k-2\mathsf{b}( v_i,v_k)v_j+\mathsf{b}( v_j,v_k)v_i,\\[2pt]
(v_k\times v_i)\times v_j&= -v_k\times (v_i\times v_j)+2\mathsf{b}( v_k,v_j)v_i-\mathsf{b}( v_k,v_i)v_j-\mathsf{b}( v_i,v_j)v_k\\
 &=(v_i\times v_j)\times v_k-\mathsf{b}( v_i,v_j)v_k-\mathsf{b}( v_i,v_k)v_j+2\mathsf{b}( v_j,v_k)v_i.
\end{split}
\]
Hence, the complete skew symmetrization of $\mathsf{b}( v_i\times v_j,v_k\times v_l)$ is $24$ times
\[
\mathsf{b}\bigl((v_i\times v_j)\times v_k,v_l\bigr)-\mathsf{b}( v_i,v_k)\mathsf{b}( v_j,v_l)+\mathsf{b}( v_i,v_l)\mathsf{b}( v_j,v_k).
\]
Therefore, the above quoted results of Schwarz and Zubkov--Shestakov show that the algebra of invariants $\FF[n\VV]^{\Aut(\VV,\times)}$ is generated by the polynomial maps:
\[
\begin{split}
(v_1,\ldots,v_n)&\longrightarrow \mathsf{b}( v_i,v_j),\\
(v_1,\ldots,v_n)&\longrightarrow \mathsf{b}( v_i,v_j\times v_k),\\
(v_1,\ldots,v_n)&\longrightarrow \mathsf{b}( v_i\times v_j,v_k\times v_l),
\end{split}
\]
and we conclude that the space $\Hom_{\Aut(\VV,\times)}(\VV^{\otimes n},\VV^{\otimes m})$,  which is isomorphic to $\Hom_{\Aut(\VV,\times)}(\VV^{\otimes (n+m)},\FF)$, i.e., the space of multilinear maps $\VV^{n+m}\rightarrow \FF$
invariant under the action of $\Aut(\VV,\times)$, is the linear span of the set of linear maps $\VV^{\otimes n}\rightarrow \VV^{\otimes m}$ generated by composition and tensor products by $1_\VV$, $\mathsf{b}$, $\mathsf{b}^t$, $\tau$ and $\times$. This is precisely $\cR_\Gamma\bigl(\Mor_{\cT_\Gamma}([n],[m])\bigr)$. 
\end{proof}

Our goal now is to prove that $\cR_\Gamma$ actually gives a bijection $\Mor_{\cT_\Gamma}([n],[m])\rightarrow \Hom_{\Aut(\VV,\times)}(\VV^{\otimes n},\VV^{\otimes m})$ and to find bases of these spaces. We will start by working in $\cT_\Gamma$.

The morphisms of $\cT_\Gamma$ are the equivalence classes of $3$-tangles modulo the relations in $\Gamma=\{\gamma_0,\gamma_1,\gamma_2\}$, so we will attempt to get a set of $3$-tangles $\gamma: [n]\rightarrow [m]$ that are representatives of a spanning set of $\Mor_{\cT_\Gamma}([n],[m])$. Later on, this spanning set will be shown to be a basis.

We will use the same symbols to denote morphisms in $\cT$ and their images in $\cT_\Gamma$. Our argument will involve using the following steps aimed at 
\begin{equation}\label{eq:steps}
\textbf{eliminating circles, crossings, triangles, squares and pentagons.}
\end{equation}
\begin{itemize}
\item[\textbf{1.}] Since $\gamma_0$ is one of the relations, we can get rid of circles, i.e., any $3$-tangle modulo $\Gamma$ equals a $3$-tangle without circles.  
Also, the relation $\gamma_2$ allows us to suppress any crossings 
$
\tik{0.3}{%
\path (1,0) coordinate (a1);
\path (3,0) coordinate (a2);
\path (1,-2) coordinate (b1);
\path (3,-2) coordinate (b2);
\draw[thick] (a2)to (b1);
\draw[cross] (a1) to (b2);
%
}
$,
i.e., a $3$-tangle with crossings is equivalent modulo $\Gamma$ to a linear span of $3$-tangles without crossings. 

Note that if there are no crossings, a $3$-tangle is the disjoint union of its connected components:
\[
\tik{0.5}{%
\draw[very thin] (0.5,0) -- (7,0);
\draw[very thin] (0.5,-2) -- (7.5,-2);
\path (1,0) coordinate (a1);
\path (2,0) coordinate (a2);
\path (3,0) coordinate (a3);
\path (4,0) coordinate (a4);
\path (5,0) coordinate (a5);
\path (6,0) coordinate (a6);
\path (1,-2) coordinate (b1);
\path (2,-2) coordinate (b2);
\path (3,-2) coordinate (b3);
\path (5,-2) coordinate (b4);
\path (7,-2) coordinate (b5);
\path (1.5,-0.5) coordinate (c1);
\path (1.5,-1.5) coordinate (c2);
\draw[thick] (a1) -- (c1) -- (a2);
\draw[thick] (b1) -- (c2) -- (b2);
\draw[thick] (c1) -- (c2);
\draw[thick] (a3) to [bend right=60] (a4);
\draw[thick] (a5) -- (b3);
\path (6,-0.7) coordinate (t1);
\path (5.5,-1.5) coordinate (t2);
\path (6.5,-1.5) coordinate (t3);
\draw[thick] (t1) -- (t2) -- (t3) -- cycle;
\draw[thick] (a6) to (t1);
\draw[thick] (t2) to (b4);
\draw[thick] (t3) to (b5);
}\ =\ 
\tik{0.5}{%
\draw[very thin] (0.5,0) -- (2.5,0);
\draw[very thin] (0.5,-2) -- (2.5,-2);
\path (1,0) coordinate (a1);
\path (2,0) coordinate (a2);
\path (1,-2) coordinate (b1);
\path (2,-2) coordinate (b2);
\path (1.5,-0.5) coordinate (c1);
\path (1.5,-1.5) coordinate (c2);
\draw[thick] (a1) -- (c1) -- (a2);
\draw[thick] (b1) -- (c2) -- (b2);
\draw[thick] (c1) -- (c2);
}
\quad
\sqcup 
\quad
\tik{0.5}{%
\draw[very thin] (2.5,0) -- (4.5,0);
\draw[very thin] (2.5,-2) -- (4.5,-2);
\path (3,0) coordinate (a3);
\path (4,0) coordinate (a4);
\draw[thick] (a3) to [bend right=60] (a4);
}
\quad
\sqcup 
\quad
\tik{0.5}{%
\draw[very thin] (4.5,0) -- (5.5,0);
\draw[very thin] (4.5,-2) -- (5.5,-2);
\path (5,0) coordinate (a5);
\path (5,-2) coordinate (b3);
\draw[thick] (a5) -- (b3);
}
\quad
\sqcup 
\quad
\tik{0.5}{%
\draw[very thin] (4.5,0) -- (7.5,0);
\draw[very thin] (4.5,-2) -- (7.5,-2);
\path (6,0) coordinate (a6);
\path (5,-2) coordinate (b4);
\path (7,-2) coordinate (b5);
\path (6,-0.7) coordinate (t1);
\path (5.5,-1.5) coordinate (t2);
\path (6.5,-1.5) coordinate (t3);
\draw[thick] (t1) -- (t2) -- (t3) -- cycle;
\draw[thick] (a6) to (t1);
\draw[thick] (t2) to (b4);
\draw[thick] (t3) to (b5);
}
\]

\medskip  

\item[\textbf{2.}] In $\cT_\Gamma$ we have:
\[
\mu\circ\tau\circ\beta^t=\begin{cases}
\mu\circ(\tau\circ\beta^t)=\mu\circ\beta^t&\text{because $\tau\circ\beta^t=\beta^t$,}\\
  (\mu\circ\tau)\circ\beta^t=-\mu\circ\beta^t&\text{because of relation $\gamma_1$.}
  \end{cases}
\]
\[
\tik{0.5}{%
\path (1,0) coordinate (a1);
\path (2,0) coordinate (a2);
\path (1,-1) coordinate (b1);
\path (2,-1) coordinate (b2);
\path (1.5,-1.5) coordinate (c);
\path (1.5,-2) coordinate (d);
\path (1.5,0.5) coordinate (e);
\draw[very thin] (0.5,-1) -- (2.5,-1);
\draw[thick] (b1) to [bend left=80] (b2);
\draw[thick] (b1) to (c);
\draw[thick] (b2) to (c);
\draw[thick] (c) to (d);
}
\quad 
=
\quad
\tik{0.5}{%
\path (1,0) coordinate (a1);
\path (2,0) coordinate (a2);
\path (1,-1) coordinate (b1);
\path (2,-1) coordinate (b2);
\path (1.5,-1.5) coordinate (c);
\path (1.5,-2) coordinate (d);
\path (1.5,0.5) coordinate (e);
\draw[thick] (a2) to (b1);
\draw[cross] (a1) to (b2);
\draw[thick] (a1) to [bend left=80] (a2);
\draw[very thin] (0.5,0) -- (2.5,0);
\draw[very thin] (0.5,-1) -- (2.5,-1);
\draw[thick] (b1) to (c);
\draw[thick] (b2) to (c);
\draw[thick] (c) to (d);
}
\quad
=
\quad -
\tik{0.5}{%
\path (1,0) coordinate (a1);
\path (2,0) coordinate (a2);
\path (1,-1) coordinate (b1);
\path (2,-1) coordinate (b2);
\path (1.5,-0.5) coordinate (c);
\path (1.5,-1.5) coordinate (d);
\path (1.5,0.5) coordinate (e);
\path (1.5,-2) coordinate (f);
\draw[thick] (a1) to [bend left=80] (a2);
\draw[very thin] (0.5,0) -- (2.5,0);
\draw[thick] (a1) to (c);
\draw[thick] (a2) to (c);
\draw[thick] (c) to (d);
}
\]
Therefore, in $\cT_\Gamma$, $\mu\circ\beta^t=0$:
\[
\tik{0.5}{%
\draw[thick] (0,0) -- (0,-1);
\draw[thick] (0,0.5) circle (0.5cm);
\node at (2,0) {$=\ 0$};
}
\]

\medskip
\item[\textbf{3.}] Let us prove next (see \cite{Rost04}) that the following relation holds:
\[
\tik{0.5}{%
\draw[thick] (0,-0.5) -- (0,-1);
\draw[thick] (0,0) circle (0.5cm);
\draw[thick] (0,0.5) -- (0,1);
\node at (2,0) {$=\ -6$};
\draw[thick] (3.5,1) -- (3.5,-1);
}
\]
Indeed, relation $\gamma_2$ gives:
\[
\begin{split}
\tik{0.5}{%
\path (-1,0.7) coordinate (x1);
\path (-1,-0.7) coordinate (x2);
\path (1,1) coordinate (f1);
\path (1,-1) coordinate (f2);
\path (0,0.4) coordinate (m1);
\path (0,-0.4) coordinate (m2);
\draw[thick] (f1) -- (m1) -- (m2) -- (f2);
\draw[thick] (m1) -- (x1);
\draw[thick] (m2) -- (x2);
\draw[thick] (x1) .. controls (-2,1) and (-2,-1) .. (x2);
} &= 
-
\tik{0.5}{%
\path (-1,0.7) coordinate (x1);
\path (-1,-0.7) coordinate (x2);
\path (1,1) coordinate (f1);
\path (1,-1) coordinate (f2);
\path (0.4,0) coordinate (m1);
\path (-0.4,0) coordinate (m2);
\draw[thick] (f1) -- (m1) -- (f2);
\draw[thick] (m1) -- (m2) -- (x1);
\draw[thick] (m2) -- (x2);
\draw[thick] (x1) .. controls (-2,1) and (-2,-1) .. (x2);
}
+2
\tik{0.5}{%
\path (-1,0.7) coordinate (x1);
\path (-1,-0.7) coordinate (x2);
\path (1,1) coordinate (f1);
\path (1,-1) coordinate (f2);
\path (0.4,0) coordinate (m1);
\path (-0.4,0) coordinate (m2);
\draw[thick] (f2) .. controls (m1) and (m1) .. (x1);
\draw[cross] (f1) .. controls (m1) and (m1) .. (x2);
\draw[thick] (x1) .. controls (-2,1) and (-2,-1) .. (x2);
}
-
\tik{0.5}{%
\path (-1,0.7) coordinate (x1);
\path (-1,-0.7) coordinate (x2);
\path (1,1) coordinate (f1);
\path (1,-1) coordinate (f2);
\path (0.4,0) coordinate (m1);
\path (-0.4,0) coordinate (m2);
\draw[thick] (f1) -- (f2);
\draw[thick] (0,0) circle (0.7cm);
}
\ -
\tik{0.5}{%
\path (-1,0.7) coordinate (x1);
\path (-1,-0.7) coordinate (x2);
\path (1,1) coordinate (f1);
\path (1,-1) coordinate (f2);
\path (0.4,0) coordinate (m1);
\path (-0.4,0) coordinate (m2);
\draw[thick] (f1) .. controls (m1) and (m1) .. (x1); 
\draw[thick] (f2) .. controls (m1) and (m1) .. (x2); 
\draw[thick] (x1) .. controls (-2,1) and (-2,-1) .. (x2);
}
\\[10pt]
&=
0\ +\ 2\ 
\tik{0.5}{%
\draw[thick] (0,1) -- (0,-1);
}
\ -\ 
\tik{0.5}{%
\draw[thick] (-0.8,0) circle (0.5cm);
\draw[thick] (0,1) -- (0,-1);
}
\ -\ 
\tik{0.5}{%
\draw[thick] (0,1) -- (0,-1);
}\\[10pt]
&=
(1-\dimm_\FF \VV)\ 
\tik{0.5}{%
\draw[thick] (0,1) -- (0,-1);
}
= -6\  \tik{0.5}{%
\draw[thick] (0,1) -- (0,-1);
}\ .
\end{split}
\]

\medskip
\item[\textbf{4.}] Again with relation $\gamma_2$,  we get:
\[
\begin{split}
\tik{0.5}{%
\path (0,0) coordinate (origin);
\path (150:0.5cm) coordinate (t1);
\path (30:0.5cm) coordinate (t2);
\path (-90: 0.5cm) coordinate (t3);
\path (150:1cm) coordinate (e1);
\path (30:1cm) coordinate (e2);
\path (-90:1cm) coordinate (e3);
\draw[thick] (t1) -- (t2) -- (t3) -- cycle;
\draw[thick] (e1) -- (t1) (e2) -- (t2) (e3) -- (t3);
}
+
\tik{0.5}{%
\path (0,0) coordinate (origin);
\path (150:0.5cm) coordinate (t1);
\path (30:0.5cm) coordinate (t2);
\path (-90: 0.5cm) coordinate (t3);
\path (150:1cm) coordinate (e1);
\path (30:1cm) coordinate (e2);
\path (-90:1cm) coordinate (e3);
\path (0,0.2) coordinate (m);
\draw[thick] (e1) -- (m) -- (e2);
\draw[thick] (m) --++(0,-0.3);
\draw[thick] (e3)-- ++(0,0.3);
\draw[thick] (0,-0.4) circle (0.3cm);
}
&=
2
\tik{0.5}{%
\path (0,0) coordinate (origin);
\path (150:0.5cm) coordinate (t1);
\path (30:0.5cm) coordinate (t2);
\path (-90: 0.5cm) coordinate (t3);
\path (150:1cm) coordinate (e1);
\path (30:1cm) coordinate (e2);
\path (-90:1cm) coordinate (e3);
\path (0,0.2) coordinate (m);
\path (0,-0.5) coordinate (p);
\draw[thick] (e2) .. controls (t1) and (-1,0) .. (p);
\draw[cross] (e1) .. controls (t2) and (1,0) .. (p);
\draw[thick] (e3) -- (p);
}
-
\tik{0.5}{%
\path (0,0) coordinate (origin);
\path (150:0.5cm) coordinate (t1);
\path (30:0.5cm) coordinate (t2);
\path (-90: 0.5cm) coordinate (t3);
\path (150:1cm) coordinate (e1);
\path (30:1cm) coordinate (e2);
\path (-90:1cm) coordinate (e3);
\path (0,0.2) coordinate (m);
\draw[thick] (e1) -- (origin) -- (e2);
\draw[thick] (origin)-- (e3);
}
-
\tik{0.5}{%
\path (0,0) coordinate (origin);
\path (150:0.5cm) coordinate (t1);
\path (30:0.5cm) coordinate (t2);
\path (-90: 0.5cm) coordinate (t3);
\path (150:1cm) coordinate (e1);
\path (30:1cm) coordinate (e2);
\path (-90:1cm) coordinate (e3);
\path (0,0.2) coordinate (m);
\draw[thick] (e1) to [bend right=40] (e2);
\draw[thick] (e3)-- ++(0,0.3);
\draw[thick] (0,-0.4) circle (0.3cm);
}
\\[10pt]
&= -3 
\tik{0.5}{%
\path (0,0) coordinate (origin);
\path (150:0.5cm) coordinate (t1);
\path (30:0.5cm) coordinate (t2);
\path (-90: 0.5cm) coordinate (t3);
\path (150:1cm) coordinate (e1);
\path (30:1cm) coordinate (e2);
\path (-90:1cm) coordinate (e3);
\path (0,0.2) coordinate (m);
\draw[thick] (e1) -- (origin) -- (e2);
\draw[thick] (origin) -- (e3);
}
\qquad\text{because of $\gamma_1$ and Step \textbf{2},}
\end{split}
\]
and because of Step \textbf{3},  we conclude
\[
\tik{0.5}{%
\path (0,0) coordinate (origin);
\path (150:0.5cm) coordinate (t1);
\path (30:0.5cm) coordinate (t2);
\path (-90: 0.5cm) coordinate (t3);
\path (150:1cm) coordinate (e1);
\path (30:1cm) coordinate (e2);
\path (-90:1cm) coordinate (e3);
\draw[thick] (t1) -- (t2) -- (t3) -- cycle;
\draw[thick] (e1) -- (t1) (e2) -- (t2) (e3) -- (t3);
}
= 3
\tik{0.5}{%
\path (0,0) coordinate (origin);
\path (150:0.5cm) coordinate (t1);
\path (30:0.5cm) coordinate (t2);
\path (-90: 0.5cm) coordinate (t3);
\path (150:1cm) coordinate (e1);
\path (30:1cm) coordinate (e2);
\path (-90:1cm) coordinate (e3);
\path (0,0.2) coordinate (m);
\draw[thick] (e1) -- (origin) -- (e2);
\draw[thick] (origin) -- (e3);
}\ .
\]

\medskip
\item[\textbf{5.}] A new application of relation $\gamma_2$ gives
\[
\begin{split}
\tik{0.5}{%
\path (0,0) coordinate (origin);
\path (135:0.5cm) coordinate (c1);
\path (45:0.5cm) coordinate (c2);
\path (-45: 0.5cm) coordinate (c3);
\path (-135: 0.5cm) coordinate (c4);
\path (135:1cm) coordinate (e1);
\path (45:1cm) coordinate (e2);
\path (-45:1cm) coordinate (e3);
\path (-135:1cm) coordinate (e4);
\draw[thick] (c1) -- (c2) -- (c3) -- (c4) -- cycle;
\draw[thick] (c1) -- (e1) (c2) -- (e2) (c3) -- (e3) (c4) -- (e4);
}
&=-
\tik{0.5}{%
\path (0,0) coordinate (origin);
\path (135:0.5cm) coordinate (c1);
\path (45:0.5cm) coordinate (c2);
\path (-45: 0.5cm) coordinate (c3);
\path (-135: 0.5cm) coordinate (c4);
\path (135:1cm) coordinate (e1);
\path (45:1cm) coordinate (e2);
\path (-45:1cm) coordinate (e3);
\path (-135:1cm) coordinate (e4);
\path (0,0.4) coordinate (m);
\path (0,0.2) coordinate (t1);
\draw[thick] (t1) -- (c3) -- (c4) -- cycle;
\draw[thick] (e1) -- (m) -- (e2);
\draw[thick] (m) -- (t1) (e4) -- (c4) (e3) -- (c3);
}
+2
\tik{0.5}{%
\path (0,0) coordinate (origin);
\path (135:0.5cm) coordinate (c1);
\path (45:0.5cm) coordinate (c2);
\path (-45: 0.5cm) coordinate (c3);
\path (-135: 0.5cm) coordinate (c4);
\path (135:1cm) coordinate (e1);
\path (45:1cm) coordinate (e2);
\path (-45:1cm) coordinate (e3);
\path (-135:1cm) coordinate (e4);
\path (0,0.4) coordinate (m);
\path (0,0.2) coordinate (t1);
\draw[thick] (e2) .. controls (c1) and (-1,0) .. (c4);
\draw[cross] (e1) .. controls (c2) and (1,0) .. (c3);
\draw[thick] (e3) -- (c3) -- (c4) -- (e4);
}-
\tik{0.5}{%
\path (0,0) coordinate (origin);
\path (135:0.5cm) coordinate (c1);
\path (45:0.5cm) coordinate (c2);
\path (-45: 0.5cm) coordinate (c3);
\path (-135: 0.5cm) coordinate (c4);
\path (135:1cm) coordinate (e1);
\path (45:1cm) coordinate (e2);
\path (-45:1cm) coordinate (e3);
\path (-135:1cm) coordinate (e4);
\path (0,0.4) coordinate (m);
\path (0,0.2) coordinate (t1);
\path (0.3,-0.4) coordinate (t2);
\path (-0.3,-0.4) coordinate (t3);
\draw[thick] (e1) -- (-0.4,0) -- (0.4,0) -- (e2);
\draw[thick] (e4) --  (-0.4,0) (e3) -- (0.4,0);
} 
-
\tik{0.5}{%
\path (0,0) coordinate (origin);
\path (135:0.5cm) coordinate (c1);
\path (45:0.5cm) coordinate (c2);
\path (-45: 0.5cm) coordinate (c3);
\path (-135: 0.5cm) coordinate (c4);
\path (135:1cm) coordinate (e1);
\path (45:1cm) coordinate (e2);
\path (-45:1cm) coordinate (e3);
\path (-135:1cm) coordinate (e4);
\path (0,0.4) coordinate (m);
\path (0,0.2) coordinate (t1);
\path (0.3,-0.4) coordinate (t2);
\path (-0.3,-0.4) coordinate (t3);
\draw[thick] (e1) to [bend right=40] (e2);
\draw[thick] (0,-0.2) circle (0.4cm);
\draw[thick] (e4) -- (c4) (e3) -- (c3);
}
\\[10pt]
&=-3
\tik{0.5}{%
\path (0,0) coordinate (origin);
\path (135:0.5cm) coordinate (c1);
\path (45:0.5cm) coordinate (c2);
\path (-45: 0.5cm) coordinate (c3);
\path (-135: 0.5cm) coordinate (c4);
\path (135:1cm) coordinate (e1);
\path (45:1cm) coordinate (e2);
\path (-45:1cm) coordinate (e3);
\path (-135:1cm) coordinate (e4);
\path (0,0.4) coordinate (m);
\path (0,0.2) coordinate (t1);
\path (0.3,-0.4) coordinate (t2);
\path (-0.3,-0.4) coordinate (t3);
\draw[thick] (e1) -- (m) -- (e2);
\draw[thick] (e3) -- (0,-0.4) -- (e4);
\draw[thick] (m) -- (0,-0.4);
}
+2
\tik{0.5}{%
\path (0,0) coordinate (origin);
\path (135:0.5cm) coordinate (c1);
\path (45:0.5cm) coordinate (c2);
\path (-45: 0.5cm) coordinate (c3);
\path (-135: 0.5cm) coordinate (c4);
\path (135:1cm) coordinate (e1);
\path (45:1cm) coordinate (e2);
\path (-45:1cm) coordinate (e3);
\path (-135:1cm) coordinate (e4);
\path (0,0.4) coordinate (m);
\path (0,0.2) coordinate (t1);
\draw[thick] (e2) .. controls (c1) and (-1,0) .. (c4);
\draw[cross] (e1) .. controls (c2) and (1,0) .. (c3);
\draw[thick] (e3) -- (c3) -- (c4) -- (e4);
}-
\tik{0.5}{%
\path (0,0) coordinate (origin);
\path (135:0.5cm) coordinate (c1);
\path (45:0.5cm) coordinate (c2);
\path (-45: 0.5cm) coordinate (c3);
\path (-135: 0.5cm) coordinate (c4);
\path (135:1cm) coordinate (e1);
\path (45:1cm) coordinate (e2);
\path (-45:1cm) coordinate (e3);
\path (-135:1cm) coordinate (e4);
\path (0,0.4) coordinate (m);
\path (0,0.2) coordinate (t1);
\path (0.3,-0.4) coordinate (t2);
\path (-0.3,-0.4) coordinate (t3);
\draw[thick] (e1) -- (-0.4,0) -- (0.4,0) -- (e2);
\draw[thick] (e4) --  (-0.4,0) (e3) -- (0.4,0);
} 
+6
\tik{0.5}{%
\path (0,0) coordinate (origin);
\path (135:0.5cm) coordinate (c1);
\path (45:0.5cm) coordinate (c2);
\path (-45: 0.5cm) coordinate (c3);
\path (-135: 0.5cm) coordinate (c4);
\path (135:1cm) coordinate (e1);
\path (45:1cm) coordinate (e2);
\path (-45:1cm) coordinate (e3);
\path (-135:1cm) coordinate (e4);
\path (0,0.4) coordinate (m);
\path (0,0.2) coordinate (t1);
\path (0.3,-0.4) coordinate (t2);
\path (-0.3,-0.4) coordinate (t3);
\draw[thick] (e1) to [bend right=40] (e2);
\draw[thick] (e3) to [bend right=40] (e4);
}
\end{split}
\]
and
\[
\begin{split}
\tik{0.5}{%
\path (0,0) coordinate (origin);
\path (135:0.5cm) coordinate (c1);
\path (45:0.5cm) coordinate (c2);
\path (-45: 0.5cm) coordinate (c3);
\path (-135: 0.5cm) coordinate (c4);
\path (135:1cm) coordinate (e1);
\path (45:1cm) coordinate (e2);
\path (-45:1cm) coordinate (e3);
\path (-135:1cm) coordinate (e4);
\path (0,0.4) coordinate (m);
\path (0,0.2) coordinate (t1);
\draw[thick] (e2) .. controls (c1) and (-1,0) .. (c4);
\draw[cross] (e1) .. controls (c2) and (1,0) .. (c3);
\draw[thick] (e3) -- (c3) -- (c4) -- (e4);
}
&=-
\tik{0.5}{%
\path (0,0) coordinate (origin);
\path (135:0.5cm) coordinate (c1);
\path (45:0.5cm) coordinate (c2);
\path (-45: 0.5cm) coordinate (c3);
\path (-135: 0.5cm) coordinate (c4);
\path (135:1cm) coordinate (e1);
\path (45:1cm) coordinate (e2);
\path (-45:1cm) coordinate (e3);
\path (-135:1cm) coordinate (e4);
\path (0,0.4) coordinate (m);
\path (0,0.2) coordinate (t1);
\draw[thick] (e2) .. controls (-0.5,0.6) and (-1,0.3) ..  (t1);
\draw[cross] (e1) .. controls (0.5,0.6) and (1,0.3) .. (t1);
\draw[thick] (t1) -- (0,-0.2);
\draw[thick] (e3) -- (0,-0.2) -- (e4);
}
+2
\tik{0.5}{%
\path (0,0) coordinate (origin);
\path (135:0.5cm) coordinate (c1);
\path (45:0.5cm) coordinate (c2);
\path (-45: 0.5cm) coordinate (c3);
\path (-135: 0.5cm) coordinate (c4);
\path (135:1cm) coordinate (e1);
\path (45:1cm) coordinate (e2);
\path (-45:1cm) coordinate (e3);
\path (-135:1cm) coordinate (e4);
\path (0,0.4) coordinate (m);
\path (0,0.2) coordinate (t1);
\draw[thick] (e2) .. controls (-0.5,0.5) and (-0.5,0.1) .. (-0.5,0);
\draw[cross] (e1) .. controls (0.5,0.5) and (0.5,0.1) .. (0.5,0);
\draw[thick] (e4) .. controls (0.5,-0.5) and (0.5,-0.1) .. (0.5,0);
\draw[cross] (e3) .. controls (-0.5,-0.5) and (-0.5,-0.1) .. (-0.5,0);
}
-
\tik{0.5}{%
\path (0,0) coordinate (origin);
\path (135:0.5cm) coordinate (c1);
\path (45:0.5cm) coordinate (c2);
\path (-45: 0.5cm) coordinate (c3);
\path (-135: 0.5cm) coordinate (c4);
\path (135:1cm) coordinate (e1);
\path (45:1cm) coordinate (e2);
\path (-45:1cm) coordinate (e3);
\path (-135:1cm) coordinate (e4);
\path (0,0.4) coordinate (m);
\path (0,0.2) coordinate (t1);
\draw[thick] (e2) -- (e4);
\draw[cross] (e1) -- (e3);
}
-
\tik{0.5}{%
\path (0,0) coordinate (origin);
\path (135:0.5cm) coordinate (c1);
\path (45:0.5cm) coordinate (c2);
\path (-45: 0.5cm) coordinate (c3);
\path (-135: 0.5cm) coordinate (c4);
\path (135:1cm) coordinate (e1);
\path (45:1cm) coordinate (e2);
\path (-45:1cm) coordinate (e3);
\path (-135:1cm) coordinate (e4);
\path (0,0.4) coordinate (m);
\path (0,0.2) coordinate (t1);
\draw[thick] (e2) .. controls (-0.5,0.5) and (-0.5,0.1) .. (-0.5,0);
\draw[cross] (e1) .. controls (0.5,0.5) and (0.5,0.1) .. (0.5,0);
\draw[thick] (0.5,0) .. controls (0.5,-0.4) and (-0.5,-0.4) .. (-0.5,0);
\draw[thick] (e3) to [bend right=30] (e4);
}
\\[10pt]
&=
\tik{0.5}{%
\path (0,0) coordinate (origin);
\path (135:0.5cm) coordinate (c1);
\path (45:0.5cm) coordinate (c2);
\path (-45: 0.5cm) coordinate (c3);
\path (-135: 0.5cm) coordinate (c4);
\path (135:1cm) coordinate (e1);
\path (45:1cm) coordinate (e2);
\path (-45:1cm) coordinate (e3);
\path (-135:1cm) coordinate (e4);
\path (0,0.4) coordinate (m);
\path (0,0.2) coordinate (t1);
\path (0.3,-0.4) coordinate (t2);
\path (-0.3,-0.4) coordinate (t3);
\draw[thick] (e1) -- (m) -- (e2);
\draw[thick] (e3) -- (0,-0.4) -- (e4);
\draw[thick] (m) -- (0,-0.4);
}
+2
\tik{0.5}{%
\path (0,0) coordinate (origin);
\path (135:0.5cm) coordinate (c1);
\path (45:0.5cm) coordinate (c2);
\path (-45: 0.5cm) coordinate (c3);
\path (-135: 0.5cm) coordinate (c4);
\path (135:1cm) coordinate (e1);
\path (45:1cm) coordinate (e2);
\path (-45:1cm) coordinate (e3);
\path (-135:1cm) coordinate (e4);
\path (0,0.4) coordinate (m);
\path (0,0.2) coordinate (t1);
\draw[thick] (e1) to [bend left=40] (e4);
\draw[thick] (e2) to [bend right=40] (e3);
}-
\tik{0.5}{%
\path (0,0) coordinate (origin);
\path (135:0.5cm) coordinate (c1);
\path (45:0.5cm) coordinate (c2);
\path (-45: 0.5cm) coordinate (c3);
\path (-135: 0.5cm) coordinate (c4);
\path (135:1cm) coordinate (e1);
\path (45:1cm) coordinate (e2);
\path (-45:1cm) coordinate (e3);
\path (-135:1cm) coordinate (e4);
\path (0,0.4) coordinate (m);
\path (0,0.2) coordinate (t1);
\draw[thick] (e2) -- (e4);
\draw[cross] (e1) -- (e3);
}
-
\tik{0.5}{%
\path (0,0) coordinate (origin);
\path (135:0.5cm) coordinate (c1);
\path (45:0.5cm) coordinate (c2);
\path (-45: 0.5cm) coordinate (c3);
\path (-135: 0.5cm) coordinate (c4);
\path (135:1cm) coordinate (e1);
\path (45:1cm) coordinate (e2);
\path (-45:1cm) coordinate (e3);
\path (-135:1cm) coordinate (e4);
\path (0,0.4) coordinate (m);
\path (0,0.2) coordinate (t1);
\path (0.3,-0.4) coordinate (t2);
\path (-0.3,-0.4) coordinate (t3);
\draw[thick] (e1) to [bend right=40] (e2);
\draw[thick] (e3) to [bend right=40] (e4);
}\ .
\end{split}.
\]
Thus,  we have
\[
\tik{0.5}{%
\path (0,0) coordinate (origin);
\path (135:0.5cm) coordinate (c1);
\path (45:0.5cm) coordinate (c2);
\path (-45: 0.5cm) coordinate (c3);
\path (-135: 0.5cm) coordinate (c4);
\path (135:1cm) coordinate (e1);
\path (45:1cm) coordinate (e2);
\path (-45:1cm) coordinate (e3);
\path (-135:1cm) coordinate (e4);
\draw[thick] (c1) -- (c2) -- (c3) -- (c4) -- cycle;
\draw[thick] (c1) -- (e1) (c2) -- (e2) (c3) -- (e3) (c4) -- (e4);
}
=
-
\tik{0.5}{%
\path (0,0) coordinate (origin);
\path (135:0.5cm) coordinate (c1);
\path (45:0.5cm) coordinate (c2);
\path (-45: 0.5cm) coordinate (c3);
\path (-135: 0.5cm) coordinate (c4);
\path (135:1cm) coordinate (e1);
\path (45:1cm) coordinate (e2);
\path (-45:1cm) coordinate (e3);
\path (-135:1cm) coordinate (e4);
\path (0,0.4) coordinate (m);
\path (0,0.2) coordinate (t1);
\path (0.3,-0.4) coordinate (t2);
\path (-0.3,-0.4) coordinate (t3);
\draw[thick] (e1) -- (m) -- (e2);
\draw[thick] (e3) -- (0,-0.4) -- (e4);
\draw[thick] (m) -- (0,-0.4);
}
-
\tik{0.5}{%
\path (0,0) coordinate (origin);
\path (135:0.5cm) coordinate (c1);
\path (45:0.5cm) coordinate (c2);
\path (-45: 0.5cm) coordinate (c3);
\path (-135: 0.5cm) coordinate (c4);
\path (135:1cm) coordinate (e1);
\path (45:1cm) coordinate (e2);
\path (-45:1cm) coordinate (e3);
\path (-135:1cm) coordinate (e4);
\path (0,0.4) coordinate (m);
\path (0,0.2) coordinate (t1);
\path (0.3,-0.4) coordinate (t2);
\path (-0.3,-0.4) coordinate (t3);
\draw[thick] (e1) -- (-0.4,0) -- (0.4,0) -- (e2);
\draw[thick] (e4) --  (-0.4,0) (e3) -- (0.4,0);
} 
+4
\tik{0.5}{%
\path (0,0) coordinate (origin);
\path (135:0.5cm) coordinate (c1);
\path (45:0.5cm) coordinate (c2);
\path (-45: 0.5cm) coordinate (c3);
\path (-135: 0.5cm) coordinate (c4);
\path (135:1cm) coordinate (e1);
\path (45:1cm) coordinate (e2);
\path (-45:1cm) coordinate (e3);
\path (-135:1cm) coordinate (e4);
\path (0,0.4) coordinate (m);
\path (0,0.2) coordinate (t1);
\draw[thick] (e1) to [bend left=40] (e4);
\draw[thick] (e2) to [bend right=40] (e3);
}
-2
\tik{0.5}{%
\path (0,0) coordinate (origin);
\path (135:0.5cm) coordinate (c1);
\path (45:0.5cm) coordinate (c2);
\path (-45: 0.5cm) coordinate (c3);
\path (-135: 0.5cm) coordinate (c4);
\path (135:1cm) coordinate (e1);
\path (45:1cm) coordinate (e2);
\path (-45:1cm) coordinate (e3);
\path (-135:1cm) coordinate (e4);
\path (0,0.4) coordinate (m);
\path (0,0.2) coordinate (t1);
\draw[thick] (e2) -- (e4);
\draw[cross] (e1) -- (e3);
}
+4
\tik{0.5}{%
\path (0,0) coordinate (origin);
\path (135:0.5cm) coordinate (c1);
\path (45:0.5cm) coordinate (c2);
\path (-45: 0.5cm) coordinate (c3);
\path (-135: 0.5cm) coordinate (c4);
\path (135:1cm) coordinate (e1);
\path (45:1cm) coordinate (e2);
\path (-45:1cm) coordinate (e3);
\path (-135:1cm) coordinate (e4);
\path (0,0.4) coordinate (m);
\path (0,0.2) coordinate (t1);
\path (0.3,-0.4) coordinate (t2);
\path (-0.3,-0.4) coordinate (t3);
\draw[thick] (e1) to [bend right=40] (e2);
\draw[thick] (e3) to [bend right=40] (e4);
}\ , 
\]
and applying relation $\gamma_2$ again gives
\[
\tik{0.5}{%
\path (0,0) coordinate (origin);
\path (135:0.5cm) coordinate (c1);
\path (45:0.5cm) coordinate (c2);
\path (-45: 0.5cm) coordinate (c3);
\path (-135: 0.5cm) coordinate (c4);
\path (135:1cm) coordinate (e1);
\path (45:1cm) coordinate (e2);
\path (-45:1cm) coordinate (e3);
\path (-135:1cm) coordinate (e4);
\draw[thick] (c1) -- (c2) -- (c3) -- (c4) -- cycle;
\draw[thick] (c1) -- (e1) (c2) -- (e2) (c3) -- (e3) (c4) -- (e4);
}
=
-2\left[\ \vrule width 0pt depth 15pt 
\tik{0.5}{%
\path (0,0) coordinate (origin);
\path (135:0.5cm) coordinate (c1);
\path (45:0.5cm) coordinate (c2);
\path (-45: 0.5cm) coordinate (c3);
\path (-135: 0.5cm) coordinate (c4);
\path (135:1cm) coordinate (e1);
\path (45:1cm) coordinate (e2);
\path (-45:1cm) coordinate (e3);
\path (-135:1cm) coordinate (e4);
\path (0,0.4) coordinate (m);
\path (0,0.2) coordinate (t1);
\path (0.3,-0.4) coordinate (t2);
\path (-0.3,-0.4) coordinate (t3);
\draw[thick] (e1) -- (m) -- (e2);
\draw[thick] (e3) -- (0,-0.4) -- (e4);
\draw[thick] (m) -- (0,-0.4);
}
+
\tik{0.5}{%
\path (0,0) coordinate (origin);
\path (135:0.5cm) coordinate (c1);
\path (45:0.5cm) coordinate (c2);
\path (-45: 0.5cm) coordinate (c3);
\path (-135: 0.5cm) coordinate (c4);
\path (135:1cm) coordinate (e1);
\path (45:1cm) coordinate (e2);
\path (-45:1cm) coordinate (e3);
\path (-135:1cm) coordinate (e4);
\path (0,0.4) coordinate (m);
\path (0,0.2) coordinate (t1);
\path (0.3,-0.4) coordinate (t2);
\path (-0.3,-0.4) coordinate (t3);
\draw[thick] (e1) -- (-0.4,0) -- (0.4,0) -- (e2);
\draw[thick] (e4) --  (-0.4,0) (e3) -- (0.4,0);
} 
\ \right]
+3\left[\ \vrule width 0pt depth 15pt
\tik{0.5}{%
\path (0,0) coordinate (origin);
\path (135:0.5cm) coordinate (c1);
\path (45:0.5cm) coordinate (c2);
\path (-45: 0.5cm) coordinate (c3);
\path (-135: 0.5cm) coordinate (c4);
\path (135:1cm) coordinate (e1);
\path (45:1cm) coordinate (e2);
\path (-45:1cm) coordinate (e3);
\path (-135:1cm) coordinate (e4);
\path (0,0.4) coordinate (m);
\path (0,0.2) coordinate (t1);
\draw[thick] (e1) to [bend left=40] (e4);
\draw[thick] (e2) to [bend right=40] (e3);
}
+
\tik{0.5}{%
\path (0,0) coordinate (origin);
\path (135:0.5cm) coordinate (c1);
\path (45:0.5cm) coordinate (c2);
\path (-45: 0.5cm) coordinate (c3);
\path (-135: 0.5cm) coordinate (c4);
\path (135:1cm) coordinate (e1);
\path (45:1cm) coordinate (e2);
\path (-45:1cm) coordinate (e3);
\path (-135:1cm) coordinate (e4);
\path (0,0.4) coordinate (m);
\path (0,0.2) coordinate (t1);
\path (0.3,-0.4) coordinate (t2);
\path (-0.3,-0.4) coordinate (t3);
\draw[thick] (e1) to [bend right=40] (e2);
\draw[thick] (e3) to [bend right=40] (e4);
}\ 
\right]
\]

\medskip
\item[\textbf{6.}] Finally, relation $\gamma_2$, together with the previous steps,  shows:
\[
\begin{split}
\tik{0.5}{%
\path (0,0) coordinate (origin);
\path (90:0.5cm) coordinate (c1);
\path (18:0.5cm) coordinate (c2);
\path (-54:0.5cm) coordinate (c3);
\path (-126:0.5cm) coordinate (c4);
\path (162:0.5cm) coordinate (c5);
\path (90:1cm) coordinate (e1);
\path (18:1cm) coordinate (e2);
\path (-54:1cm) coordinate (e3);
\path (-126:1cm) coordinate (e4);
\path (162:1cm) coordinate (e5);
\draw[thick] (c1) -- (c2) -- (c3) -- (c4) -- (c5) -- cycle;
\draw[thick] (c1) -- (e1) (c2) -- (e2) (c3) -- (e3) (c4) -- (e4) (c5) -- (e5);
}
&=-
\tik{0.5}{%
\path (0,0) coordinate (origin);
\path (90:0.5cm) coordinate (c1);
\path (18:0.5cm) coordinate (c2);
\path (-54:0.5cm) coordinate (c3);
\path (-126:0.5cm) coordinate (c4);
\path (162:0.5cm) coordinate (c5);
\path (90:1cm) coordinate (e1);
\path (18:1cm) coordinate (e2);
\path (-54:1cm) coordinate (e3);
\path (-126:1cm) coordinate (e4);
\path (162:1cm) coordinate (e5);
\path (126:0.6cm) coordinate (m1);
\path (126:0.2cm) coordinate (m2);
\draw[thick] (c2) -- (c3) -- (c4);
\draw[thick] (c2) -- (e2) (c3) -- (e3) (c4) -- (e4);
\draw[thick] (c2) -- (m2) -- (c4);
\draw[thick] (m2) -- (m1);
\draw[thick] (e1) -- (m1) -- (e5);
}
+2
\tik{0.5}{%
\path (0,0) coordinate (origin);
\path (90:0.5cm) coordinate (c1);
\path (18:0.5cm) coordinate (c2);
\path (-54:0.5cm) coordinate (c3);
\path (-126:0.5cm) coordinate (c4);
\path (162:0.5cm) coordinate (c5);
\path (90:1cm) coordinate (e1);
\path (18:1cm) coordinate (e2);
\path (-54:1cm) coordinate (e3);
\path (-126:1cm) coordinate (e4);
\path (162:1cm) coordinate (e5);
\draw[thick] (e5) .. controls (c1) and (c1) .. (c2);
\draw[cross] (e1) .. controls (c5) and (c5) .. (c4);
\draw[thick]  (c2) -- (c3) -- (c4);
\draw[thick] (c2) -- (e2) (c3) -- (e3) (c4) -- (e4);
}
-
\tik{0.5}{%
\path (0,0) coordinate (origin);
\path (90:0.5cm) coordinate (c1);
\path (18:0.5cm) coordinate (c2);
\path (-54:0.5cm) coordinate (c3);
\path (-126:0.5cm) coordinate (c4);
\path (162:0.5cm) coordinate (c5);
\path (90:1cm) coordinate (e1);
\path (18:1cm) coordinate (e2);
\path (-54:1cm) coordinate (e3);
\path (-126:1cm) coordinate (e4);
\path (162:1cm) coordinate (e5);
\draw[thick] (e1) .. controls (c1) and (c5) .. (e5);
\draw[thick] (c2) -- (c3) -- (c4);
\draw[thick] (c2) -- (e2) (c3) -- (e3) (c4) -- (e4) ;
\draw[thick] (c2) -- (c4);
}
-
\tik{0.5}{%
\path (0,0) coordinate (origin);
\path (90:0.5cm) coordinate (c1);
\path (18:0.5cm) coordinate (c2);
\path (-54:0.5cm) coordinate (c3);
\path (-126:0.5cm) coordinate (c4);
\path (162:0.5cm) coordinate (c5);
\path (90:1cm) coordinate (e1);
\path (18:1cm) coordinate (e2);
\path (-54:1cm) coordinate (e3);
\path (-126:1cm) coordinate (e4);
\path (162:1cm) coordinate (e5);
\draw[thick] (c2) .. controls (c1) and (c1) .. (e1);
\draw[thick] (c4) .. controls (c5) and (c5) .. (e5);
\draw[thick]  (c2) -- (c3) -- (c4) ;
\draw[thick]  (c2) -- (e2) (c3) -- (e3) (c4) -- (e4)  ;
}
\\[10pt]
&= 
2
\tik{0.5}{%
\path (0,0) coordinate (origin);
\path (90:0.5cm) coordinate (c1);
\path (18:0.5cm) coordinate (c2);
\path (-54:0.5cm) coordinate (c3);
\path (-126:0.5cm) coordinate (c4);
\path (162:0.5cm) coordinate (c5);
\path (90:1cm) coordinate (e1);
\path (18:1cm) coordinate (e2);
\path (-54:1cm) coordinate (e3);
\path (-126:1cm) coordinate (e4);
\path (162:1cm) coordinate (e5);
\draw[thick] (e5) .. controls (c5) and (c5) .. (c1);
\draw[thick] (e4) .. controls (c4) and (c4) .. (c3);
\draw[thick] (c1) -- (c2) -- (c3) ;
\draw[thick] (c1) -- (e1) (c2) -- (e2) (c3) -- (e3) ;
}
+2
\tik{0.5}{%
\path (0,0) coordinate (origin);
\path (90:0.5cm) coordinate (c1);
\path (18:0.5cm) coordinate (c2);
\path (-54:0.5cm) coordinate (c3);
\path (-126:0.5cm) coordinate (c4);
\path (162:0.5cm) coordinate (c5);
\path (90:1cm) coordinate (e1);
\path (18:1cm) coordinate (e2);
\path (-54:1cm) coordinate (e3);
\path (-126:1cm) coordinate (e4);
\path (162:1cm) coordinate (e5);
\draw[thick] (e1) .. controls (c1) and (c1) .. (c5);
\draw[thick] (e2) .. controls (c2) and (c2) .. (c3);
\draw[thick] (c5) -- (c4) -- (c3) ;
\draw[thick] (c5) -- (e5) (c4) -- (e4) (c3) -- (e3) ;
}
-
\tik{0.5}{%
\path (0,0) coordinate (origin);
\path (90:0.5cm) coordinate (c1);
\path (18:0.5cm) coordinate (c2);
\path (-54:0.5cm) coordinate (c3);
\path (-126:0.5cm) coordinate (c4);
\path (162:0.5cm) coordinate (c5);
\path (90:1cm) coordinate (e1);
\path (18:1cm) coordinate (e2);
\path (-54:1cm) coordinate (e3);
\path (-126:1cm) coordinate (e4);
\path (162:1cm) coordinate (e5);
\draw[thick] (c2) .. controls (c1) and (c1) .. (e1);
\draw[thick] (c4) .. controls (c5) and (c5) .. (e5);
\draw[thick]  (c2) -- (c3) -- (c4) ;
\draw[thick]  (c2) -- (e2) (c3) -- (e3) (c4) -- (e4) (c5) ;
}
\\
&\qquad -3\left[\ \vrule width 0pt depth 15pt
\tik{0.5}{%
\path (0,0) coordinate (origin);
\path (90:0.5cm) coordinate (c1);
\path (18:0.5cm) coordinate (c2);
\path (-54:0.5cm) coordinate (c3);
\path (-126:0.5cm) coordinate (c4);
\path (162:0.5cm) coordinate (c5);
\path (90:1cm) coordinate (e1);
\path (18:1cm) coordinate (e2);
\path (-54:1cm) coordinate (e3);
\path (-126:1cm) coordinate (e4);
\path (162:1cm) coordinate (e5);
\draw[thick] (e5) .. controls (c5) and (c5) .. (c1);
\draw[thick] (e2) .. controls (c2) and (c2) .. (c1);
\draw[thick]  (c3) -- (c4) ;
\draw[thick]  (c1) -- (e1) (c3) -- (e3) (c4) -- (e4)  ;
}
+
\tik{0.5}{%
\path (0,0) coordinate (origin);
\path (90:0.5cm) coordinate (c1);
\path (18:0.5cm) coordinate (c2);
\path (-54:0.5cm) coordinate (c3);
\path (-126:0.5cm) coordinate (c4);
\path (162:0.5cm) coordinate (c5);
\path (90:1cm) coordinate (e1);
\path (18:1cm) coordinate (e2);
\path (-54:1cm) coordinate (e3);
\path (-126:1cm) coordinate (e4);
\path (162:1cm) coordinate (e5);
\draw[thick] (e4) .. controls (c4) and (c4) .. (c5);
\draw[thick] (e1) .. controls (c1) and (c1) .. (c5);
\draw[thick]  (c2) -- (c3) ;
\draw[thick]  (c5) -- (e5) (c2) -- (e2) (c3) -- (e3)  ;
}
+
\tik{0.5}{%
\path (0,0) coordinate (origin);
\path (90:0.5cm) coordinate (c1);
\path (18:0.5cm) coordinate (c2);
\path (-54:0.5cm) coordinate (c3);
\path (-126:0.5cm) coordinate (c4);
\path (162:0.5cm) coordinate (c5);
\path (90:1cm) coordinate (e1);
\path (18:1cm) coordinate (e2);
\path (-54:1cm) coordinate (e3);
\path (-126:1cm) coordinate (e4);
\path (162:1cm) coordinate (e5);
\draw[thick] (e2) .. controls (c2) and (c2) .. (c3);
\draw[thick] (e4) .. controls (c4) and (c4) .. (c3);
\draw[thick]  (c5) -- (c1) ;
\draw[thick]  (c3) -- (e3) (c5) -- (e5) (c1) -- (e1)  ;
}
\right]
+2
\tik{0.5}{%
\path (0,0) coordinate (origin);
\path (90:0.5cm) coordinate (c1);
\path (18:0.5cm) coordinate (c2);
\path (-54:0.5cm) coordinate (c3);
\path (-126:0.5cm) coordinate (c4);
\path (162:0.5cm) coordinate (c5);
\path (90:1cm) coordinate (e1);
\path (18:1cm) coordinate (e2);
\path (-54:1cm) coordinate (e3);
\path (-126:1cm) coordinate (e4);
\path (162:1cm) coordinate (e5);
\draw[thick] (e5) .. controls (c1) and (c1) .. (c2);
\draw[cross] (e1) .. controls (c5) and (c5) .. (c4);
\draw[thick]  (c2) -- (c3) -- (c4);
\draw[thick] (c2) -- (e2) (c3) -- (e3) (c4) -- (e4);
}
\end{split}
\]
and
\[
\begin{split}
\tik{0.5}{%
\path (0,0) coordinate (origin);
\path (90:0.5cm) coordinate (c1);
\path (18:0.5cm) coordinate (c2);
\path (-54:0.5cm) coordinate (c3);
\path (-126:0.5cm) coordinate (c4);
\path (162:0.5cm) coordinate (c5);
\path (90:1cm) coordinate (e1);
\path (18:1cm) coordinate (e2);
\path (-54:1cm) coordinate (e3);
\path (-126:1cm) coordinate (e4);
\path (162:1cm) coordinate (e5);
\draw[thick] (e5) .. controls (c1) and (c1) .. (c2);
\draw[cross] (e1) .. controls (c5) and (c5) .. (c4);
\draw[thick]  (c2) -- (c3) -- (c4);
\draw[thick] (c2) -- (e2) (c3) -- (e3) (c4) -- (e4);
}
&=-
\tik{0.5}{%
\path (0,0) coordinate (origin);
\path (90:0.5cm) coordinate (c1);
\path (18:0.5cm) coordinate (c2);
\path (-54:0.5cm) coordinate (c3);
\path (-126:0.5cm) coordinate (c4);
\path (162:0.5cm) coordinate (c5);
\path (90:1cm) coordinate (e1);
\path (18:1cm) coordinate (e2);
\path (-54:1cm) coordinate (e3);
\path (-126:1cm) coordinate (e4);
\path (162:1cm) coordinate (e5);
\path (126:0.6cm) coordinate (m1);
\path (126:0.2cm) coordinate (m2);
\path (-18:0.6cm) coordinate (n1);
\draw[thick] (e5) .. controls (c1) and (c1) .. (origin);
\draw[cross] (e1) .. controls (c5) and (c5) .. (c4);
\draw[thick]  (e4) -- (c4) -- (origin) -- (n1) -- (e2);
\draw[thick] (e3) -- (n1);
}
+2
\tik{0.5}{%
\path (0,0) coordinate (origin);
\path (90:0.5cm) coordinate (c1);
\path (18:0.5cm) coordinate (c2);
\path (-54:0.5cm) coordinate (c3);
\path (-126:0.5cm) coordinate (c4);
\path (162:0.5cm) coordinate (c5);
\path (90:1cm) coordinate (e1);
\path (18:1cm) coordinate (e2);
\path (-54:1cm) coordinate (e3);
\path (-126:1cm) coordinate (e4);
\path (162:1cm) coordinate (e5);
\path (126:0.6cm) coordinate (m1);
\path (126:0.2cm) coordinate (m2);
\path (-18:0.6cm) coordinate (n1);
\draw[thick] (e1) .. controls (c5) and (c5) .. (c4);
\draw[thick] (e2) .. controls (c3) and (c3) .. (c4);
\draw[cross] (e5) to [bend left=60] (e3);
\draw[thick] (e4) -- (c4);
}
-
\tik{0.5}{%
\path (0,0) coordinate (origin);
\path (90:0.5cm) coordinate (c1);
\path (18:0.5cm) coordinate (c2);
\path (-54:0.5cm) coordinate (c3);
\path (-126:0.5cm) coordinate (c4);
\path (162:0.5cm) coordinate (c5);
\path (90:1cm) coordinate (e1);
\path (18:1cm) coordinate (e2);
\path (-54:1cm) coordinate (e3);
\path (-126:1cm) coordinate (e4);
\path (162:1cm) coordinate (e5);
\path (126:0.6cm) coordinate (m1);
\path (126:0.2cm) coordinate (m2);
\path (-18:0.6cm) coordinate (n1);
\draw[thick] (e1) .. controls (c5) and (c5) .. (c4);
\draw[cross] (e5) .. controls (c1) and (c3) .. (c4);
\draw[thick] (e2) to [bend right=60] (e3);
\draw[thick] (e4) -- (c4);
}
-
\tik{0.5}{%
\path (0,0) coordinate (origin);
\path (90:0.5cm) coordinate (c1);
\path (18:0.5cm) coordinate (c2);
\path (-54:0.5cm) coordinate (c3);
\path (-126:0.5cm) coordinate (c4);
\path (162:0.5cm) coordinate (c5);
\path (90:1cm) coordinate (e1);
\path (18:1cm) coordinate (e2);
\path (-54:1cm) coordinate (e3);
\path (-126:1cm) coordinate (e4);
\path (162:1cm) coordinate (e5);
\path (126:0.6cm) coordinate (m1);
\path (126:0.2cm) coordinate (m2);
\path (-90:0.6cm) coordinate (n1);
\draw[thick] (e1) -- (n1);
\draw[cross] (e5) to [bend right=30] (e2);
\draw[thick] (e4) -- (n1) -- (e3);
}
\\[10pt]
&=-
\tik{0.5}{%
\path (0,0) coordinate (origin);
\path (90:0.5cm) coordinate (c1);
\path (18:0.5cm) coordinate (c2);
\path (-54:0.5cm) coordinate (c3);
\path (-126:0.5cm) coordinate (c4);
\path (162:0.5cm) coordinate (c5);
\path (90:1cm) coordinate (e1);
\path (18:1cm) coordinate (e2);
\path (-54:1cm) coordinate (e3);
\path (-126:1cm) coordinate (e4);
\path (162:1cm) coordinate (e5);
\path (126:0.6cm) coordinate (m1);
\path (126:0.2cm) coordinate (m2);
\path (-18:0.6cm) coordinate (n1);
\draw[thick] (e5) .. controls (c1) and (c1) .. (origin);
\draw[cross] (e1) .. controls (c5) and (c5) .. (c4);
\draw[thick]  (e4) -- (c4) -- (origin) -- (n1) -- (e2);
\draw[thick] (e3) -- (n1);
}
+2
\tik{0.5}{%
\path (0,0) coordinate (origin);
\path (90:0.5cm) coordinate (c1);
\path (18:0.5cm) coordinate (c2);
\path (-54:0.5cm) coordinate (c3);
\path (-126:0.5cm) coordinate (c4);
\path (162:0.5cm) coordinate (c5);
\path (90:1cm) coordinate (e1);
\path (18:1cm) coordinate (e2);
\path (-54:1cm) coordinate (e3);
\path (-126:1cm) coordinate (e4);
\path (162:1cm) coordinate (e5);
\path (126:0.6cm) coordinate (m1);
\path (126:0.2cm) coordinate (m2);
\path (54:0.6cm) coordinate (n1);
\draw[thick] (e4) -- (n1);
\draw[cross] (e3) to [bend right=30] (e5);
\draw[thick] (e2) -- (n1) -- (e1);
}
+
\tik{0.5}{%
\path (0,0) coordinate (origin);
\path (90:0.5cm) coordinate (c1);
\path (18:0.5cm) coordinate (c2);
\path (-54:0.5cm) coordinate (c3);
\path (-126:0.5cm) coordinate (c4);
\path (162:0.5cm) coordinate (c5);
\path (90:1cm) coordinate (e1);
\path (18:1cm) coordinate (e2);
\path (-54:1cm) coordinate (e3);
\path (-126:1cm) coordinate (e4);
\path (162:1cm) coordinate (e5);
\draw[thick] (e4) .. controls (c4) and (c4) .. (c5);
\draw[thick] (e1) .. controls (c1) and (c1) .. (c5);
\draw[thick]  (c2) -- (c3) ;
\draw[thick]  (c5) -- (e5) (c2) -- (e2) (c3) -- (e3)  ;
}
-
\tik{0.5}{%
\path (0,0) coordinate (origin);
\path (90:0.5cm) coordinate (c1);
\path (18:0.5cm) coordinate (c2);
\path (-54:0.5cm) coordinate (c3);
\path (-126:0.5cm) coordinate (c4);
\path (162:0.5cm) coordinate (c5);
\path (90:1cm) coordinate (e1);
\path (18:1cm) coordinate (e2);
\path (-54:1cm) coordinate (e3);
\path (-126:1cm) coordinate (e4);
\path (162:1cm) coordinate (e5);
\path (126:0.6cm) coordinate (m1);
\path (126:0.2cm) coordinate (m2);
\path (-90:0.6cm) coordinate (n1);
\draw[thick] (e1) -- (n1);
\draw[cross] (e5) to [bend right=30] (e2);
\draw[thick] (e4) -- (n1) -- (e3);
}\ ,
\\[15pt]
\tik{0.5}{%
\path (0,0) coordinate (origin);
\path (90:0.5cm) coordinate (c1);
\path (18:0.5cm) coordinate (c2);
\path (-54:0.5cm) coordinate (c3);
\path (-126:0.5cm) coordinate (c4);
\path (162:0.5cm) coordinate (c5);
\path (90:1cm) coordinate (e1);
\path (18:1cm) coordinate (e2);
\path (-54:1cm) coordinate (e3);
\path (-126:1cm) coordinate (e4);
\path (162:1cm) coordinate (e5);
\path (126:0.6cm) coordinate (m1);
\path (126:0.2cm) coordinate (m2);
\path (-18:0.6cm) coordinate (n1);
\draw[thick] (e5) .. controls (c1) and (c1) .. (origin);
\draw[cross] (e1) .. controls (c5) and (c5) .. (c4);
\draw[thick]  (e4) -- (c4) -- (origin) -- (n1) -- (e2);
\draw[thick] (e3) -- (n1);
}
&= -
\tik{0.5}{%
\path (0,0) coordinate (origin);
\path (90:0.5cm) coordinate (c1);
\path (18:0.5cm) coordinate (c2);
\path (-54:0.5cm) coordinate (c3);
\path (-126:0.5cm) coordinate (c4);
\path (162:0.5cm) coordinate (c5);
\path (90:1cm) coordinate (e1);
\path (18:1cm) coordinate (e2);
\path (-54:1cm) coordinate (e3);
\path (-126:1cm) coordinate (e4);
\path (162:1cm) coordinate (e5);
\path (162:0.4cm) coordinate (m1);
\path (126:0.2cm) coordinate (m2);
\path (-162:0.4cm) coordinate (n1);
\draw[thick] (e1) .. controls (m1) and (e5) .. (n1);
\draw[cross] (e5) .. controls (c1) and (origin) .. (n1);
\draw[thick] (e2) .. controls (c2) and (c2) .. (c3);
\draw[thick]  (e3) -- (c3) -- (c4) -- (e4);
\draw[thick] (c4) -- (n1);
}
+2
\tik{0.5}{%
\path (0,0) coordinate (origin);
\path (90:0.5cm) coordinate (c1);
\path (18:0.5cm) coordinate (c2);
\path (-54:0.5cm) coordinate (c3);
\path (-126:0.5cm) coordinate (c4);
\path (162:0.5cm) coordinate (c5);
\path (90:1cm) coordinate (e1);
\path (18:1cm) coordinate (e2);
\path (-54:1cm) coordinate (e3);
\path (-126:1cm) coordinate (e4);
\path (162:1cm) coordinate (e5);
\path (-162:0.9cm) coordinate (m1);
\path (126:0.2cm) coordinate (m2);
\path (-18:0.6cm) coordinate (n1);
\draw[thick] (e1) .. controls (m1) and (m1) .. (n1);
\draw[cross] (e5) .. controls (origin) and (origin) .. (e4);
\draw[thick] (e2) -- (n1) -- (e3);
}
-
\tik{0.5}{%
\path (0,0) coordinate (origin);
\path (90:0.5cm) coordinate (c1);
\path (18:0.5cm) coordinate (c2);
\path (-54:0.5cm) coordinate (c3);
\path (-126:0.5cm) coordinate (c4);
\path (162:0.5cm) coordinate (c5);
\path (90:1cm) coordinate (e1);
\path (18:1cm) coordinate (e2);
\path (-54:1cm) coordinate (e3);
\path (-126:1cm) coordinate (e4);
\path (162:1cm) coordinate (e5);
\path (126:0.6cm) coordinate (m1);
\path (126:0.2cm) coordinate (m2);
\path (-18:0.6cm) coordinate (n1);
\draw[thick] (e5) -- (n1);
\draw[cross] (e4) to [bend right=30] (e1);
\draw[thick] (e3) -- (n1) -- (e2);
}
-
\tik{0.5}{%
\path (0,0) coordinate (origin);
\path (90:0.5cm) coordinate (c1);
\path (18:0.5cm) coordinate (c2);
\path (-54:0.5cm) coordinate (c3);
\path (-126:0.5cm) coordinate (c4);
\path (162:0.5cm) coordinate (c5);
\path (90:1cm) coordinate (e1);
\path (18:1cm) coordinate (e2);
\path (-54:1cm) coordinate (e3);
\path (-126:1cm) coordinate (e4);
\path (162:1cm) coordinate (e5);
\path (126:0.6cm) coordinate (m1);
\path (126:0.2cm) coordinate (m2);
\path (-90:0.6cm) coordinate (n1);
\draw[thick] (e1) .. controls (-0.5,0) and (-0.2,-0.2) .. (origin);
\draw[cross] (e5) .. controls (0,0.5) and (0.2,0.2) .. (origin);
\draw[thick] (e4) .. controls (c4) and (c4) .. (c3);
\draw[thick] (e2) .. controls (c2) and (c2) .. (c3);
\draw[thick] (e3) -- (c3);
}
\\[10pt]
&=
\tik{0.5}{%
\path (0,0) coordinate (origin);
\path (90:0.5cm) coordinate (c1);
\path (18:0.5cm) coordinate (c2);
\path (-54:0.5cm) coordinate (c3);
\path (-126:0.5cm) coordinate (c4);
\path (162:0.5cm) coordinate (c5);
\path (90:1cm) coordinate (e1);
\path (18:1cm) coordinate (e2);
\path (-54:1cm) coordinate (e3);
\path (-126:1cm) coordinate (e4);
\path (162:1cm) coordinate (e5);
\draw[thick] (e1) .. controls (c1) and (c1) .. (c5);
\draw[thick] (e2) .. controls (c2) and (c2) .. (c3);
\draw[thick] (c5) -- (c4) -- (c3) ;
\draw[thick] (c5) -- (e5) (c4) -- (e4) (c3) -- (e3) ;
}
+2
\tik{0.5}{%
\path (0,0) coordinate (origin);
\path (90:0.5cm) coordinate (c1);
\path (18:0.5cm) coordinate (c2);
\path (-54:0.5cm) coordinate (c3);
\path (-126:0.5cm) coordinate (c4);
\path (162:0.5cm) coordinate (c5);
\path (90:1cm) coordinate (e1);
\path (18:1cm) coordinate (e2);
\path (-54:1cm) coordinate (e3);
\path (-126:1cm) coordinate (e4);
\path (162:1cm) coordinate (e5);
\draw[thick] (e1) .. controls (c1) and (c1) .. (c2);
\draw[thick] (e3) .. controls (c3) and (c3) .. (c2);
\draw[thick]  (c4) -- (c5) ;
\draw[thick]  (c2) -- (e2) (c4) -- (e4) (c5) -- (e5)  ;
}
-
\tik{0.5}{%
\path (0,0) coordinate (origin);
\path (90:0.5cm) coordinate (c1);
\path (18:0.5cm) coordinate (c2);
\path (-54:0.5cm) coordinate (c3);
\path (-126:0.5cm) coordinate (c4);
\path (162:0.5cm) coordinate (c5);
\path (90:1cm) coordinate (e1);
\path (18:1cm) coordinate (e2);
\path (-54:1cm) coordinate (e3);
\path (-126:1cm) coordinate (e4);
\path (162:1cm) coordinate (e5);
\path (126:0.6cm) coordinate (m1);
\path (126:0.2cm) coordinate (m2);
\path (-18:0.6cm) coordinate (n1);
\draw[thick] (e5) -- (n1);
\draw[cross] (e4) to [bend right=30] (e1);
\draw[thick] (e3) -- (n1) -- (e2);
}
-
\tik{0.5}{%
\path (0,0) coordinate (origin);
\path (90:0.5cm) coordinate (c1);
\path (18:0.5cm) coordinate (c2);
\path (-54:0.5cm) coordinate (c3);
\path (-126:0.5cm) coordinate (c4);
\path (162:0.5cm) coordinate (c5);
\path (90:1cm) coordinate (e1);
\path (18:1cm) coordinate (e2);
\path (-54:1cm) coordinate (e3);
\path (-126:1cm) coordinate (e4);
\path (162:1cm) coordinate (e5);
\draw[thick] (e2) .. controls (c2) and (c2) .. (c3);
\draw[thick] (e4) .. controls (c4) and (c4) .. (c3);
\draw[thick]  (c5) -- (c1) ;
\draw[thick]  (c3) -- (e3) (c5) -- (e5) (c1) -- (e1)  ;
}
\end{split}.
\]
Hence,
\[
\begin{split}
\tik{0.5}{%
\path (0,0) coordinate (origin);
\path (90:0.5cm) coordinate (c1);
\path (18:0.5cm) coordinate (c2);
\path (-54:0.5cm) coordinate (c3);
\path (-126:0.5cm) coordinate (c4);
\path (162:0.5cm) coordinate (c5);
\path (90:1cm) coordinate (e1);
\path (18:1cm) coordinate (e2);
\path (-54:1cm) coordinate (e3);
\path (-126:1cm) coordinate (e4);
\path (162:1cm) coordinate (e5);
\draw[thick] (e5) .. controls (c1) and (c1) .. (c2);
\draw[cross] (e1) .. controls (c5) and (c5) .. (c4);
\draw[thick]  (c2) -- (c3) -- (c4);
\draw[thick] (c2) -- (e2) (c3) -- (e3) (c4) -- (e4);
}
&=
\tik{0.5}{%
\path (0,0) coordinate (origin);
\path (90:0.5cm) coordinate (c1);
\path (18:0.5cm) coordinate (c2);
\path (-54:0.5cm) coordinate (c3);
\path (-126:0.5cm) coordinate (c4);
\path (162:0.5cm) coordinate (c5);
\path (90:1cm) coordinate (e1);
\path (18:1cm) coordinate (e2);
\path (-54:1cm) coordinate (e3);
\path (-126:1cm) coordinate (e4);
\path (162:1cm) coordinate (e5);
\draw[thick] (e1) .. controls (c1) and (c1) .. (c5);
\draw[thick] (e2) .. controls (c2) and (c2) .. (c3);
\draw[thick] (c5) -- (c4) -- (c3) ;
\draw[thick] (c5) -- (e5) (c4) -- (e4) (c3) -- (e3) ;
}
+
\tik{0.5}{%
\path (0,0) coordinate (origin);
\path (90:0.5cm) coordinate (c1);
\path (18:0.5cm) coordinate (c2);
\path (-54:0.5cm) coordinate (c3);
\path (-126:0.5cm) coordinate (c4);
\path (162:0.5cm) coordinate (c5);
\path (90:1cm) coordinate (e1);
\path (18:1cm) coordinate (e2);
\path (-54:1cm) coordinate (e3);
\path (-126:1cm) coordinate (e4);
\path (162:1cm) coordinate (e5);
\draw[thick] (e4) .. controls (c4) and (c4) .. (c5);
\draw[thick] (e1) .. controls (c1) and (c1) .. (c5);
\draw[thick]  (c2) -- (c3) ;
\draw[thick]  (c5) -- (e5) (c2) -- (e2) (c3) -- (e3)  ;
}
-2
\tik{0.5}{%
\path (0,0) coordinate (origin);
\path (90:0.5cm) coordinate (c1);
\path (18:0.5cm) coordinate (c2);
\path (-54:0.5cm) coordinate (c3);
\path (-126:0.5cm) coordinate (c4);
\path (162:0.5cm) coordinate (c5);
\path (90:1cm) coordinate (e1);
\path (18:1cm) coordinate (e2);
\path (-54:1cm) coordinate (e3);
\path (-126:1cm) coordinate (e4);
\path (162:1cm) coordinate (e5);
\draw[thick] (e1) .. controls (c1) and (c1) .. (c2);
\draw[thick] (e3) .. controls (c3) and (c3) .. (c2);
\draw[thick]  (c4) -- (c5) ;
\draw[thick]  (c2) -- (e2) (c4) -- (e4) (c5) -- (e5)  ;
}
+
\tik{0.5}{%
\path (0,0) coordinate (origin);
\path (90:0.5cm) coordinate (c1);
\path (18:0.5cm) coordinate (c2);
\path (-54:0.5cm) coordinate (c3);
\path (-126:0.5cm) coordinate (c4);
\path (162:0.5cm) coordinate (c5);
\path (90:1cm) coordinate (e1);
\path (18:1cm) coordinate (e2);
\path (-54:1cm) coordinate (e3);
\path (-126:1cm) coordinate (e4);
\path (162:1cm) coordinate (e5);
\draw[thick] (e2) .. controls (c2) and (c2) .. (c3);
\draw[thick] (e4) .. controls (c4) and (c4) .. (c3);
\draw[thick]  (c5) -- (c1) ;
\draw[thick]  (c3) -- (e3) (c5) -- (e5) (c1) -- (e1)  ;
}
\\
&\hspace*{1in}
+2
\tik{0.5}{%
\path (0,0) coordinate (origin);
\path (90:0.5cm) coordinate (c1);
\path (18:0.5cm) coordinate (c2);
\path (-54:0.5cm) coordinate (c3);
\path (-126:0.5cm) coordinate (c4);
\path (162:0.5cm) coordinate (c5);
\path (90:1cm) coordinate (e1);
\path (18:1cm) coordinate (e2);
\path (-54:1cm) coordinate (e3);
\path (-126:1cm) coordinate (e4);
\path (162:1cm) coordinate (e5);
\path (126:0.6cm) coordinate (m1);
\path (126:0.2cm) coordinate (m2);
\path (56:0.6cm) coordinate (n1);
\draw[thick] (e4) -- (n1);
\draw[cross] (e3) to [bend right=30] (e5);
\draw[thick] (e2) -- (n1) -- (e1);
}
-
\tik{0.5}{%
\path (0,0) coordinate (origin);
\path (90:0.5cm) coordinate (c1);
\path (18:0.5cm) coordinate (c2);
\path (-54:0.5cm) coordinate (c3);
\path (-126:0.5cm) coordinate (c4);
\path (162:0.5cm) coordinate (c5);
\path (90:1cm) coordinate (e1);
\path (18:1cm) coordinate (e2);
\path (-54:1cm) coordinate (e3);
\path (-126:1cm) coordinate (e4);
\path (162:1cm) coordinate (e5);
\path (126:0.6cm) coordinate (m1);
\path (126:0.2cm) coordinate (m2);
\path (-90:0.6cm) coordinate (n1);
\draw[thick] (e1) -- (n1);
\draw[cross] (e5) to [bend right=30] (e2);
\draw[thick] (e4) -- (n1) -- (e3);
}
+
\tik{0.5}{%
\path (0,0) coordinate (origin);
\path (90:0.5cm) coordinate (c1);
\path (18:0.5cm) coordinate (c2);
\path (-54:0.5cm) coordinate (c3);
\path (-126:0.5cm) coordinate (c4);
\path (162:0.5cm) coordinate (c5);
\path (90:1cm) coordinate (e1);
\path (18:1cm) coordinate (e2);
\path (-54:1cm) coordinate (e3);
\path (-126:1cm) coordinate (e4);
\path (162:1cm) coordinate (e5);
\path (126:0.6cm) coordinate (m1);
\path (126:0.2cm) coordinate (m2);
\path (-18:0.6cm) coordinate (n1);
\draw[thick] (e5) -- (n1);
\draw[cross] (e4) to [bend right=30] (e1);
\draw[thick] (e3) -- (n1) -- (e2);
}
\end{split}
\]
and consequently,
\[
\begin{split}
\tik{0.5}{%
\path (0,0) coordinate (origin);
\path (90:0.5cm) coordinate (c1);
\path (18:0.5cm) coordinate (c2);
\path (-54:0.5cm) coordinate (c3);
\path (-126:0.5cm) coordinate (c4);
\path (162:0.5cm) coordinate (c5);
\path (90:1cm) coordinate (e1);
\path (18:1cm) coordinate (e2);
\path (-54:1cm) coordinate (e3);
\path (-126:1cm) coordinate (e4);
\path (162:1cm) coordinate (e5);
\draw[thick] (c1) -- (c2) -- (c3) -- (c4) -- (c5) -- cycle;
\draw[thick] (c1) -- (e1) (c2) -- (e2) (c3) -- (e3) (c4) -- (e4) (c5) -- (e5);
\node at (152:1cm) {\tiny $1$};
\node at (80:1cm) {\tiny $2$};
\node at (28:1cm) {\tiny $3$};
\node at (-42:1cm) {\tiny $4$};
\node at (-136:1cm) {\tiny $5$};
}
&=
2
\tik{0.5}{%
\path (0,0) coordinate (origin);
\path (90:0.5cm) coordinate (c1);
\path (18:0.5cm) coordinate (c2);
\path (-54:0.5cm) coordinate (c3);
\path (-126:0.5cm) coordinate (c4);
\path (162:0.5cm) coordinate (c5);
\path (90:1cm) coordinate (e1);
\path (18:1cm) coordinate (e2);
\path (-54:1cm) coordinate (e3);
\path (-126:1cm) coordinate (e4);
\path (162:1cm) coordinate (e5);
\draw[thick] (e4) .. controls (c4) and (c4) .. (c3);
\draw[thick] (e5) .. controls (c5) and (c5) .. (c1);
\draw[thick] (c3) -- (c2) -- (c1) ;
\draw[thick] (c3) -- (e3) (c2) -- (e2) (c1) -- (e1) ;
}
-
\tik{0.5}{%
\path (0,0) coordinate (origin);
\path (90:0.5cm) coordinate (c1);
\path (18:0.5cm) coordinate (c2);
\path (-54:0.5cm) coordinate (c3);
\path (-126:0.5cm) coordinate (c4);
\path (162:0.5cm) coordinate (c5);
\path (90:1cm) coordinate (e1);
\path (18:1cm) coordinate (e2);
\path (-54:1cm) coordinate (e3);
\path (-126:1cm) coordinate (e4);
\path (162:1cm) coordinate (e5);
\draw[thick] (e5) .. controls (c5) and (c5) .. (c4);
\draw[thick] (e1) .. controls (c1) and (c1) .. (c2);
\draw[thick] (c4) -- (c3) -- (c2) ;
\draw[thick] (c4) -- (e4) (c3) -- (e3) (c2) -- (e2) ;
}
-3
\tik{0.5}{%
\path (0,0) coordinate (origin);
\path (90:0.5cm) coordinate (c1);
\path (18:0.5cm) coordinate (c2);
\path (-54:0.5cm) coordinate (c3);
\path (-126:0.5cm) coordinate (c4);
\path (162:0.5cm) coordinate (c5);
\path (90:1cm) coordinate (e1);
\path (18:1cm) coordinate (e2);
\path (-54:1cm) coordinate (e3);
\path (-126:1cm) coordinate (e4);
\path (162:1cm) coordinate (e5);
\draw[thick] (e5) .. controls (c5) and (c5) .. (c1);
\draw[thick] (e2) .. controls (c2) and (c2) .. (c1);
\draw[thick]  (c3) -- (c4) ;
\draw[thick]  (c1) -- (e1) (c3) -- (e3) (c4) -- (e4)  ;
}
-4
\tik{0.5}{%
\path (0,0) coordinate (origin);
\path (90:0.5cm) coordinate (c1);
\path (18:0.5cm) coordinate (c2);
\path (-54:0.5cm) coordinate (c3);
\path (-126:0.5cm) coordinate (c4);
\path (162:0.5cm) coordinate (c5);
\path (90:1cm) coordinate (e1);
\path (18:1cm) coordinate (e2);
\path (-54:1cm) coordinate (e3);
\path (-126:1cm) coordinate (e4);
\path (162:1cm) coordinate (e5);
\draw[thick] (e1) .. controls (c1) and (c1) .. (c2);
\draw[thick] (e3) .. controls (c3) and (c3) .. (c2);
\draw[thick]  (c4) -- (c5) ;
\draw[thick]  (c2) -- (e2) (c4) -- (e4) (c5) -- (e5)  ;
}
-
\tik{0.5}{%
\path (0,0) coordinate (origin);
\path (90:0.5cm) coordinate (c1);
\path (18:0.5cm) coordinate (c2);
\path (-54:0.5cm) coordinate (c3);
\path (-126:0.5cm) coordinate (c4);
\path (162:0.5cm) coordinate (c5);
\path (90:1cm) coordinate (e1);
\path (18:1cm) coordinate (e2);
\path (-54:1cm) coordinate (e3);
\path (-126:1cm) coordinate (e4);
\path (162:1cm) coordinate (e5);
\draw[thick] (e4) .. controls (c4) and (c4) .. (c5);
\draw[thick] (e1) .. controls (c1) and (c1) .. (c5);
\draw[thick]  (c2) -- (c3) ;
\draw[thick]  (c5) -- (e5) (c2) -- (e2) (c3) -- (e3)  ;
}
-
\tik{0.5}{%
\path (0,0) coordinate (origin);
\path (90:0.5cm) coordinate (c1);
\path (18:0.5cm) coordinate (c2);
\path (-54:0.5cm) coordinate (c3);
\path (-126:0.5cm) coordinate (c4);
\path (162:0.5cm) coordinate (c5);
\path (90:1cm) coordinate (e1);
\path (18:1cm) coordinate (e2);
\path (-54:1cm) coordinate (e3);
\path (-126:1cm) coordinate (e4);
\path (162:1cm) coordinate (e5);
\draw[thick] (e2) .. controls (c2) and (c2) .. (c3);
\draw[thick] (e4) .. controls (c4) and (c4) .. (c3);
\draw[thick]  (c5) -- (c1) ;
\draw[thick]  (c3) -- (e3) (c5) -- (e5) (c1) -- (e1)  ;
}
\\
&\hspace*{1in}
+2
\tik{0.5}{%
\path (0,0) coordinate (origin);
\path (90:0.5cm) coordinate (c1);
\path (18:0.5cm) coordinate (c2);
\path (-54:0.5cm) coordinate (c3);
\path (-126:0.5cm) coordinate (c4);
\path (162:0.5cm) coordinate (c5);
\path (90:1cm) coordinate (e1);
\path (18:1cm) coordinate (e2);
\path (-54:1cm) coordinate (e3);
\path (-126:1cm) coordinate (e4);
\path (162:1cm) coordinate (e5);
\path (126:0.6cm) coordinate (m1);
\path (126:0.2cm) coordinate (m2);
\path (-18:0.6cm) coordinate (n1);
\draw[thick] (e5) -- (n1);
\draw[cross] (e4) to [bend right=30] (e1);
\draw[thick] (e3) -- (n1) -- (e2);
}
-2
\tik{0.5}{%
\path (0,0) coordinate (origin);
\path (90:0.5cm) coordinate (c1);
\path (18:0.5cm) coordinate (c2);
\path (-54:0.5cm) coordinate (c3);
\path (-126:0.5cm) coordinate (c4);
\path (162:0.5cm) coordinate (c5);
\path (90:1cm) coordinate (e1);
\path (18:1cm) coordinate (e2);
\path (-54:1cm) coordinate (e3);
\path (-126:1cm) coordinate (e4);
\path (162:1cm) coordinate (e5);
\path (126:0.6cm) coordinate (m1);
\path (126:0.2cm) coordinate (m2);
\path (-90:0.6cm) coordinate (n1);
\draw[thick] (e1) -- (n1);
\draw[cross] (e5) to [bend right=30] (e2);
\draw[thick] (e4) -- (n1) -- (e3);
}
+4
\tik{0.5}{%
\path (0,0) coordinate (origin);
\path (90:0.5cm) coordinate (c1);
\path (18:0.5cm) coordinate (c2);
\path (-54:0.5cm) coordinate (c3);
\path (-126:0.5cm) coordinate (c4);
\path (162:0.5cm) coordinate (c5);
\path (90:1cm) coordinate (e1);
\path (18:1cm) coordinate (e2);
\path (-54:1cm) coordinate (e3);
\path (-126:1cm) coordinate (e4);
\path (162:1cm) coordinate (e5);
\path (126:0.6cm) coordinate (m1);
\path (126:0.2cm) coordinate (m2);
\path (56:0.6cm) coordinate (n1);
\draw[thick] (e4) -- (n1);
\draw[cross] (e3) to [bend right=30] (e5);
\draw[thick] (e2) -- (n1) -- (e1);
}
\\
&= 2a_3-a_4-3b_2-4b_3-b_1-b_4+2c_1-2c_2+4c_5,
\end{split}
\]
where
\[
a_1=
\tik{0.5}{%
\path (0,0) coordinate (origin);
\path (90:0.5cm) coordinate (c1);
\path (18:0.5cm) coordinate (c2);
\path (-54:0.5cm) coordinate (c3);
\path (-126:0.5cm) coordinate (c4);
\path (162:0.5cm) coordinate (c5);
\path (90:1cm) coordinate (e1);
\path (18:1cm) coordinate (e2);
\path (-54:1cm) coordinate (e3);
\path (-126:1cm) coordinate (e4);
\path (162:1cm) coordinate (e5);
\draw[thick] (e2) .. controls (c2) and (c2) .. (c1);
\draw[thick] (e3) .. controls (c3) and (c3) .. (c4);
\draw[thick] (c1) -- (c5) -- (c4) ;
\draw[thick] (c1) -- (e1) (c5) -- (e5) (c4) -- (e4) ;
\node at (150:1cm) {\tiny $1$};
}
,
\qquad
b_1=
\tik{0.5}{%
\path (0,0) coordinate (origin);
\path (90:0.5cm) coordinate (c1);
\path (18:0.5cm) coordinate (c2);
\path (-54:0.5cm) coordinate (c3);
\path (-126:0.5cm) coordinate (c4);
\path (162:0.5cm) coordinate (c5);
\path (90:1cm) coordinate (e1);
\path (18:1cm) coordinate (e2);
\path (-54:1cm) coordinate (e3);
\path (-126:1cm) coordinate (e4);
\path (162:1cm) coordinate (e5);
\draw[thick] (e4) .. controls (c4) and (c4) .. (c5);
\draw[thick] (e1) .. controls (c1) and (c1) .. (c5);
\draw[thick]  (c2) -- (c3) ;
\draw[thick]  (c5) -- (e5) (c2) -- (e2) (c3) -- (e3)  ;
\node at (150:1cm) {\tiny $1$};
}
,
\qquad
c_1=
\tik{0.5}{%
\path (0,0) coordinate (origin);
\path (90:0.5cm) coordinate (c1);
\path (18:0.5cm) coordinate (c2);
\path (-54:0.5cm) coordinate (c3);
\path (-126:0.5cm) coordinate (c4);
\path (162:0.5cm) coordinate (c5);
\path (90:1cm) coordinate (e1);
\path (18:1cm) coordinate (e2);
\path (-54:1cm) coordinate (e3);
\path (-126:1cm) coordinate (e4);
\path (162:1cm) coordinate (e5);
\path (126:0.6cm) coordinate (m1);
\path (126:0.2cm) coordinate (m2);
\path (-18:0.6cm) coordinate (n1);
\draw[thick] (e5) -- (n1);
\draw[cross] (e4) to [bend right=30] (e1);
\draw[thick] (e3) -- (n1) -- (e2);
\node at (150:1cm) {\tiny $1$};
}
,
\]
and the rest are cyclic permutations of them. But from relation $\gamma_2$,  we obtain
\[
2
\tik{0.5}{%
\path (0,0) coordinate (origin);
\path (90:0.5cm) coordinate (c1);
\path (18:0.5cm) coordinate (c2);
\path (-54:0.5cm) coordinate (c3);
\path (-126:0.5cm) coordinate (c4);
\path (162:0.5cm) coordinate (c5);
\path (90:1cm) coordinate (e1);
\path (18:1cm) coordinate (e2);
\path (-54:1cm) coordinate (e3);
\path (-126:1cm) coordinate (e4);
\path (162:1cm) coordinate (e5);
\path (126:0.6cm) coordinate (m1);
\path (126:0.2cm) coordinate (m2);
\path (-18:0.6cm) coordinate (n1);
\draw[thick] (e5) -- (n1);
\draw[cross] (e4) to [bend right=30] (e1);
\draw[thick] (e3) -- (n1) -- (e2);
%
}
=
\tik{0.5}{%
\path (0,0) coordinate (origin);
\path (90:0.5cm) coordinate (c1);
\path (18:0.5cm) coordinate (c2);
\path (-54:0.5cm) coordinate (c3);
\path (-126:0.5cm) coordinate (c4);
\path (162:0.5cm) coordinate (c5);
\path (90:1cm) coordinate (e1);
\path (18:1cm) coordinate (e2);
\path (-54:1cm) coordinate (e3);
\path (-126:1cm) coordinate (e4);
\path (162:1cm) coordinate (e5);
\draw[thick] (e1) .. controls (c1) and (c1) .. (c5);
\draw[thick] (e2) .. controls (c2) and (c2) .. (c3);
\draw[thick] (c5) -- (c4) -- (c3) ;
\draw[thick] (c5) -- (e5) (c4) -- (e4) (c3) -- (e3) ;
%
}
+
\tik{0.5}{%
\path (0,0) coordinate (origin);
\path (90:0.5cm) coordinate (c1);
\path (18:0.5cm) coordinate (c2);
\path (-54:0.5cm) coordinate (c3);
\path (-126:0.5cm) coordinate (c4);
\path (162:0.5cm) coordinate (c5);
\path (90:1cm) coordinate (e1);
\path (18:1cm) coordinate (e2);
\path (-54:1cm) coordinate (e3);
\path (-126:1cm) coordinate (e4);
\path (162:1cm) coordinate (e5);
\draw[thick] (e3) .. controls (c3) and (c3) .. (c2);
\draw[thick] (e4) .. controls (c4) and (c4) .. (c5);
\draw[thick] (c2) -- (c1) -- (c5) ;
\draw[thick] (c2) -- (e2) (c1) -- (e1) (c5) -- (e5) ;
%
}
+
\tik{0.5}{%
\path (0,0) coordinate (origin);
\path (90:0.5cm) coordinate (c1);
\path (18:0.5cm) coordinate (c2);
\path (-54:0.5cm) coordinate (c3);
\path (-126:0.5cm) coordinate (c4);
\path (162:0.5cm) coordinate (c5);
\path (90:1cm) coordinate (e1);
\path (18:1cm) coordinate (e2);
\path (-54:1cm) coordinate (e3);
\path (-126:1cm) coordinate (e4);
\path (162:1cm) coordinate (e5);
\draw[thick] (e2) .. controls (c2) and (c2) .. (c3);
\draw[thick] (e4) .. controls (c4) and (c4) .. (c3);
\draw[thick]  (c5) -- (c1) ;
\draw[thick]  (c3) -- (e3) (c5) -- (e5) (c1) -- (e1)  ;
%
}
+
\tik{0.5}{%
\path (0,0) coordinate (origin);
\path (90:0.5cm) coordinate (c1);
\path (18:0.5cm) coordinate (c2);
\path (-54:0.5cm) coordinate (c3);
\path (-126:0.5cm) coordinate (c4);
\path (162:0.5cm) coordinate (c5);
\path (90:1cm) coordinate (e1);
\path (18:1cm) coordinate (e2);
\path (-54:1cm) coordinate (e3);
\path (-126:1cm) coordinate (e4);
\path (162:1cm) coordinate (e5);
\draw[thick] (e1) .. controls (c1) and (c1) .. (c2);
\draw[thick] (e3) .. controls (c3) and (c3) .. (c2);
\draw[thick]  (c4) -- (c5) ;
\draw[thick]  (c2) -- (e2) (c4) -- (e4) (c5) -- (e5)  ;
%
}
\]
that is:
\[
2c_1=a_5+a_2+b_4+b_3.
\]
Permuting cyclically,  we get:
\[
\begin{split}
2c_1&=a_2+a_5+b_3+b_4,\\
-2c_2&=-a_3-a_1-b_4-b_5,\\
4c_5&=2a_1+2a_4+2b_2+2b_3,
\end{split}
\]
and, therefore,
\[
\begin{split}
\tik{0.5}{%
\path (0,0) coordinate (origin);
\path (90:0.5cm) coordinate (c1);
\path (18:0.5cm) coordinate (c2);
\path (-54:0.5cm) coordinate (c3);
\path (-126:0.5cm) coordinate (c4);
\path (162:0.5cm) coordinate (c5);
\path (90:1cm) coordinate (e1);
\path (18:1cm) coordinate (e2);
\path (-54:1cm) coordinate (e3);
\path (-126:1cm) coordinate (e4);
\path (162:1cm) coordinate (e5);
\draw[thick] (c1) -- (c2) -- (c3) -- (c4) -- (c5) -- cycle;
\draw[thick] (c1) -- (e1) (c2) -- (e2) (c3) -- (e3) (c4) -- (e4) (c5) -- (e5);
}
&= 2a_3-a_4-3b_2-4b_3-b_1-b_4+2c_1-2c_2+4c_5\\
&=\left[ a_1+a_2+a_3+a_4+a_5\right] -\left[ b_1+b_2+b_3+b_4+b_5\right],
\end{split}
\]
that is,
\[
\begin{split}
\tik{0.5}{%
\path (0,0) coordinate (origin);
\path (90:0.5cm) coordinate (c1);
\path (18:0.5cm) coordinate (c2);
\path (-54:0.5cm) coordinate (c3);
\path (-126:0.5cm) coordinate (c4);
\path (162:0.5cm) coordinate (c5);
\path (90:1cm) coordinate (e1);
\path (18:1cm) coordinate (e2);
\path (-54:1cm) coordinate (e3);
\path (-126:1cm) coordinate (e4);
\path (162:1cm) coordinate (e5);
\draw[thick] (c1) -- (c2) -- (c3) -- (c4) -- (c5) -- cycle;
\draw[thick] (c1) -- (e1) (c2) -- (e2) (c3) -- (e3) (c4) -- (e4) (c5) -- (e5);
}
&=
\left[\ \vrule width 0pt depth 15pt 
\tik{0.5}{%
\path (0,0) coordinate (origin);
\path (90:0.5cm) coordinate (c1);
\path (18:0.5cm) coordinate (c2);
\path (-54:0.5cm) coordinate (c3);
\path (-126:0.5cm) coordinate (c4);
\path (162:0.5cm) coordinate (c5);
\path (90:1cm) coordinate (e1);
\path (18:1cm) coordinate (e2);
\path (-54:1cm) coordinate (e3);
\path (-126:1cm) coordinate (e4);
\path (162:1cm) coordinate (e5);
\draw[thick] (e2) .. controls (c2) and (c2) .. (c1);
\draw[thick] (e3) .. controls (c3) and (c3) .. (c4);
\draw[thick] (c1) -- (c5) -- (c4) ;
\draw[thick] (c1) -- (e1) (c5) -- (e5) (c4) -- (e4) ;
}
+
\tik{0.5}{%
\path (0,0) coordinate (origin);
\path (90:0.5cm) coordinate (c1);
\path (18:0.5cm) coordinate (c2);
\path (-54:0.5cm) coordinate (c3);
\path (-126:0.5cm) coordinate (c4);
\path (162:0.5cm) coordinate (c5);
\path (90:1cm) coordinate (e1);
\path (18:1cm) coordinate (e2);
\path (-54:1cm) coordinate (e3);
\path (-126:1cm) coordinate (e4);
\path (162:1cm) coordinate (e5);
\draw[thick] (e3) .. controls (c3) and (c3) .. (c2);
\draw[thick] (e4) .. controls (c4) and (c4) .. (c5);
\draw[thick] (c2) -- (c1) -- (c5) ;
\draw[thick] (c2) -- (e2) (c1) -- (e1) (c5) -- (e5) ;
}
+
\tik{0.5}{%
\path (0,0) coordinate (origin);
\path (90:0.5cm) coordinate (c1);
\path (18:0.5cm) coordinate (c2);
\path (-54:0.5cm) coordinate (c3);
\path (-126:0.5cm) coordinate (c4);
\path (162:0.5cm) coordinate (c5);
\path (90:1cm) coordinate (e1);
\path (18:1cm) coordinate (e2);
\path (-54:1cm) coordinate (e3);
\path (-126:1cm) coordinate (e4);
\path (162:1cm) coordinate (e5);
\draw[thick] (e4) .. controls (c4) and (c4) .. (c3);
\draw[thick] (e5) .. controls (c5) and (c5) .. (c1);
\draw[thick] (c3) -- (c2) -- (c1) ;
\draw[thick] (c3) -- (e3) (c2) -- (e2) (c1) -- (e1) ;
}
+
\tik{0.5}{%
\path (0,0) coordinate (origin);
\path (90:0.5cm) coordinate (c1);
\path (18:0.5cm) coordinate (c2);
\path (-54:0.5cm) coordinate (c3);
\path (-126:0.5cm) coordinate (c4);
\path (162:0.5cm) coordinate (c5);
\path (90:1cm) coordinate (e1);
\path (18:1cm) coordinate (e2);
\path (-54:1cm) coordinate (e3);
\path (-126:1cm) coordinate (e4);
\path (162:1cm) coordinate (e5);
\draw[thick] (e5) .. controls (c5) and (c5) .. (c4);
\draw[thick] (e1) .. controls (c1) and (c1) .. (c2);
\draw[thick] (c4) -- (c3) -- (c2) ;
\draw[thick] (c4) -- (e4) (c3) -- (e3) (c2) -- (e2) ;
}
+\tik{0.5}{%
\path (0,0) coordinate (origin);
\path (90:0.5cm) coordinate (c1);
\path (18:0.5cm) coordinate (c2);
\path (-54:0.5cm) coordinate (c3);
\path (-126:0.5cm) coordinate (c4);
\path (162:0.5cm) coordinate (c5);
\path (90:1cm) coordinate (e1);
\path (18:1cm) coordinate (e2);
\path (-54:1cm) coordinate (e3);
\path (-126:1cm) coordinate (e4);
\path (162:1cm) coordinate (e5);
\draw[thick] (e1) .. controls (c1) and (c1) .. (c5);
\draw[thick] (e2) .. controls (c2) and (c2) .. (c3);
\draw[thick] (c5) -- (c4) -- (c3) ;
\draw[thick] (c5) -- (e5) (c4) -- (e4) (c3) -- (e3) ;
}
\ \right]\\
&\quad 
-\left[\ \vrule width 0pt depth 15pt 
\tik{0.5}{%
\path (0,0) coordinate (origin);
\path (90:0.5cm) coordinate (c1);
\path (18:0.5cm) coordinate (c2);
\path (-54:0.5cm) coordinate (c3);
\path (-126:0.5cm) coordinate (c4);
\path (162:0.5cm) coordinate (c5);
\path (90:1cm) coordinate (e1);
\path (18:1cm) coordinate (e2);
\path (-54:1cm) coordinate (e3);
\path (-126:1cm) coordinate (e4);
\path (162:1cm) coordinate (e5);
\draw[thick] (e4) .. controls (c4) and (c4) .. (c5);
\draw[thick] (e1) .. controls (c1) and (c1) .. (c5);
\draw[thick]  (c2) -- (c3) ;
\draw[thick]  (c5) -- (e5) (c2) -- (e2) (c3) -- (e3)  ;
}
+
\tik{0.5}{%
\path (0,0) coordinate (origin);
\path (90:0.5cm) coordinate (c1);
\path (18:0.5cm) coordinate (c2);
\path (-54:0.5cm) coordinate (c3);
\path (-126:0.5cm) coordinate (c4);
\path (162:0.5cm) coordinate (c5);
\path (90:1cm) coordinate (e1);
\path (18:1cm) coordinate (e2);
\path (-54:1cm) coordinate (e3);
\path (-126:1cm) coordinate (e4);
\path (162:1cm) coordinate (e5);
\draw[thick] (e5) .. controls (c5) and (c5) .. (c1);
\draw[thick] (e2) .. controls (c2) and (c2) .. (c1);
\draw[thick]  (c3) -- (c4) ;
\draw[thick]  (c1) -- (e1) (c3) -- (e3) (c4) -- (e4)  ;
}
+
\tik{0.5}{%
\path (0,0) coordinate (origin);
\path (90:0.5cm) coordinate (c1);
\path (18:0.5cm) coordinate (c2);
\path (-54:0.5cm) coordinate (c3);
\path (-126:0.5cm) coordinate (c4);
\path (162:0.5cm) coordinate (c5);
\path (90:1cm) coordinate (e1);
\path (18:1cm) coordinate (e2);
\path (-54:1cm) coordinate (e3);
\path (-126:1cm) coordinate (e4);
\path (162:1cm) coordinate (e5);
\draw[thick] (e1) .. controls (c1) and (c1) .. (c2);
\draw[thick] (e3) .. controls (c3) and (c3) .. (c2);
\draw[thick]  (c4) -- (c5) ;
\draw[thick]  (c2) -- (e2) (c4) -- (e4) (c5) -- (e5)  ;
}
+
\tik{0.5}{%
\path (0,0) coordinate (origin);
\path (90:0.5cm) coordinate (c1);
\path (18:0.5cm) coordinate (c2);
\path (-54:0.5cm) coordinate (c3);
\path (-126:0.5cm) coordinate (c4);
\path (162:0.5cm) coordinate (c5);
\path (90:1cm) coordinate (e1);
\path (18:1cm) coordinate (e2);
\path (-54:1cm) coordinate (e3);
\path (-126:1cm) coordinate (e4);
\path (162:1cm) coordinate (e5);
\draw[thick] (e2) .. controls (c2) and (c2) .. (c3);
\draw[thick] (e4) .. controls (c4) and (c4) .. (c3);
\draw[thick]  (c5) -- (c1) ;
\draw[thick]  (c3) -- (e3) (c5) -- (e5) (c1) -- (e1)  ;
}
+
\tik{0.5}{%
\path (0,0) coordinate (origin);
\path (90:0.5cm) coordinate (c1);
\path (18:0.5cm) coordinate (c2);
\path (-54:0.5cm) coordinate (c3);
\path (-126:0.5cm) coordinate (c4);
\path (162:0.5cm) coordinate (c5);
\path (90:1cm) coordinate (e1);
\path (18:1cm) coordinate (e2);
\path (-54:1cm) coordinate (e3);
\path (-126:1cm) coordinate (e4);
\path (162:1cm) coordinate (e5);
\draw[thick] (e3) .. controls (c3) and (c3) .. (c4);
\draw[thick] (e5) .. controls (c5) and (c5) .. (c4);
\draw[thick]  (c1) -- (c2) ;
\draw[thick]  (c4) -- (e4) (c1) -- (e1) (c2) -- (e2)  ;
}
\ \right]
\end{split}
\]
\end{itemize}
The reader may compare the relations obtained so far with the results in  \cite[Theorem 1.1]{Kup2} and  \cite{Kup1} for $q=1$.

\begin{remark} 
The $3$-tangles $[n]\rightarrow [0]$ without subgraphs of the form
\begin{equation}\label{eq:avoided}
\tik{0.5}{%
\draw[thick] (0,0) circle (0.5cm);
}\ ,\quad
\tik{0.5}{%
\draw[thick] (0,0) -- (0,-0.5);
\draw[thick] (0,0.5) circle (0.5cm);
}\ ,\quad
\tik{0.5}{%
\draw[thick] (0,-0.5) -- (0,-1);
\draw[thick] (0,0) circle (0.5cm);
\draw[thick] (0,0.5) -- (0,1);
}\ ,\quad 
\tik{0.5}{%
\path (0,0) coordinate (origin);
\path (150:0.5cm) coordinate (t1);
\path (30:0.5cm) coordinate (t2);
\path (-90: 0.5cm) coordinate (t3);
\path (150:1cm) coordinate (e1);
\path (30:1cm) coordinate (e2);
\path (-90:1cm) coordinate (e3);
\draw[thick] (t1) -- (t2) -- (t3) -- cycle;
\draw[thick] (e1) -- (t1) (e2) -- (t2) (e3) -- (t3);
}\ ,\quad
\tik{0.5}{%
\path (0,0) coordinate (origin);
\path (135:0.5cm) coordinate (c1);
\path (45:0.5cm) coordinate (c2);
\path (-45: 0.5cm) coordinate (c3);
\path (-135: 0.5cm) coordinate (c4);
\path (135:1cm) coordinate (e1);
\path (45:1cm) coordinate (e2);
\path (-45:1cm) coordinate (e3);
\path (-135:1cm) coordinate (e4);
\draw[thick] (c1) -- (c2) -- (c3) -- (c4) -- cycle;
\draw[thick] (c1) -- (e1) (c2) -- (e2) (c3) -- (e3) (c4) -- (e4);
}\ ,\quad
\tik{0.5}{%
\path (0,0) coordinate (origin);
\path (90:0.5cm) coordinate (c1);
\path (18:0.5cm) coordinate (c2);
\path (-54:0.5cm) coordinate (c3);
\path (-126:0.5cm) coordinate (c4);
\path (162:0.5cm) coordinate (c5);
\path (90:1cm) coordinate (e1);
\path (18:1cm) coordinate (e2);
\path (-54:1cm) coordinate (e3);
\path (-126:1cm) coordinate (e4);
\path (162:1cm) coordinate (e5);
\draw[thick] (c1) -- (c2) -- (c3) -- (c4) -- (c5) -- cycle;
\draw[thick] (c1) -- (e1) (c2) -- (e2) (c3) -- (e3) (c4) -- (e4) (c5) -- (e5);
}\ ,
\end{equation}
correspond bijectively with the \emph{nonpositive diagrams} of \cite[Definition 4.2]{Westbury2007}.  Steps \textbf{1--6} show that such $3$-tangles form a spanning set of $\Mor_{\cT_\Gamma}([n],[0])$.
\end{remark}

\begin{theorem}\label{th:G2bis}
Let $n,m\in\NN$,  and assume that the characteristic of $\FF$ is $0$.   Let $\Gamma = \{\gamma_0,\gamma_1, \gamma_2\}$, 
where $\gamma_0$ is as in \eqref{eq:gamma0}, $\gamma_1$ as in \eqref{eq:gamma1}, and $\gamma_2$ as in \eqref{eq:gamma2}. 
\begin{itemize}
\item[{\rm (a)}]  
The classes modulo $\Gamma$ of the $3$-tangles $[n]\rightarrow [m]$ without crossings and without any of the subgraphs in \eqref{eq:avoided} form a basis of $\Mor_{\cT_\Gamma}([n],[m])$.
\item[{\rm (b)}]  
The functor $\cR_\Gamma$ from Corollary \ref{co:RGamma} gives a linear isomorphism 
\[
\Mor_{\cT_\Gamma}([n],[m])\rightarrow \Hom_{\Aut(\VV,\times)}\bigl(\VV^{\otimes n},\VV^{\otimes m}\bigr),
\]
where $\frV = (\VV,\mathsf{b}, \times)$  is  a 7-dimensional vector space $\VV$ endowed with a nonzero cross product $x\times y$ relative to the nondegenerate symmetric bilinear form $\mathsf{b}$,  and $\Aut(\VV,\times)$ is a simple algebraic group of type $\GG_2$. 
\item [{\rm (c)}]   The $3$-tangles $[n]\rightarrow [n]$ as in part (a)  give a basis of the centralizer algebra $$\End_{\Aut(\VV,\times)}\bigl(\VV^{\otimes n}\bigr)\simeq \Mor_{\cT_\Gamma}([n],[n]).$$ 
\end{itemize}
\end{theorem}
\begin{proof}  
The results in 
\cite[Theorems 5.1 and 6.10]{Kup1} or \cite{Westbury2007} (which are valid in characteristic $0$) show that for any $n\in\NN$,  the dimension of $\Hom_{\Aut(\VV,\times)}\bigl(\VV^{\otimes (n+m)},\FF\bigr)$ coincides with the number of $3$-tangles $[n+m]\rightarrow [0]$ without crossings and subgraphs in \eqref{eq:avoided}. Since $\cR_\Gamma$ gives a surjective linear isomorphism $\Mor_{\cT_\Gamma}([n+m],[0])\rightarrow \Hom_{\Aut(\VV,\times)}\bigl(\VV^{\otimes (n+m)},\FF\bigr)$, this proves that this map is a bijection.

The theorem now follows using the linear isomorphisms 
\[
\Phi_{n,m}:\Hom_\FF\bigl(\VV^{\otimes n},\VV^{\otimes m}\bigr)\rightarrow \Hom_\FF\bigl(\VV^{\otimes (n+m)},\FF\bigr)
\]
in Proposition \ref{pr:tPhiPsi} and the bijection 
\[
\Phi_{n,m}:\Mor_{\cT_\Gamma}([n],[m])\rightarrow \Mor_{\cT_\Gamma}([n+m],[0])
\] 
induced by the one in \eqref{eq:Phinm}, noting that this bijection sends $3$-tangles $[n]\rightarrow [m]$ without crossings and any 
of the subgraphs in \eqref{eq:avoided} to $3$-tangles $[n+m]\rightarrow [0]$ with the same properties.
\end{proof}

\begin{remark}The sequence $c(n)=\dimm_\FF\Mor_{\cT_\Gamma}([n],[0])$ is the sequence $\#$A059710 in \cite{OEIS}. The dimension of $\Mor_{\cT_\Gamma}([n],[m])\simeq \Mor_{\cT_\Gamma}([n+m],[0])$ is then $c(n+m)$.  \end{remark}  

\begin{remark}\label{rem:mult}  Multiplication in the basis of part (c) of this theorem corresponds to composing by bordism and applying
Steps \textbf{1-6} of \eqref{eq:steps}.    One has to suppress circles and cycles of arbitrary length and
normalize the resulting $3$-tangles  using the arguments prior  to Theorem \ref{th:G2}. 
\end{remark}

\begin{remark} Relation $\gamma_2$ gives that in $\cT_\Gamma$, 
\[
\tik{0.5}{%
\path (0,0) coordinate (origin);
\path (134:1cm) coordinate (e1);
\path (54:1cm) coordinate (e2);
\path (-18:1cm) coordinate (e3);
\path (-90:1cm) coordinate (e4);
\path (-162:1cm) coordinate (e5);
\path (90:0.5cm) coordinate (m);
\path (-54:0.5cm) coordinate (n);
\draw[thick] (e1) -- (m) -- (e2); 
\draw[thick] (e4) -- (n) -- (e3); 
\draw[thick] (m) -- (origin) (n) -- (origin) (e5) -- (origin);
\draw[thin,dotted,rotate=-54] (0.3,0) ellipse (0.6cm and 0.4cm);
}
\ =\ 
-
\tik{0.5}{%
\path (0,0) coordinate (origin);
\path (134:1cm) coordinate (e1);
\path (54:1cm) coordinate (e2);
\path (-18:1cm) coordinate (e3);
\path (-90:1cm) coordinate (e4);
\path (-162:1cm) coordinate (e5);
\path (90:0.5cm) coordinate (m);
\path (-126:0.5cm) coordinate (n);
\draw[thick] (e1) -- (m) -- (e2); 
\draw[thick] (e4) -- (n) -- (e5); 
\draw[thick] (m) -- (origin) (n) -- (origin) (e3) -- (origin);
}
\ +\ 
\begin{matrix}
\text{\small linear combination of $3$-tangles}\\[-2pt]
\text{\small with a lower number of trivalent nodes}\\[-2pt]
\text{\small (although with crossings),}
\end{matrix}
\]
but the graph on the right is obtained by rotating counterclockwise the graph on the left an angle of $\frac{4\pi}{3}$.
If we then repeat the argument five times,  we obtain
\[
\tik{0.5}{%
\path (0,0) coordinate (origin);
\path (134:1cm) coordinate (e1);
\path (54:1cm) coordinate (e2);
\path (-18:1cm) coordinate (e3);
\path (-90:1cm) coordinate (e4);
\path (-162:1cm) coordinate (e5);
\path (90:0.5cm) coordinate (m);
\path (-54:0.5cm) coordinate (n);
\draw[thick] (e1) -- (m) -- (e2); 
\draw[thick] (e4) -- (n) -- (e3); 
\draw[thick] (m) -- (origin) (n) -- (origin) (e5) -- (origin);
}
\ =\ 
-
\tik{0.5}{%
\path (0,0) coordinate (origin);
\path (134:1cm) coordinate (e1);
\path (54:1cm) coordinate (e2);
\path (-18:1cm) coordinate (e3);
\path (-90:1cm) coordinate (e4);
\path (-162:1cm) coordinate (e5);
\path (90:0.5cm) coordinate (m);
\path (-54:0.5cm) coordinate (n);
\draw[thick] (e1) -- (m) -- (e2); 
\draw[thick] (e4) -- (n) -- (e3); 
\draw[thick] (m) -- (origin) (n) -- (origin) (e5) -- (origin);
}
\ +\ 
\begin{matrix}
\text{\small linear combination of $3$-tangles}\\[-2pt]
\text{\small with a lower number of trivalent nodes,}
\end{matrix}
\]
so we get
\[
\tik{0.5}{%
\path (0,0) coordinate (origin);
\path (134:1cm) coordinate (e1);
\path (54:1cm) coordinate (e2);
\path (-18:1cm) coordinate (e3);
\path (-90:1cm) coordinate (e4);
\path (-162:1cm) coordinate (e5);
\path (90:0.5cm) coordinate (m);
\path (-54:0.5cm) coordinate (n);
\draw[thick] (e1) -- (m) -- (e2); 
\draw[thick] (e4) -- (n) -- (e3); 
\draw[thick] (m) -- (origin) (n) -- (origin) (e5) -- (origin);
}
\ =\ 
\begin{matrix}
\text{\small linear combination of $3$-tangles}\\[-2pt]
\text{\small with a lower number of trivalent nodes.}
\end{matrix}
\]

This proves that $\Mor_{\cT_\Gamma}([n],[0])$ is spanned by $3$-tangles where crossings are allowed, but where each connected component has at most two trivalent nodes, so each connected component is of the form:
\begin{itemize}
\item
$
\tik{0.5}{%
\node[punto] (a1) at (1,0) [label=above: {\tiny $i$}] {};
\node[punto] (a2) at (2,0) [label=above: {\tiny $j$}] {};
\path (1.5,-0.5) coordinate (b1);
\draw[thick] (a1) .. controls (1,-0.3) and (1.2,-0.5) .. (b1);
\draw[thick] (a2) .. controls (2,-0.3) and (1.8,-0.5) .. (b1);
\draw[very thin] (0.5,0) -- (2.5,0);
}
$, which corresponds under $\cR_\Gamma$ to $\mathsf{b}( v_i,v_j)$,

\item
$
\tik{0.5}{%
\node[punto] (a1) at (1,0) [label=above: {\tiny $i$}] {};
\node[punto] (a2) at (2,0) [label=above: {\tiny $j$}] {};
\node[punto] (a3) at (3,0) [label=above: {\tiny $k$}] {};
\path (1.5,-0.5) coordinate (b1);
\draw[thick] (a1) .. controls (1,-0.3) and (1.2,-0.5) .. (b1);
\draw[thick] (a2) .. controls (2,-0.3) and (1.8,-0.5) .. (b1);
\draw[thick] (b1) .. controls (1.5,-0.8) and (3,-1) .. (a3);
\draw[very thin] (0.5,0)-- (3.5,0);
}
$, which corresponds under $\cR_\Gamma$ to $\mathsf{b}( v_i\times v_j,v_k)$, and

\item
$
\tik{0.5}{%
\node[punto] (a1) at (1,0) [label=above: {\tiny $i$}] {};
\node[punto] (a2) at (2,0) [label=above: {\tiny $j$}] {};
\node[punto] (a3) at (3,0) [label=above: {\tiny $k$}] {};
\node[punto] (a4) at (4,0) [label=above: {\tiny $l$}] {};
\path (1.5,-0.5) coordinate (b1);
\path (3.5,-0.5) coordinate (b2);
\path (2.5,-1.5) coordinate (c);
\draw[thick] (a1) .. controls (1,-0.3) and (1.2,-0.5) .. (b1);
\draw[thick] (a2) .. controls (2,-0.3) and (1.8,-0.5) .. (b1);
\draw[thick] (a3) .. controls (3,-0.3) and (3.2,-0.5) .. (b2);
\draw[thick] (a4) .. controls (4,-0.3) and (3.8,-0.5) .. (b2);
\draw[thick] (b1) .. controls (1.5,-1) and (3.5,-1) .. (b2);
\draw[very thin] (0.5,0)-- (4.5,0);
}
$, which corresponds under $\cR_\Gamma$ to $\mathsf{b}( v_i\times v_j,v_k\times v_l)$.
\end{itemize}

Hence, from the fact that $\cR_\Gamma$ gives a linear surjection $\Mor_{\cT_\Gamma}([n],[0])\rightarrow \Hom_{\Aut(\VV,\times)}\bigl(\VV^{\otimes n},\FF\bigr)$, which follows from the arguments in the proof of \cite[Theorem 5.1]{Kup1}, we recover the First Fundamental Theorem of Invariant Theory for $\GG_2$ in characteristic $0$ \cite{Schwarz}, used at the beginning of  Theorem \ref{th:G2}.

The $3$-tangles with crossings and connected components with at most two trivalent nodes do not form a basis of the spaces $\Mor_{\cT_\Gamma}([n],[m])$.
\end{remark}

%
%

\section{Three-dimensional cross products}\label{se:3_dim}

In this section, $\frV=(\VV,\mathsf{b},\times)$ will consist of a $3$-dimensional vector space over a field $\FF$ of characteristic $\ne 2$, endowed with a nonzero cross product $x\times y$, relative to a nondegenerate symmetric bilinear form $\mathsf{b}$.

Then the following conditions hold:
\begin{subequations}\label{eq:CPbis}
\begin{align}
& \mathsf{b}( x\times y,z)=\mathsf{b}( x,y\times z),\label{eq:CPbis1}\\
& x\times y =-y\times x,\label{eq:CPbis2}\\
& (x\times y)\times z=\mathsf{b}( x,z)y-\mathsf{b}( y,z)x,\label{eq:CPbis3}
\end{align}
\end{subequations}
for all $x,y,z\in \VV$. Conditions \eqref{eq:CPbis2} and \eqref{eq:CPbis3} can be expressed as \\  $\wh{\mathsf{c}}_1(x_1\otimes x_2)=0=
\wh{\mathsf{c}}_2(x_1\otimes x_2\otimes x_3)$ for any $x_1,x_2,x_3\in \VV$, where
\begin{equation}\label{eq:c3}
\begin{split}
&\wh{\mathsf{c}}_1(x_1\otimes x_2)=x_1\times x_2+x_2\times x_1,\\
&\wh{\mathsf{c}}_2(x_1\otimes x_2\otimes x_3)=(x_1\times x_2)\times x_3-\mathsf{b}( x_1,x_3)x_2+\mathsf{b}( x_2,x_3)x_1,
\end{split}
\end{equation}  
so $\wh{\mathsf{c}}_1\in\Hom_\FF(\VV^{\otimes 2},\VV)$ and $\wh{\mathsf{c}}_2\in\Hom_\FF(\VV^{\otimes 3},\VV)$,
and under $\cR_{\cV}$, these are the images respectively  of 
\begin{equation}\label{eq:whgamma1} \wh\gamma_1=\mu+\mu\circ\tau\in\Mor_{\cT_2}([2],[1])\end{equation}
and of

\[
\wh\gamma_2'=\ 
\tik{0.5}{%
\draw[very thin] (0,0) -- (4,0);
\draw[very thin] (0,-2) -- (4,-2);
\path (1,0) coordinate (a1);
\path (2,0) coordinate (a2);
\path (3,0) coordinate (a3);
\path (2,-2) coordinate (b1);
\path (1.5,-0.5) coordinate (m1);
\path (2,-1.2) coordinate (m2);
\draw[thick] (a1) -- (m1) (a2) -- (m1) (a3)-- (m2) (m1) -- (m2) (m2) -- (b1);
}
-
\tik{0.5}{%
\path (1,0) coordinate (a1);
\path (2,0) coordinate (a2);
\path (3,0) coordinate (a3);
\path (2,-2) coordinate (b1);
\draw[thick] (a2) -- (b1);
\draw[cross] (a1) .. controls (1,-0.8) and (3,-0.8) .. (a3);
\draw[very thin] (0,0) -- (4,0);
\draw[very thin] (0,-2) -- (4,-2);
}
+
\tik{0.5}{%
\path (1,0) coordinate (a1);
\path (2,0) coordinate (a2);
\path (3,0) coordinate (a3);
\path (2,-2) coordinate (b1);
\draw[thick] (a1) -- (b1);
\draw[thick] (a2) to [bend right=60] (a3);
\draw[very thin] (0,0) -- (4,0);
\draw[very thin] (0,-2) -- (4,-2);
}
\ ,\qquad \wh\gamma_2'\in\Mor_{\cT_2}([3],[1]).
\]
Let us consider, instead of $\wh\gamma_2'$, the $3$-tangle
\begin{equation}\label{eq:whgamma2}
\wh\gamma_2=(\wh\gamma_2'\sqcup  \II_1)\circ(\II_2\sqcup \beta_2^t):\ 
\tik{0.5}{%
\path (1,0) coordinate (a1);
\path (3,0) coordinate (a2);
\path (1,-2) coordinate (b1);
\path (3,-2) coordinate (b2);
\path (2,-0.6) coordinate (m1);
\path (2,-1.4) coordinate (m2);
\draw[thick] (a1) -- (m1) -- (a2);
\draw[thick] (b1) -- (m2) -- (b2);
\draw[thick] (m1) -- (m2);
\draw[very thin] (0,0) -- (4,0);
\draw[very thin] (0,-2) -- (4,-2);
}
-
\tik{0.5}{%
\path (1,0) coordinate (a1);
\path (3,0) coordinate (a2);
\path (1,-2) coordinate (b1);
\path (3,-2) coordinate (b2);
\draw[thick] (a2)to (b1);
\draw[cross] (a1) to (b2);
\draw[very thin] (0,0) -- (4,0);
\draw[very thin] (0,-2) -- (4,-2);
}
+
\tik{0.5}{%
\path (1,0) coordinate (a1);
\path (3,0) coordinate (a2);
\path (1,-2) coordinate (b1);
\path (3,-2) coordinate (b2);
\draw[thick] (a1)to [bend left=60] (b1);
\draw[thick] (a2) to [bend right=60] (b2);
\draw[very thin] (0,0) -- (4,0);
\draw[very thin] (0,-2) -- (4,-2);
}\ ,
\end{equation}
where  $\wh\gamma_2\in\Mor_{\cT_2}([2],[2])$.
(Conversely, we can recover $\wh\gamma_2'$ from $\wh\gamma_2$ by composition and disjoint union.)

Also, 
\begin{equation}\label{eq:c4}
\wh{\mathsf{c}}_0=(\mathsf{b}\circ \mathsf{b}^t)(1)- \dimm_\FF\VV =(\mathsf{b}\circ \mathsf{b}^t)(1)-3\in\FF\simeq\Hom_\FF(\FF,\FF)\end{equation}
 is the image of
\begin{equation}\label{eq:whgamma0}
\wh \gamma_0:=\beta\circ\beta^t -3:\  \ 
\tik{0.5}{%
\draw[thick] (0,0) circle (0.5cm);
}
\ - 3\cdot 1,\qquad\qquad \wh\gamma_0\in\Mor_{\cT_2}([0],[0]).
\end{equation}

Therefore, for $\frV=(\VV,\mathsf{b},\times)$ as above, $\frV$ is of tensor type $\{\wh{\mathsf{c}}_0,
\wh{\mathsf{c}}_1,\wh{\mathsf{c}}_2\}$, and hence Corollary \ref{co:RGamma} gives a functor $\cR_{\wh \Gamma}:\cT_{\wh \Gamma}\rightarrow \cV$, with $\wh \Gamma=\{\wh \gamma_0,\wh\gamma_1,\wh \gamma_2\}$. As in the previous section, the image $\cR_{\wh \Gamma}\bigl(\Mor_{\cT_{\wh \Gamma}}([n],[m])\bigr)$ is contained in $\Hom_{\Aut(\VV,\times)}(\VV^{\otimes n},\VV^{\otimes m})$, but here $\Aut(\VV,\times)$ is the special orthogonal group $\SO(\VV,\mathsf{b})$.

\begin{theorem}
Let $\FF$ be an infinite field of characteristic $\ne 2$, and let $\wh \Gamma = \{\wh \gamma_0, \wh \gamma_1, \wh \gamma_2\}$,
for $\wh \gamma_0$ as in \eqref{eq:whgamma0}, $\wh\gamma_1$ as in \eqref{eq:whgamma1}, and $\wh \gamma_2$
as in \eqref{eq:whgamma2}.   Then for any $n,m\in\NN$,
\[
\cR_{\wh \Gamma}\bigl(\Mor_{\cT_{\wh \Gamma}}([n],[m])\bigr)=\Hom_{\SO(\VV,\mathsf{b})}(\VV^{\otimes n},\VV^{\otimes m}).
\]
\end{theorem}
\begin{proof}
It is enough to note that the classical invariant theory for $\SO(\VV,\mathsf{b})$ (see \cite{Weyl} or \cite[Theorem 5.6]{dCP}) shows that 
$\Hom_{\SO(\VV,\mathsf{b})}(\VV^{\otimes n},\VV^{\otimes m})$ is generated by composition and tensor products by $1_\VV$, $\mathsf{b},\mathsf{b}^t,\tau$ and $\times$, because $\mathsf{b}( x\times y,z)$ is a determinant map on $\VV$.
\end{proof}

As in Section \ref{se:7dim},  we can get rid of circles and crossings in $\cT_{\wh \Gamma}$,  and
\[
\tik{0.5}{%
\draw[thick] (0,0) -- (0,-1);
\draw[thick] (0,0.5) circle (0.5cm);
\node at (2,0) {$=\ 0$};
},\qquad
\tik{0.5}{%
\draw[thick] (0,-0.5) -- (0,-1);
\draw[thick] (0,0) circle (0.5cm);
\draw[thick] (0,0.5) -- (0,1);
\node at (2,0) {$=\ -2$};
\draw[thick] (3.5,1) -- (3.5,-1);
}
\]
Also, from $\wh \gamma_2$ and the $3$-tangle obtained `rotating it' we obtain:
\begin{equation}\label{eq:assoc1}
\tik{0.5}{%
\path (1,0) coordinate (a1);
\path (3,0) coordinate (a2);
\path (1,-2) coordinate (b1);
\path (3,-2) coordinate (b2);
\path (2,-0.6) coordinate (m1);
\path (2,-1.4) coordinate (m2);
\draw[thick] (a1) -- (m1) -- (a2);
\draw[thick] (b1) -- (m2) -- (b2);
\draw[thick] (m1) -- (m2);
\draw[very thin] (0,0) -- (4,0);
\draw[very thin] (0,-2) -- (4,-2);
}
-
\tik{0.5}{%
\path (1,0) coordinate (a1);
\path (3,0) coordinate (a2);
\path (1,-2) coordinate (b1);
\path (3,-2) coordinate (b2);
\path (1.6,-1) coordinate (m1);
\path (2.4,-1) coordinate (m2);
\draw[thick] (a1) -- (m1) -- (b1);
\draw[thick] (a2) -- (m2) -- (b2);
\draw[thick] (m1) -- (m2);
\draw[very thin] (0,0) -- (4,0);
\draw[very thin] (0,-2) -- (4,-2);
}
\ =\ 
\tik{0.5}{%
\path (1,0) coordinate (a1);
\path (3,0) coordinate (a2);
\path (1,-2) coordinate (b1);
\path (3,-2) coordinate (b2);
\draw[thick] (b1)to [bend left=60] (b2);
\draw[thick] (a1) to [bend right=60] (a2);
\draw[very thin] (0,0) -- (4,0);
\draw[very thin] (0,-2) -- (4,-2);
}
-
\tik{0.5}{%
\path (1,0) coordinate (a1);
\path (3,0) coordinate (a2);
\path (1,-2) coordinate (b1);
\path (3,-2) coordinate (b2);
\draw[thick] (a1)to [bend left=60] (b1);
\draw[thick] (a2) to [bend right=60] (b2);
\draw[very thin] (0,0) -- (4,0);
\draw[very thin] (0,-2) -- (4,-2);
}
\end{equation}
This allows us to replace cycles by linear combinations of $3$-tangles without cycles. For instance:
\[
\begin{split}
\tik{0.5}{%
\path (0,0) coordinate (origin);
\path (150:0.5cm) coordinate (t1);
\path (30:0.5cm) coordinate (t2);
\path (-90: 0.5cm) coordinate (t3);
\path (150:1cm) coordinate (e1);
\path (30:1cm) coordinate (e2);
\path (-90:1cm) coordinate (e3);
\draw[thick] (t1) -- (t2) -- (t3) -- cycle;
\draw[thick] (e1) -- (t1) (e2) -- (t2) (e3) -- (t3);
\draw[thin,dotted] (0,0.2) ellipse (0.6cm and 0.4cm);
}
&=
\tik{0.5}{%
\path (0,0) coordinate (origin);
\path (150:0.5cm) coordinate (t1);
\path (30:0.5cm) coordinate (t2);
\path (-90: 0.5cm) coordinate (t3);
\path (150:1cm) coordinate (e1);
\path (30:1cm) coordinate (e2);
\path (-90:1cm) coordinate (e3);
\path (0,0.2) coordinate (m);
\draw[thick] (e1) -- (m) -- (e2);
\draw[thick] (m) --++(0,-0.3);
\draw[thick] (e3)-- ++(0,0.3);
\draw[thick] (0,-0.4) circle (0.3cm);
}
+
\tik{0.5}{%
\path (0,0) coordinate (origin);
\path (150:0.5cm) coordinate (t1);
\path (30:0.5cm) coordinate (t2);
\path (-90: 0.5cm) coordinate (t3);
\path (150:1cm) coordinate (e1);
\path (30:1cm) coordinate (e2);
\path (-90:1cm) coordinate (e3);
\path (0,0.2) coordinate (m);
\draw[thick] (e1) -- (origin) -- (e2);
\draw[thick] (origin)-- (e3);
}
-
\tik{0.5}{%
\path (0,0) coordinate (origin);
\path (150:0.5cm) coordinate (t1);
\path (30:0.5cm) coordinate (t2);
\path (-90: 0.5cm) coordinate (t3);
\path (150:1cm) coordinate (e1);
\path (30:1cm) coordinate (e2);
\path (-90:1cm) coordinate (e3);
\path (0,0.2) coordinate (m);
\draw[thick] (e1) to [bend right=40] (e2);
\draw[thick] (e3)-- ++(0,0.3);
\draw[thick] (0,-0.4) circle (0.3cm);
}
\ =\  - 
\tik{0.5}{%
\path (0,0) coordinate (origin);
\path (150:0.5cm) coordinate (t1);
\path (30:0.5cm) coordinate (t2);
\path (-90: 0.5cm) coordinate (t3);
\path (150:1cm) coordinate (e1);
\path (30:1cm) coordinate (e2);
\path (-90:1cm) coordinate (e3);
\path (0,0.2) coordinate (m);
\draw[thick] (e1) -- (origin) -- (e2);
\draw[thick] (origin) -- (e3);
}
\\[10pt]
\tik{0.5}{%
\path (0,0) coordinate (origin);
\path (135:0.5cm) coordinate (c1);
\path (45:0.5cm) coordinate (c2);
\path (-45: 0.5cm) coordinate (c3);
\path (-135: 0.5cm) coordinate (c4);
\path (135:1cm) coordinate (e1);
\path (45:1cm) coordinate (e2);
\path (-45:1cm) coordinate (e3);
\path (-135:1cm) coordinate (e4);
\draw[thick] (c1) -- (c2) -- (c3) -- (c4) -- cycle;
\draw[thick] (c1) -- (e1) (c2) -- (e2) (c3) -- (e3) (c4) -- (e4);
\draw[thin,dotted] (-0.3,0) ellipse (0.4cm and 0.6cm);
}
&=
\tik{0.5}{%
\path (0,0) coordinate (origin);
\path (135:0.5cm) coordinate (c1);
\path (45:0.5cm) coordinate (c2);
\path (-45: 0.5cm) coordinate (c3);
\path (-135: 0.5cm) coordinate (c4);
\path (135:1cm) coordinate (e1);
\path (45:1cm) coordinate (e2);
\path (-45:1cm) coordinate (e3);
\path (-135:1cm) coordinate (e4);
\path (-0.4,0) coordinate (m);
\path (origin) coordinate (t1);
\draw[thick] (t1) -- (c2) -- (c3) -- cycle;
\draw[thick] (e1) -- (m) -- (e4);
\draw[thick] (m) -- (t1) (e2) -- (c2) (e3) -- (c3);
}
+
\tik{0.5}{%
\path (0,0) coordinate (origin);
\path (135:0.5cm) coordinate (c1);
\path (45:0.5cm) coordinate (c2);
\path (-45: 0.5cm) coordinate (c3);
\path (-135: 0.5cm) coordinate (c4);
\path (135:1cm) coordinate (e1);
\path (45:1cm) coordinate (e2);
\path (-45:1cm) coordinate (e3);
\path (-135:1cm) coordinate (e4);
\draw[thick] (e1) .. controls (c1) and (c1) .. (c2);
\draw[cross] (e4) .. controls (c4) and (c4) .. (c3);
\draw[thick] (e2) -- (c2) -- (c3) -- (e3);
}
-
\tik{0.5}{%
\path (0,0) coordinate (origin);
\path (135:0.5cm) coordinate (c1);
\path (45:0.5cm) coordinate (c2);
\path (-45: 0.5cm) coordinate (c3);
\path (-135: 0.5cm) coordinate (c4);
\path (135:1cm) coordinate (e1);
\path (45:1cm) coordinate (e2);
\path (-45:1cm) coordinate (e3);
\path (-135:1cm) coordinate (e4);
\path (0,0.4) coordinate (m);
\path (0,0.2) coordinate (t1);
\path (0.3,-0.4) coordinate (t2);
\path (-0.3,-0.4) coordinate (t3);
\draw[thick] (e4) to [bend right=40] (e1);
\draw[thick] (c2) .. controls (-0.2,0.5) and (-0.2,-0.5) .. (c3);
\draw[thick] (e2) -- (c2) -- (c3) -- (e3);
}
\ =\ 
\tik{0.5}{%
\path (0,0) coordinate (origin);
\path (135:0.5cm) coordinate (c1);
\path (45:0.5cm) coordinate (c2);
\path (-45: 0.5cm) coordinate (c3);
\path (-135: 0.5cm) coordinate (c4);
\path (135:1cm) coordinate (e1);
\path (45:1cm) coordinate (e2);
\path (-45:1cm) coordinate (e3);
\path (-135:1cm) coordinate (e4);
\path (-0.4,0) coordinate (m1);
\path (0.4,0) coordinate (m2);
\draw[thick] (e1) -- (m1) -- (e4);
\draw[thick] (e2) -- (m2) -- (e3);
\draw[thick] (m1) -- (m2);
} 
+
\tik{0.5}{%
\path (0,0) coordinate (origin);
\path (135:0.5cm) coordinate (c1);
\path (45:0.5cm) coordinate (c2);
\path (-45: 0.5cm) coordinate (c3);
\path (-135: 0.5cm) coordinate (c4);
\path (135:1cm) coordinate (e1);
\path (45:1cm) coordinate (e2);
\path (-45:1cm) coordinate (e3);
\path (-135:1cm) coordinate (e4);
\path (0,0.4) coordinate (m1);
\path (0,-0.4) coordinate (m2);
\draw[thick] (e1) -- (m1) -- (e2);
\draw[thick] (e3) -- (m2) -- (e4);
\draw[thick] (m1) -- (m2);
} 
+2
\tik{0.5}{%
\path (0,0) coordinate (origin);
\path (135:0.5cm) coordinate (c1);
\path (45:0.5cm) coordinate (c2);
\path (-45: 0.5cm) coordinate (c3);
\path (-135: 0.5cm) coordinate (c4);
\path (135:1cm) coordinate (e1);
\path (45:1cm) coordinate (e2);
\path (-45:1cm) coordinate (e3);
\path (-135:1cm) coordinate (e4);
\path (-0.4,0) coordinate (m1);
\path (0.4,0) coordinate (m2);
\draw[thick] (e4) .. controls (c4) and (c1) .. (e1);
\draw[thick] (e2) .. controls (c2) and (c3) .. (e3);
} 
\\
&=
2
\tik{0.5}{%
\path (0,0) coordinate (origin);
\path (135:0.5cm) coordinate (c1);
\path (45:0.5cm) coordinate (c2);
\path (-45: 0.5cm) coordinate (c3);
\path (-135: 0.5cm) coordinate (c4);
\path (135:1cm) coordinate (e1);
\path (45:1cm) coordinate (e2);
\path (-45:1cm) coordinate (e3);
\path (-135:1cm) coordinate (e4);
\path (0,0.4) coordinate (m1);
\path (0,-0.4) coordinate (m2);
\draw[thick] (e1) -- (m1) -- (e2);
\draw[thick] (e3) -- (m2) -- (e4);
\draw[thick] (m1) -- (m2);
} 
+3
\tik{0.5}{%
\path (0,0) coordinate (origin);
\path (135:0.5cm) coordinate (c1);
\path (45:0.5cm) coordinate (c2);
\path (-45: 0.5cm) coordinate (c3);
\path (-135: 0.5cm) coordinate (c4);
\path (135:1cm) coordinate (e1);
\path (45:1cm) coordinate (e2);
\path (-45:1cm) coordinate (e3);
\path (-135:1cm) coordinate (e4);
\path (-0.4,0) coordinate (m1);
\path (0.4,0) coordinate (m2);
\draw[thick] (e4) .. controls (c4) and (c1) .. (e1);
\draw[thick] (e2) .. controls (c2) and (c3) .. (e3);
} 
-
\tik{0.5}{%
\path (0,0) coordinate (origin);
\path (135:0.5cm) coordinate (c1);
\path (45:0.5cm) coordinate (c2);
\path (-45: 0.5cm) coordinate (c3);
\path (-135: 0.5cm) coordinate (c4);
\path (135:1cm) coordinate (e1);
\path (45:1cm) coordinate (e2);
\path (-45:1cm) coordinate (e3);
\path (-135:1cm) coordinate (e4);
\path (-0.4,0) coordinate (m1);
\path (0.4,0) coordinate (m2);
\draw[thick] (e1) .. controls (c1) and (c2) .. (e2);
\draw[thick] (e4) .. controls (c4) and (c3) .. (e3);
} 
\\[10pt]
\tik{0.5}{%
\path (0,0) coordinate (origin);
\path (90:0.5cm) coordinate (c1);
\path (18:0.5cm) coordinate (c2);
\path (-54:0.5cm) coordinate (c3);
\path (-126:0.5cm) coordinate (c4);
\path (162:0.5cm) coordinate (c5);
\path (90:1cm) coordinate (e1);
\path (18:1cm) coordinate (e2);
\path (-54:1cm) coordinate (e3);
\path (-126:1cm) coordinate (e4);
\path (162:1cm) coordinate (e5);
\draw[thick] (c1) -- (c2) -- (c3) -- (c4) -- (c5) -- cycle;
\draw[thick] (c1) -- (e1) (c2) -- (e2) (c3) -- (e3) (c4) -- (e4) (c5) -- (e5);
\draw[thin,dotted,rotate=36] (0,0.3) ellipse (0.6cm and 0.4cm);
}
&=
\tik{0.5}{%
\path (0,0) coordinate (origin);
\path (90:0.5cm) coordinate (c1);
\path (18:0.5cm) coordinate (c2);
\path (-54:0.5cm) coordinate (c3);
\path (-126:0.5cm) coordinate (c4);
\path (162:0.5cm) coordinate (c5);
\path (90:1cm) coordinate (e1);
\path (18:1cm) coordinate (e2);
\path (-54:1cm) coordinate (e3);
\path (-126:1cm) coordinate (e4);
\path (162:1cm) coordinate (e5);
\path (126:0.6cm) coordinate (m1);
\path (126:0.2cm) coordinate (m2);
\draw[thick] (c2) -- (c3) -- (c4);
\draw[thick] (c2) -- (e2) (c3) -- (e3) (c4) -- (e4);
\draw[thick] (c2) -- (m2) -- (c4);
\draw[thick] (m2) -- (m1);
\draw[thick] (e1) -- (m1) -- (e5);
}
+
\tik{0.5}{%
\path (0,0) coordinate (origin);
\path (90:0.5cm) coordinate (c1);
\path (18:0.5cm) coordinate (c2);
\path (-54:0.5cm) coordinate (c3);
\path (-126:0.5cm) coordinate (c4);
\path (162:0.5cm) coordinate (c5);
\path (90:1cm) coordinate (e1);
\path (18:1cm) coordinate (e2);
\path (-54:1cm) coordinate (e3);
\path (-126:1cm) coordinate (e4);
\path (162:1cm) coordinate (e5);
\draw[thick] (c2) .. controls (c1) and (c1) .. (e1);
\draw[thick] (c4) .. controls (c5) and (c5) .. (e5);
\draw[thick]  (c2) -- (c3) -- (c4) ;
\draw[thick]  (c2) -- (e2) (c3) -- (e3) (c4) -- (e4)  ;
}
-
\tik{0.5}{%
\path (0,0) coordinate (origin);
\path (90:0.5cm) coordinate (c1);
\path (18:0.5cm) coordinate (c2);
\path (-54:0.5cm) coordinate (c3);
\path (-126:0.5cm) coordinate (c4);
\path (162:0.5cm) coordinate (c5);
\path (90:1cm) coordinate (e1);
\path (18:1cm) coordinate (e2);
\path (-54:1cm) coordinate (e3);
\path (-126:1cm) coordinate (e4);
\path (162:1cm) coordinate (e5);
\draw[thick] (e1) .. controls (c1) and (c5) .. (e5);
\draw[thick] (c2) -- (c3) -- (c4);
\draw[thick] (c2) -- (e2) (c3) -- (e3) (c4) -- (e4) ;
\draw[thick] (c2) -- (c4);
}
\\
&= 
2
\tik{0.5}{%
\path (0,0) coordinate (origin);
\path (90:0.5cm) coordinate (c1);
\path (18:0.5cm) coordinate (c2);
\path (-54:0.5cm) coordinate (c3);
\path (-126:0.5cm) coordinate (c4);
\path (162:0.5cm) coordinate (c5);
\path (90:1cm) coordinate (e1);
\path (18:1cm) coordinate (e2);
\path (-54:1cm) coordinate (e3);
\path (-126:1cm) coordinate (e4);
\path (162:1cm) coordinate (e5);
\draw[thick] (e5) .. controls (c5) and (c5) .. (c1);
\draw[thick] (e4) .. controls (c4) and (c4) .. (c3);
\draw[thick] (c1) -- (c2) -- (c3) ;
\draw[thick] (c1) -- (e1) (c2) -- (e2) (c3) -- (e3) ;
}
+
\tik{0.5}{%
\path (0,0) coordinate (origin);
\path (90:0.5cm) coordinate (c1);
\path (18:0.5cm) coordinate (c2);
\path (-54:0.5cm) coordinate (c3);
\path (-126:0.5cm) coordinate (c4);
\path (162:0.5cm) coordinate (c5);
\path (90:1cm) coordinate (e1);
\path (18:1cm) coordinate (e2);
\path (-54:1cm) coordinate (e3);
\path (-126:1cm) coordinate (e4);
\path (162:1cm) coordinate (e5);
\draw[thick] (c2) .. controls (c1) and (c1) .. (e1);
\draw[thick] (c4) .. controls (c5) and (c5) .. (e5);
\draw[thick]  (c2) -- (c3) -- (c4) ;
\draw[thick]  (c2) -- (e2) (c3) -- (e3) (c4) -- (e4) (c5) ;
}
+3
\tik{0.5}{%
\path (0,0) coordinate (origin);
\path (90:0.5cm) coordinate (c1);
\path (18:0.5cm) coordinate (c2);
\path (-54:0.5cm) coordinate (c3);
\path (-126:0.5cm) coordinate (c4);
\path (162:0.5cm) coordinate (c5);
\path (90:1cm) coordinate (e1);
\path (18:1cm) coordinate (e2);
\path (-54:1cm) coordinate (e3);
\path (-126:1cm) coordinate (e4);
\path (162:1cm) coordinate (e5);
\draw[thick] (e4) .. controls (c4) and (c4) .. (c5);
\draw[thick] (e1) .. controls (c1) and (c1) .. (c5);
\draw[thick]  (c2) -- (c3) ;
\draw[thick]  (c5) -- (e5) (c2) -- (e2) (c3) -- (e3)  ;
}
-
\tik{0.5}{%
\path (0,0) coordinate (origin);
\path (90:0.5cm) coordinate (c1);
\path (18:0.5cm) coordinate (c2);
\path (-54:0.5cm) coordinate (c3);
\path (-126:0.5cm) coordinate (c4);
\path (162:0.5cm) coordinate (c5);
\path (90:1cm) coordinate (e1);
\path (18:1cm) coordinate (e2);
\path (-54:1cm) coordinate (e3);
\path (-126:1cm) coordinate (e4);
\path (162:1cm) coordinate (e5);
\draw[thick] (e5) .. controls (c5) and (c5) .. (c1);
\draw[thick] (e2) .. controls (c2) and (c2) .. (c1);
\draw[thick]  (c3) -- (c4) ;
\draw[thick]  (c1) -- (e1) (c3) -- (e3) (c4) -- (e4)  ;
}
+
\tik{0.5}{%
\path (0,0) coordinate (origin);
\path (90:0.5cm) coordinate (c1);
\path (18:0.5cm) coordinate (c2);
\path (-54:0.5cm) coordinate (c3);
\path (-126:0.5cm) coordinate (c4);
\path (162:0.5cm) coordinate (c5);
\path (90:1cm) coordinate (e1);
\path (18:1cm) coordinate (e2);
\path (-54:1cm) coordinate (e3);
\path (-126:1cm) coordinate (e4);
\path (162:1cm) coordinate (e5);
\draw[thick] (e2) .. controls (c2) and (c2) .. (c3);
\draw[thick] (e4) .. controls (c4) and (c4) .. (c3);
\draw[thick]  (c5) -- (c1) ;
\draw[thick]  (c3) -- (e3) (c5) -- (e5) (c1) -- (e1)  ;
}
\\[10pt]
\tik{0.5}{%
\path (0,0) coordinate (origin);
\path (150:0.5cm) coordinate (c1);
\path (90:0.5cm) coordinate (c2);
\path (30:0.5cm) coordinate (c3);
\path (-30:0.5cm) coordinate (c4);
\path (-90:0.5cm) coordinate (c5);
\path (-150:0.5cm) coordinate (c6);
\path (150:1cm) coordinate (e1);
\path (90:1cm) coordinate (e2);
\path (30:1cm) coordinate (e3);
\path (-30:1cm) coordinate (e4);
\path (-90:1cm) coordinate (e5);
\path (-150:1cm) coordinate (e6);
\draw[thick] (c1) -- (c2) -- (c3) -- (c4) -- (c5) -- (c6) -- cycle;
\draw[thick] (c1) -- (e1) (c2) -- (e2) (c3) -- (e3) (c4) -- (e4) (c5) -- (e5) (c6) -- (e6);
\draw[thin,dotted] (-0.3,0) ellipse (0.4cm and 0.6cm);
}
&=
\tik{0.5}{%
\path (0,0) coordinate (origin);
\path (150:0.5cm) coordinate (c1);
\path (90:0.5cm) coordinate (c2);
\path (30:0.5cm) coordinate (c3);
\path (-30:0.5cm) coordinate (c4);
\path (-90:0.5cm) coordinate (c5);
\path (-150:0.5cm) coordinate (c6);
\path (150:1cm) coordinate (e1);
\path (90:1cm) coordinate (e2);
\path (30:1cm) coordinate (e3);
\path (-30:1cm) coordinate (e4);
\path (-90:1cm) coordinate (e5);
\path (-150:1cm) coordinate (e6);
\path (-180:0.7cm) coordinate (m1);
\path (-180:0.3cm) coordinate (m2);
\draw[thick] (m2) -- (c2) -- (c3) -- (c4) -- (c5) -- cycle;
\draw[thick] (m1) -- (m2) (c2) -- (e2) (c3) -- (e3) (c4) -- (e4) (c5) -- (e5);
\draw[thick] (e1) -- (m1) -- (e6);
}
+
\tik{0.5}{%
\path (0,0) coordinate (origin);
\path (150:0.5cm) coordinate (c1);
\path (90:0.5cm) coordinate (c2);
\path (30:0.5cm) coordinate (c3);
\path (-30:0.5cm) coordinate (c4);
\path (-90:0.5cm) coordinate (c5);
\path (-150:0.5cm) coordinate (c6);
\path (150:1cm) coordinate (e1);
\path (90:1cm) coordinate (e2);
\path (30:1cm) coordinate (e3);
\path (-30:1cm) coordinate (e4);
\path (-90:1cm) coordinate (e5);
\path (-150:1cm) coordinate (e6);
\path (-180:0.7cm) coordinate (m1);
\path (-180:0.3cm) coordinate (m2);
\draw[thick] (e2) -- (c2) -- (c3) -- (c4) -- (c5) -- (e5);
\draw[thick] (e3) -- (c3) (e4) -- (c4);
\draw[thick] (c2) .. controls (c1) and (c1) .. (e1);
\draw[thick] (c5) .. controls (c6) and (c6) .. (e6);
}
-
\tik{0.5}{%
\path (0,0) coordinate (origin);
\path (150:0.5cm) coordinate (c1);
\path (90:0.5cm) coordinate (c2);
\path (30:0.5cm) coordinate (c3);
\path (-30:0.5cm) coordinate (c4);
\path (-90:0.5cm) coordinate (c5);
\path (-150:0.5cm) coordinate (c6);
\path (150:1cm) coordinate (e1);
\path (90:1cm) coordinate (e2);
\path (30:1cm) coordinate (e3);
\path (-30:1cm) coordinate (e4);
\path (-90:1cm) coordinate (e5);
\path (-150:1cm) coordinate (e6);
\path (-180:0.7cm) coordinate (m1);
\path (-180:0.3cm) coordinate (m2);
\draw[thick] (e2) -- (c2) -- (c3) -- (c4) -- (c5) -- (e5);
\draw[thick] (e3) -- (c3) (e4) -- (c4);
\draw[thick] (e1) to [bend left=60] (e6);
\draw[thick] (c5) to [bend left=60] (c2);
}
\ =\ \cdots
\end{split}
\]

Therefore, any $3$-tangle is equivalent in $\cT_{\wh\Gamma}$ to a linear combination of $3$-tangles without crossings, circles and cycles (i.e., a linear combination of planar trees). Let us call these \emph{tree $3$-tangles}.  
 
Moreover, equation \eqref{eq:assoc1} can be transformed easily to
\begin{equation}\label{eq:assoc2}
\tik{0.5}{%
\path (1,0) coordinate (a1);
\path (2,0) coordinate (a2);
\path (3,0) coordinate (a3);
\path (2,-2) coordinate (b1);
\path (1.5,-0.5) coordinate (m1);
\path (2,-1.2) coordinate (m2);
\draw[thick] (a1) -- (m1) (a2) -- (m1) (a3)-- (m2) (m1) -- (m2) (m2) -- (b1);
}
\ -\ 
\tik{0.5}{%
\path (1,0) coordinate (a1);
\path (2,0) coordinate (a2);
\path (3,0) coordinate (a3);
\path (2,-2) coordinate (b1);
\path (2.5,-0.5) coordinate (m1);
\path (2,-1.2) coordinate (m2);
\draw[thick] (a1) -- (m2) (a2) -- (m1) (a3)-- (m1) (m1) -- (m2) (m2) -- (b1);
}
\ =\ 
\tik{0.5}{%
\path (1,0) coordinate (a1);
\path (2,0) coordinate (a2);
\path (3,0) coordinate (a3);
\path (2,-2) coordinate (b1);
\draw[thick] (a3) -- (b1);
\draw[thick] (a1) to [bend right=60] (a2);
}
\ -\ 
\tik{0.5}{%
\path (1,0) coordinate (a1);
\path (2,0) coordinate (a2);
\path (3,0) coordinate (a3);
\path (2,-2) coordinate (b1);
\draw[thick] (a1) -- (b1);
\draw[thick] (a2) to [bend right=60] (a3);
}
\end{equation}
that is, 
\[
\tik{0.5}{%
\path (1,0) coordinate (a1);
\path (2,0) coordinate (a2);
\path (3,0) coordinate (a3);
\path (2,-2) coordinate (b1);
\path (1.5,-0.5) coordinate (m1);
\path (2,-1.2) coordinate (m2);
\draw[thick] (a1) -- (m1) (a2) -- (m1) (a3)-- (m2) (m1) -- (m2) (m2) -- (b1);
}
\ -\ 
\tik{0.5}{%
\path (1,0) coordinate (a1);
\path (2,0) coordinate (a2);
\path (3,0) coordinate (a3);
\path (2,-2) coordinate (b1);
\path (2.5,-0.5) coordinate (m1);
\path (2,-1.2) coordinate (m2);
\draw[thick] (a1) -- (m2) (a2) -- (m1) (a3)-- (m1) (m1) -- (m2) (m2) -- (b1);
}
\ =\ 
\begin{matrix}
\text{\small linear combination of tree $3$-tangles}\\[-2pt]
\text{\small with a lower number of trivalent nodes.}
\end{matrix}
\]
and this gives a variation of associativity. For instance:
\[
\tik{0.5}{%
\path (1,0) coordinate (a1);
\path (2,0) coordinate (a2);
\path (3,0) coordinate (a3);
\path (4,0) coordinate (a4);
\path (5,0) coordinate (a5);
\path (3,-3) coordinate (b1);
\path (1.5,-0.8) coordinate (m1);
\path (3.5,-0.5) coordinate (m2);
\path (4,-1.3) coordinate (m3);
\path (3,-2) coordinate (m4);
\draw[thick] (a1) -- (m1) (a2) -- (m1) (a3)-- (m2) (a4) -- (m2);
\draw[thick] (m2) -- (m3) (a5) -- (m3) (m1) -- (m4)  (m3) -- (m4) (m4) -- (b1);
}
-
\tik{0.5}{%
\path (1,0) coordinate (a1);
\path (2,0) coordinate (a2);
\path (3,0) coordinate (a3);
\path (4,0) coordinate (a4);
\path (5,0) coordinate (a5);
\path (3,-3) coordinate (b1);
\path (1.5,-0.5) coordinate (m1);
\path (2,-1) coordinate (m2);
\path (2.5,-1.5) coordinate (m3);
\path (3,-2) coordinate (m4);
\draw[thick] (a1) -- (m1) (a2) -- (m1) (m1) -- (m2) (a3)-- (m2) (m2) -- (m3);
\draw[thick] (a4) -- (m3) (m3) -- (m4) (a5) -- (m4)  (m4) -- (b1);
}
\ =\ 
\begin{matrix}
\text{\small linear combination of tree $3$-tangles}\\[-2pt]
\text{\small with a lower number of trivalent nodes.}
\end{matrix}
\]

For any $n,m\in\NN$, $n+m\geq 2$, the $3$-tangle $[n]\rightarrow [m]$ given by:
\begin{equation} \begin{split} 
\begin{aligned}
\tik{0.5}{%
\node[punto] (a1) at (1,0)  [label=above: {\tiny $1$}] {};
\node[punto] (a2) at (2,0)  [label=above: {\tiny $2$}] {};
\node[punto] (a3) at (3,0)  [label=above: {\tiny $3$}] {};
\node[punto] (a4) at (5,0)  [label=above: {\tiny $n$}] {};
\node at (4,0.3) {$\cdots$};
\node[punto] (b1) at (1.5,-2)  [label=below: {\tiny $1'$}] {};
\node[punto] (b2) at (2.5,-2)  [label=below: {\tiny $2'$}] {};
\node[punto] (b3) at (4.5,-2)  [label=below: {\tiny $m'$}] {};
\node at (3.5,-2.3) {$\cdots$};
\path (1.25,-0.2) coordinate (m1);
\path (1.5,-0.4) coordinate (m2);
\path (2,-0.8) coordinate (m3);
\path (2.2,-1.3) coordinate (n1);
\path (1.8,-1.7) coordinate (n2);
\draw[thick] (a1) -- (m1) -- (m2) -- (m3) -- (n1) -- (n2) -- (b1);
\draw[thick] (a2) -- (m1) (a3) -- (m2) (a4) -- (m3) (b2) -- (n2) (b3) -- (n1);
\draw[very thin] (0,0) -- (6,0);
\draw[very thin] (0.5,-2) -- (5.5,-2);
\node at (3,-0.3) {$\cdots$};
\node at (3,-1.9) {$\cdots$};
}
& 
\quad \text{if $n\ne 0\ne m$,}
\\[10pt]
\tik{0.5}{%
\node[punto] (a1) at (1,0)  [label=above: {\tiny $1$}] {};
\node[punto] (a2) at (2,0)  [label=above: {\tiny $2$}] {};
\node[punto] (a3) at (3,0)  [label=above: {\tiny $3$}] {};
\node[punto] (a4) at (5,0)  [label=above: {\tiny $n$}] {};
\node at (4,0.3) {$\cdots$};
\node at (4,-0.3) {$\cdots$};
\draw[thick] (a1) -- (2,-1) -- (3,-1.2) -- (4,-1) -- (a4);
\draw[thick] (a2) -- (2,-1) (a3) -- (3,-1.2);
\draw[very thin] (0,0) -- (6,0);
}
&
\quad \text{if $m=0$,}
\\[10pt]
\tik{0.5}{%
\node[punto] (a1) at (1,0)  [label=below: {\tiny $1'$}] {};
\node[punto] (a2) at (2,0)  [label=below: {\tiny $2'$}] {};
\node[punto] (a3) at (3,0)  [label=below: {\tiny $3'$}] {};
\node[punto] (a4) at (5,0)  [label=below: {\tiny $m'$}] {};
\node at (4,0.3) {$\cdots$};
\node at (4,-0.3) {$\cdots$};
\draw[thick] (a1) -- (2,1) -- (3,1.2) -- (4,1) -- (a4);
\draw[thick] (a2) -- (2,1) (a3) -- (3,1.2);
\draw[very thin] (0,0) -- (6,0);
}&\quad\text{if $n=0$,}
\end{aligned}
\end{split}
\end{equation} 
will be called the \emph{normalized connected $3$-tangle} in $\Mor_{\cT_{\wh\Gamma}}([n],[m])$. The associativity above shows that any $3$-tangle in $\Mor_{\cT_{\wh\Gamma}}([n],[m])$ is a linear combination of $3$-tangles which are a disjoint union of normalized connected $3$-tangles.  Disjoint unions of normalized connected $3$-tangles will be called \emph{normalized} $3$-tangles. They are determined by the $1$-valent vertices in each connected component. For instance, if $T:[6]\rightarrow [3]$ is a normalized $3$-tangle, $\partial T=\{1,2,3,4,5,6,1',2',3'\}$, and the connected components split $\partial T$ into $\{1,2,3,6,3'\}$, $\{4,5\}$, and $\{1',2'\}$, then we have:
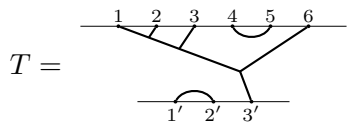
\[
T=\ 
\tik{0.5}{%
\node[punto] (a1) at (1,0)  [label=above: {\tiny $1$}] {};
\node[punto] (a2) at (2,0)  [label=above: {\tiny $2$}] {};
\node[punto] (a3) at (3,0)  [label=above: {\tiny $3$}] {};
\node[punto] (a4) at (4,0)  [label=above: {\tiny $4$}] {};
\node[punto] (a5) at (5,0)  [label=above: {\tiny $5$}] {};
\node[punto] (a6) at (6,0)  [label=above: {\tiny $6$}] {};
\node[punto] (b1) at (2.5,-2)  [label=below: {\tiny $1'$}] {};
\node[punto] (b2) at (3.5,-2)  [label=below: {\tiny $2'$}] {};
\node[punto] (b3) at (4.5,-2)  [label=below: {\tiny $3'$}] {};
\draw[thick] (a1) -- (1.8,-0.3) -- (2.6,-0.6) -- (4.2,-1.2) -- (b3);
\draw[thick] (a2) -- (1.8,-0.3) (a3) -- (2.6,-0.6) (a6) -- (4.2,-1.2);
\draw[thick] (b1) to [bend left=60] (b2);
\draw[thick] (a4) to [bend right=60] (a5);
\draw[very thin] (0,0) -- (7,0) (1.5,-2) -- (5.5,-2);
}
\]

The number of normalized $3$-tangles $[n]\rightarrow [m]$ is the number $a(n+m)$ of \emph{Catalan partitions} of $n+m$, i.e., partitions of $n+m$ points around a circle into a disjoint union of subsets whose convex hulls are disjoint and such that all subsets have at least two elements.   The sequence $a(n)$ is defined recursively by 
\begin{equation}\label{eq:Catalan}
a(0)=1,\ a(1)=0,\quad a(n)=\frac{(n-1)\bigl(2a(n-1)+3a(n-2)\bigr)}{n+1},
\end{equation} 
and the terms $a(2n)$ are the sequence $\#$A099251  in \cite{OEIS}.  
The Catalan partition corresponding to the normalized $3$-tangle above is pictured as follows:
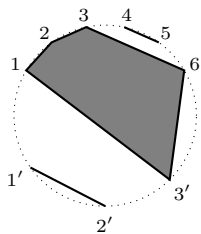
\[
\tik{0.8}{%
\draw[thin, dotted] (0,0) circle (1.5cm);
\path (150:1.5cm) coordinate (a1);
\path (126:1.5cm) coordinate (a2);
\path (102:1.5cm) coordinate (a3);
\path (78:1.5cm) coordinate (a4);
\path (54:1.5cm) coordinate (a5);
\path (30:1.5cm) coordinate (a6);
\node at (150:1.7cm) {\tiny $1$};
\node at (126:1.7cm) {\tiny $2$};
\node at (102:1.7cm) {\tiny $3$};
\node at (78:1.7cm) {\tiny $4$};
\node at (54:1.7cm) {\tiny $5$};
\node at (30:1.7cm) {\tiny $6$};
\path (-145:1.5cm) coordinate (b1);
\path (-90:1.5cm) coordinate (b2);
\path (-45:1.5cm) coordinate (b3);
\node at (-145:1.8cm) {\tiny $1'$};
\node at (-90:1.8cm) {\tiny $2'$};
\node at (-45:1.8cm) {\tiny $3'$};
\fill[color=gray] (a1) -- (a2) -- (a3) -- (a6) -- (b3) -- cycle;
\draw[thick] (a1) -- (a2) -- (a3) -- (a6) -- (b3) -- cycle;
\draw[thick] (a4) -- (a5) (b1) -- (b2);
}
\]

\begin{theorem}\label{th:sl2}
Let $n,m\in\NN$,  and assume that the characteristic of $\FF$ is $0$.   Let $\frV=(\VV,\mathsf{b},\times)$ be a $3$-dimensional vector space $\VV$ endowed
with a nonzero cross product $x\times y$ relative to the nondegenerate symmetric bilinear form $\mathsf{b}$, and let $\wh\Gamma = 
\{\wh \gamma_0, \wh \gamma_1, \wh \gamma_2\}$, where $\wh\gamma_0$ is as in \eqref{eq:whgamma0}, $\wh\gamma_1$ as
in \eqref{eq:whgamma1}, and $\wh\gamma_2$ as in \eqref{eq:whgamma2}.   
\begin{itemize}
\item[{\rm (a)}]  The classes modulo $\wh\Gamma$ of normalized $3$-tangles $[n]\rightarrow [m]$ form a basis of $\Mor_{\cT_{\wh\Gamma}}([n],[m])$.

\item[{\rm (b)}]     $\cR_{\wh\Gamma}$  gives a linear isomorphism 
\[
\Mor_{\cT_{\wh\Gamma}}([n],[m])\rightarrow \Hom_{\SO(\VV,\mathsf{b})}(\VV^{\otimes n},\VV^{\otimes m}).  
\]
\item[{\rm (c)}]  The normalized $3$-tangles $[n]\rightarrow [n]$ give a basis of the centralizer algebra $$\End_{\SO(\VV,\mathsf{b})}\bigl(\VV^{\otimes n}\bigr)\simeq \Mor_{\cT_{\wh\Gamma}}([n],[n]),$$ and $\dimm  \End_{\SO(\VV,\mathsf{b})}\bigl(\VV^{\otimes n}\bigr)$ equals the number $a(2n)$ of Catalan partitions. 

\end{itemize}
\end{theorem}
\begin{proof}
Extending scalars, we may assume that $\FF$ is algebraically closed.  
Then $(\VV,\mathsf{b})$ is isomorphic to the adjoint module $\frsl_2$ for the Lie algebra  $\frsl_2$ of
$2 \times 2$ matrices of trace 0 over $\FF$, because $\frsl_2$ is the orthogonal Lie algebra relative to its Killing form, which is
a nondegenerate symmetric bilinear form. 
Thus, $\Hom_{\SO(\VV,\mathsf{b})}(\VV^{\otimes n},\FF)\cong\Hom_{\frsl_2}(\frsl_2^{\otimes n},\FF)$, whose dimension also equals  the number $a(n)$ of Catalan partitions. The same argument as in Theorem \ref{th:G2} now gives the result.
\end{proof}

\begin{remark}
In particular, for $n=m$, Theorem \ref{th:sl2} also gives a basis of the centralizer algebra   $\End_{\frsl_2}(\frsl_2^{\otimes n})$ of
$\frsl_2$ over $\FF$ acting on tensor powers of its adjoint module.  Multiplication in this basis can be achieved as in Remark
\ref{rem:mult}.   
\end{remark}

\begin{remark}\label{re:det}
Relation $\wh\gamma_2$ implies that $\Mor_{\cT_{\wh\Gamma}}([n],[0])$ is spanned by $3$-tangles where crossings are allowed, and where each connected component has at most one trivalent node. If there are two connected components with one trivalent node, we can proceed as follows, using $\wh\gamma_2$ repeatedly:
\[
\begin{split}
&
\tik{0.3}{%
\path (1,0) coordinate (a1);
\path (2,0) coordinate (a2);
\path (3,0) coordinate (a3);
\path (4,0) coordinate (a4);
\path (5,0) coordinate (a5);
\path (6,0) coordinate (a6);
\draw[thick] (a1) .. controls (1,-0.5) and (1.5,-1) .. (2,-1);
\draw[thick] (a3) .. controls (3,-0.5) and (2.5,-1) .. (2,-1);
\draw[thick] (a2) -- (2,-1);
\draw[thick] (a4) .. controls (4,-0.5) and (4.5,-1) .. (5,-1);
\draw[thick] (a6) .. controls (6,-0.5) and (5.5,-1) .. (5,-1);
\draw[thick] (a5) -- (5,-1);
\draw[very thin] (0,0) -- (7,0);
}
\ = \ 
\tik{0.3}{%
\path (1,0) coordinate (a1);
\path (2,0) coordinate (a2);
\path (3,0) coordinate (a3);
\path (4,0) coordinate (a4);
\path (5,0) coordinate (a5);
\path (6,0) coordinate (a6);
\draw[thick] (a1) to (1,-3);
\draw[thick] (1,-3) .. controls (1,-3.5) and (1,-3.5) .. (1.8,-3.5);
\draw[thick] (2.5,-3) .. controls (2.5,-3.5) and (2.5,-3.5) .. (1.8,-3.5);
\draw[thick] (2.5,-3) -- (2.5,-0.7) (2.5,-0.7) -- (a2) (2.5,-0.7) -- (a3);
\draw[thick] (a4) to (4,-3);
\draw[thick] (4,-3) .. controls (4,-3.5) and (4,-3.5) .. (4.8,-3.5);
\draw[thick] (5.5,-3) .. controls (5.5,-3.5) and (5.5,-3.5) .. (4.8,-3.5);
\draw[thick] (5.5,-3) -- (5.5,-0.7) (5.5,-0.7) -- (a5) (5.5,-0.7) -- (a6);
\draw[very thin] (0,0) -- (7,0);
\draw[very thin,dotted] (0,-1.5) -- (7,-1.5);
\draw[very thin,dotted] (0,-3) -- (7,-3);
}
\ =\ 
\tik{0.3}{%
\path (1,0) coordinate (a1);
\path (2,0) coordinate (a2);
\path (3,0) coordinate (a3);
\path (4,0) coordinate (a4);
\path (5,0) coordinate (a5);
\path (6,0) coordinate (a6);
\draw[thick] (a1) -- (1,-3);
\draw[thick] (1,-3) .. controls (1,-3.5) and (1,-3.5) .. (1.8,-3.5);
\draw[thick] (2.5,-3) .. controls (2.5,-3.5) and (2.5,-3.5) .. (1.8,-3.5);
\draw[cross] (2.5,-0.7) -- (2.5,-1.5) -- (4,-3);
\draw[thick] (a2) -- (2.5,-0.7) -- (a3);
\draw[cross] (2.5,-3) -- (4,-1.5);
\draw[thick] (a4) -- (4,-1.5);
\draw[thick] (4,-3) .. controls (4,-3.5) and (4,-3.5) .. (4.8,-3.5);
\draw[thick] (5.5,-3) .. controls (5.5,-3.5) and (5.5,-3.5) .. (4.8,-3.5);
\draw[thick] (5.5,-3) -- (5.5,-0.7) (5.5,-0.7) -- (a5) (5.5,-0.7) -- (a6);
\draw[very thin] (0,0) -- (7,0);
}
\ -\ 
\tik{0.3}{%
\path (1,0) coordinate (a1);
\path (2,0) coordinate (a2);
\path (3,0) coordinate (a3);
\path (4,0) coordinate (a4);
\path (5,0) coordinate (a5);
\path (6,0) coordinate (a6);
\draw[thick] (a1) -- (1,-3);
\draw[thick] (1,-3) .. controls (1,-3.5) and (1,-3.5) .. (1.8,-3.5);
\draw[thick] (2.5,-3) .. controls (2.5,-3.5) and (2.5,-3.5) .. (1.8,-3.5);
\draw[thick] (a2) -- (2.5,-0.7) -- (a3);
\draw[thick] (2.5,-0.7) -- (2.5,-1.5) -- (3.3,-1.9) -- (3.3,-2.6) -- (2.5,-3);
\draw[thick] (a4) -- (4,-1.5) -- (3.3,-1.9);
\draw[thick] (3.3,-2.6) -- (4,-3);
\draw[thick] (5.5,-3) -- (5.5,-0.7) (5.5,-0.7) -- (a5) (5.5,-0.7) -- (a6);
\draw[thick] (4,-3) .. controls (4,-3.5) and (4,-3.5) .. (4.8,-3.5);
\draw[thick] (5.5,-3) .. controls (5.5,-3.5) and (5.5,-3.5) .. (4.8,-3.5);
\draw[very thin] (0,0) -- (7,0);
}
\\[10pt]
&\qquad =\
\tik{0.3}{%
\path (1,0) coordinate (a1);
\path (2,0) coordinate (a2);
\path (3,0) coordinate (a3);
\path (4,0) coordinate (a4);
\path (5,0) coordinate (a5);
\path (6,0) coordinate (a6);
\draw[thick] (a3) .. controls (3,-1) and (6,-1) .. (a6);
\draw[cross] (a2) .. controls (2,-1.8) and (5,-1.8) .. (a5);
\draw[cross] (a1) .. controls (1,-2) and (4,-2) .. (a4);
\draw[very thin] (0,0) -- (7,0);
}
- 
\tik{0.3}{%
\path (1,0) coordinate (a1);
\path (2,0) coordinate (a2);
\path (3,0) coordinate (a3);
\path (4,0) coordinate (a4);
\path (5,0) coordinate (a5);
\path (6,0) coordinate (a6);
\draw[thick] (a3) .. controls (3,-1) and (5,-1) .. (a5);
\draw[cross] (a2) .. controls (2,-1.8) and (6,-1.8) .. (a6);
\draw[cross] (a1) .. controls (1,-2) and (4,-2) .. (a4);
\draw[very thin] (0,0) -- (7,0);
}
- 
\tik{0.3}{%
\path (1,0) coordinate (a1);
\path (2,0) coordinate (a2);
\path (3,0) coordinate (a3);
\path (4,0) coordinate (a4);
\path (5,0) coordinate (a5);
\path (6,0) coordinate (a6);
\draw[thick] (a5) -- (5.5,-0.5) -- (a6);
\draw[thick] (a1) -- (2,-1) -- (a3);
\draw[thick] (2,-1) .. controls (2,-2) and (5.5,-2) .. (5.5,-0.5);
\draw[cross] (a2) .. controls (2,-1) and (4,-1) .. (a4);
\draw[very thin] (0,0) -- (7,0);
}
+
\tik{0.3}{%
\path (1,0) coordinate (a1);
\path (2,0) coordinate (a2);
\path (3,0) coordinate (a3);
\path (4,0) coordinate (a4);
\path (5,0) coordinate (a5);
\path (6,0) coordinate (a6);
\draw[thick] (a5) -- (5.5,-0.5) -- (a6);
\draw[thick] (a1) -- (1.5,-0.5) -- (a2);
\draw[thick] (a3) .. controls (3,-1) and (4,-1) .. (a4);
\draw[cross] (1.5,-0.5) .. controls (1.5,-2) and (5.5,-2) .. (5.5,-0.5);
\draw[very thin] (0,0) -- (7,0);
}
\\[10pt]
&\qquad =\ 
\tik{0.3}{%
\path (1,0) coordinate (a1);
\path (2,0) coordinate (a2);
\path (3,0) coordinate (a3);
\path (4,0) coordinate (a4);
\path (5,0) coordinate (a5);
\path (6,0) coordinate (a6);
\draw[thick] (a3) .. controls (3,-1) and (6,-1) .. (a6);
\draw[cross] (a2) .. controls (2,-1.8) and (5,-1.8) .. (a5);
\draw[cross] (a1) .. controls (1,-2) and (4,-2) .. (a4);
\draw[very thin] (0,0) -- (7,0);
}
-
\tik{0.3}{%
\path (1,0) coordinate (a1);
\path (2,0) coordinate (a2);
\path (3,0) coordinate (a3);
\path (4,0) coordinate (a4);
\path (5,0) coordinate (a5);
\path (6,0) coordinate (a6);
\draw[thick] (a3) .. controls (3,-1) and (5,-1) .. (a5);
\draw[cross] (a2) .. controls (2,-1.8) and (6,-1.8) .. (a6);
\draw[cross] (a1) .. controls (1,-2) and (4,-2) .. (a4);
\draw[very thin] (0,0) -- (7,0);
}
-
\tik{0.3}{%
\path (1,0) coordinate (a1);
\path (2,0) coordinate (a2);
\path (3,0) coordinate (a3);
\path (4,0) coordinate (a4);
\path (5,0) coordinate (a5);
\path (6,0) coordinate (a6);
\draw[thick] (a3) .. controls (3,-1) and (6,-1) .. (a6);
\draw[cross] (a2) .. controls (2,-1.2) and (4,-1.2) .. (a4);
\draw[cross] (a1) .. controls (1,-2) and (5,-2) .. (a5);
\draw[very thin] (0,0) -- (7,0);
}
\\
&\qquad\qquad +
\tik{0.3}{%
\path (1,0) coordinate (a1);
\path (2,0) coordinate (a2);
\path (3,0) coordinate (a3);
\path (4,0) coordinate (a4);
\path (5,0) coordinate (a5);
\path (6,0) coordinate (a6);
\draw[thick] (a3) .. controls (3,-1) and (5,-1) .. (a5);
\draw[cross] (a2) .. controls (2,-1.2) and (4,-1.2) .. (a4);
\draw[cross] (a1) .. controls (1,-2) and (6,-2) .. (a6);
\draw[very thin] (0,0) -- (7,0);
}
+
\tik{0.3}{%
\path (1,0) coordinate (a1);
\path (2,0) coordinate (a2);
\path (3,0) coordinate (a3);
\path (4,0) coordinate (a4);
\path (5,0) coordinate (a5);
\path (6,0) coordinate (a6);
\draw[thick] (a3) .. controls (3,-0.5) and (4,-0.5) .. (a4);
\draw[thick] (a2) .. controls (2,-1) and (6,-1) .. (a6);
\draw[cross] (a1) .. controls (1,-2) and (5,-2) .. (a5);
\draw[very thin] (0,0) -- (7,0);
}
-
\tik{0.3}{%
\path (1,0) coordinate (a1);
\path (2,0) coordinate (a2);
\path (3,0) coordinate (a3);
\path (4,0) coordinate (a4);
\path (5,0) coordinate (a5);
\path (6,0) coordinate (a6);
\draw[thick] (a3) .. controls (3,-0.5) and (4,-0.5) .. (a4);
\draw[thick] (a2) .. controls (2,-1) and (5,-1) .. (a5);
\draw[thick] (a1) .. controls (1,-2) and (6,-2) .. (a6);
\draw[very thin] (0,0) -- (7,0);
}
\end{split}
\]
This corresponds, under $\cR_{\cT_{\wh\Gamma}}$ to the identity
\[
\mathsf{b}( x_1\times x_2,x_3)\,\mathsf{b}( y_1\times y_2,y_3)=\det\begin{pmatrix} \mathsf{b}( x_i,y_j)\end{pmatrix}.
\]

Hence, $\Mor_{\cT_{\wh\Gamma}}([n],[0])$ is spanned by $3$-tangles, where there is at most one connected component with a trivalent node (and in this case $n$ must be odd). This reflects the fact that the algebra of invariants $\FF[n\VV]^{\SO(\VV,\mathsf{b})}$ for a $3$-dimensional cross product algebra $\frV =(\VV,\mathsf{b},\times)$  is generated by the polynomial maps $\mathsf{b}( x_i,x_j)$ and $\det(x_i,x_j,x_k)=\mathsf{b}( x_i\times x_j,x_k)$,  and, if $n$ is even, $\FF[n\VV]^{\SO(\VV,\mathsf{b})}=\FF[n\VV]^{\Ort(\VV,\mathsf{b})}$ is generated simply by the maps $\mathsf{b}( x_i,x_j)$, where
$i,j,k \in \{1,2,3\}$.  In particular,  the isomorphism $\Mor_{\cT_{\wh\Gamma}}([n],[n])\cong \Mor_{\cT_{\wh\Gamma}}([2n],[0])$
shows that the $3$-tangles can be folded to give 2-row Brauer algebra diagrams and 
reflects the fact that for $n =0,1,\dots$, the centralizer algebra $\End_{\SO(\VV,\mathsf{b})}(\VV^{\otimes n})=\End_{\Ort(\VV,\mathsf{b})}(\VV^{\otimes n})$ is a homomorphic image of the Brauer algebra $\mathsf{B}_n(3)$. (Further details on Brauer algebras can be found in \cite{Bra37}.)     
\end{remark}

%
%

\section{A $(1\,|\, 2)$-dimensional cross product}\label{se:dim1_2}

The $3$-dimensional \emph{Kaplansky superalgebra}  
over a field $\FF$ of characteristic $\ne 2$ is a Jordan superalgebra $\VV$ with one-dimensional even part $\VV\subo=\FF e$ and two-dimensional odd part $\VV\subuno =\FF p\oplus\FF q$, and with supercommutative multiplication given by:
\begin{equation}\label{eq:K3}
e\times e=e,\ e\times x=x\times e=\frac{1}{2}x\ \forall x\in \VV\subuno,\
p\times p=q\times q=0,\ p\times q=-q\times p=e.
\end{equation}
For any two homogeneous elements $x,y\in \VV$, the supercommutativity of the product gives $x\times y=(-1)^{xy}y\times x$, where $(-1)^{xy}$ is $-1$ if $x$ and $y$ are both odd, and it is $1$ otherwise.

Consider the even nondegenerate supersymmetric bilinear form \\ $\mathsf{b}:\VV\times \VV\rightarrow \FF$ such that
\[
\mathsf{b}( e,e)=\frac{1}{2},\quad \mathsf{b}( p,q)=1.
\]
The evenness of $\mathsf{b}$ means $\mathsf{b}( \VV\subo,\VV\subuno)=0=\mathsf{b}( \VV\subuno,\VV\subo)$, and the supersymmetry means that $\mathsf{b}$ is symmetric on $\VV\subo$ and skew symmetric on $\VV\subuno$. Hence, equation \eqref{eq:K3} may be rewritten as
\begin{equation}\label{eq:K3bis}
e\times e=e,\quad e\times x=x\times e=\frac{1}{2}x,\quad x\times y=\mathsf{b}( x,y)e,
\end{equation}
for all $x,y\in \VV\subuno$.

\begin{lemma}\label{le:K3} Let $\VV$ be the $3$-dimensional Kaplansky superalgebra.  Assume $\{x_1,x_2,x_3\}$ and $\{y_1,y_2,y_3\}$ are dual bases of $\VV$ relative to $\mathsf{b}$ (i.e. $\mathsf{b}( x_i,y_j)=\delta_{i,j}$)  
consisting of homogeneous elements.
\begin{romanenumerate}
\item The form $\mathsf{b}$ is associative.

\item $\sum_{i=1}^3 \mathsf{b}( y_i,x_i)=-1$ and $\sum_{i=1}^3 y_i\times x_i=0$.

\item For any homogeneous elements $z_1,z_2,z_3,z_4\in \VV$,
\begin{multline*}
\mathsf{b}( z_1\times z_2,z_3\times z_4)=\mathsf{b}( z_1,z_2)\mathsf{b}( z_3,z_4)\\
+(-1)^{z_2z_3}\frac{1}{2}\mathsf{b}( z_1,z_3)\mathsf{b}( z_2,z_4)+\frac{1}{2}\mathsf{b}( z_1,z_4)\mathsf{b}( z_2,z_3).
\end{multline*}
(Note that $\mathsf{b}( z_1,z_4)\mathsf{b}( z_2,z_3)=(-1)^{z_4z_2}(-1)^{z_4z_3}\mathsf{b}( z_1,z_4)\mathsf{b}( z_2,z_3)$, because $\mathsf{b}( z_2,z_3)=0$ unless $z_2$ and $z_3$ have the same parity.)
\end{romanenumerate}
\end{lemma}
\begin{proof}
The expressions $\sum_{i=1}^3\mathsf{b}( y_i,x_i)$ and $\sum_{i=1}^3y_i\times x_i$ do not depend on the dual bases chosen, so we may take $x_1=e=\frac{1}{2}y_1$, $x_2=p=-y_3$, $x_3=q=y_2$. The result now follows by straightforward computations.
\end{proof}

Consider the triple $\frV=(\VV,\mathsf{b},\times)$, but now denote by $\tau$ the super-switch map: $\tau: \VV^{\otimes 2}\rightarrow \VV^{\otimes 2}$, $x\otimes y\mapsto (-1)^{xy}y\otimes x$, for any homogeneous elements $x,y\in \VV$. As in Theorem \ref{th:RV}, there is a unique functor $\cR_{\frV}:\cT\rightarrow \cV$ with the same properties expressed there (but taking into account that now $\tau$ is the super-switch map).

Consider also the following linear maps (where $\mm(x\otimes y)=x\times y$):  
\[
\begin{split}
&\mathsf{c_0^s}\in\Hom_\FF(\VV^{\otimes 0},\VV^{\otimes 0}):\ 1\mapsto (\mathsf{b}\circ \mathsf{b}^t)(1)+1,\\
&\mathsf{c_1^s}\in\Hom_\FF(\VV^{\otimes 2},\VV):\ x\otimes y\mapsto (1-\tau)\circ \mm(x\otimes y)=x\times y-(-1)^{xy}y\times x,\\
&\mathsf{c_2^s}\in\Hom_\FF(\VV^{\otimes 3},\VV):\ x\otimes y\otimes z\mapsto (x\times y)\times z 
     - \mathsf{b}( x,y)z\\[-2pt]
     &\hspace*{2.7in}+(-1)^{yz}\frac{1}{2}\mathsf{b}( x,z)y+\frac{1}{2}\mathsf{b}( y,z)x,\\
&\mathsf{c_3^s}\in\Hom_\FF(\FF,\VV):\ 1\mapsto (\mm \circ \mathsf{b}^t)(1),
\end{split}
\]
for homogeneous elements $x,y,z\in \VV$. The supercommutativity of $\VV$ shows that $\mathsf{c_1^s}=0$, and Lemma \ref{le:K3} and the nondegeneracy of $\mathsf{b}$ prove that $\mathsf{c_0^s}=0$, $\mathsf{c_2^s}=0$, and $\mathsf{c_3^s}=0$.

The linear maps $\mathsf{c_0^s}$, $\mathsf{c_1^s}$, $\mathsf{c_2^s}$, and $\mathsf{c_3^s}$, are the images under $\cR_\frV$ of:
\begin{equation}\label{eq:gammas}
\begin{split}
\gamma_0^{\mathsf s}&=\ 
\tik{0.5}{%
\draw[thick] (0,0) circle (0.5cm);
}
+1
\\[10pt]
\gamma_1^{\mathsf s}&=\ 
\tik{0.5}{%
\path (0,0) coordinate (origin);
\path (150:0.5cm) coordinate (t1);
\path (30:0.5cm) coordinate (t2);
\path (-90: 0.5cm) coordinate (t3);
\path (150:1cm) coordinate (e1);
\path (30:1cm) coordinate (e2);
\path (-90:1cm) coordinate (e3);
\path (0,0.2) coordinate (m);
\draw[thick] (e1) -- (origin) -- (e2);
\draw[thick] (origin)-- (e3);
}
\ -\ 
\tik{0.5}{%
\path (0,0) coordinate (origin);
\path (150:0.5cm) coordinate (t1);
\path (30:0.5cm) coordinate (t2);
\path (-90: 0.5cm) coordinate (t3);
\path (150:1cm) coordinate (e1);
\path (30:1cm) coordinate (e2);
\path (-90:1cm) coordinate (e3);
\path (0,0.2) coordinate (m);
\path (0,-0.5) coordinate (p);
\draw[thick] (e2) .. controls (t1) and (-1,0) .. (p);
\draw[cross] (e1) .. controls (t2) and (1,0) .. (p);
\draw[thick] (e3) -- (p);
}
\\[10pt]
\tilde\gamma_2^{\mathsf s}&=\ 
\tik{0.5}{%
\path (1,0) coordinate (a1);
\path (2,0) coordinate (a2);
\path (3,0) coordinate (a3);
\path (2,-2) coordinate (b1);
\path (1.5,-0.5) coordinate (m1);
\path (2,-1.2) coordinate (m2);
\draw[thick] (a1) -- (m1) (a2) -- (m1) (a3)-- (m2) (m1) -- (m2) (m2) -- (b1);
}
\ -\ 
\tik{0.5}{%
\path (1,0) coordinate (a1);
\path (2,0) coordinate (a2);
\path (3,0) coordinate (a3);
\path (2,-2) coordinate (b1);
\draw[thick] (a3) -- (b1);
\draw[thick] (a1) to [bend right=60] (a2);
}
\ -\frac{1}{2}
\tik{0.5}{%
\path (1,0) coordinate (a1);
\path (2,0) coordinate (a2);
\path (3,0) coordinate (a3);
\path (2,-2) coordinate (b1);
\draw[thick] (a2) -- (b1);
\draw[cross] (a1) .. controls (1,-1) and (3,-1) .. (a3);
}
\ -\frac{1}{2}\ 
\tik{0.5}{%
\path (1,0) coordinate (a1);
\path (2,0) coordinate (a2);
\path (3,0) coordinate (a3);
\path (2,-2) coordinate (b1);
\draw[thick] (a1) -- (b1);
\draw[thick] (a2) to [bend right=60] (a3);
}
\\[10pt]
\gamma_3^{\mathsf s}&=\ 
\tik{0.5}{%
\draw[thick] (0,0) -- (0,-1);
\draw[thick] (0,0.5) circle (0.5cm);
}.
\end{split}
\end{equation}  

Instead of $\tilde\gamma_2^{\mathsf s}$, let us consider the element $\gamma_2^{\mathsf s}\in\Mor_{\cT_2}([2],[2])$ given by $\gamma_2^{\mathsf s}=(\tilde\gamma_2^{\mathsf s}\sqcup  \II_1)\circ (\II_2\sqcup \beta^t)$:
\begin{equation}\label{eq:whgamma2s}
\gamma_2^{\mathsf s}=\ 
\tik{0.5}{%
\path (1,0) coordinate (a1);
\path (3,0) coordinate (a2);
\path (1,-2) coordinate (b1);
\path (3,-2) coordinate (b2);
\draw[thick] (a1) -- (2,-0.6) -- (a2);
\draw[thick] (b1) -- (2,-1.4) -- (b2);
\draw[thick] (2,-0.6) -- (2,-1.4);
}
\ -\ 
\tik{0.5}{%
\path (1,0) coordinate (a1);
\path (3,0) coordinate (a2);
\path (1,-2) coordinate (b1);
\path (3,-2) coordinate (b2);
\draw[thick] (a1) to [bend right=60]  (a2);
\draw[thick] (b1) to [bend left=60] (b2);
}
\ -\frac{1}{2}\ 
\tik{0.5}{%
\path (1,0) coordinate (a1);
\path (3,0) coordinate (a2);
\path (1,-2) coordinate (b1);
\path (3,-2) coordinate (b2);
\draw[thick] (a2) -- (b1);
\draw[cross] (a1) -- (b2);
}
\ -\frac{1}{2}\ 
\tik{0.5}{%
\path (1,0) coordinate (a1);
\path (3,0) coordinate (a2);
\path (1,-2) coordinate (b1);
\path (3,-2) coordinate (b2);
\draw[thick] (a1) to [bend left=60] (b1);
\draw[thick] (a2) to [bend right=60] (b2);
}
\end{equation}

As in Corollary \ref{co:RGamma}, if 
 $\Gamma^{\mathsf s}=\{\gamma_0^{\mathsf s},\gamma_1^{\mathsf s},\gamma_2^{\mathsf s},\gamma_3^{\mathsf s}\}$, there is a unique functor $\cR_{\Gamma^{\mathsf s}}:\cT_{\Gamma^{\mathsf s}}\rightarrow \cV$ such that $\cR_\cV=\cR_{\Gamma^{\mathsf s}}\circ\cP$, with $\cP$ the natural projection $\cT \rightarrow \cT_{\Gamma^{\mathsf s}}$.

The Lie superalgebra of derivations of $(\VV,\times)$ is the orthosymplectic superalgebra $\frosp(\VV,\mathsf{b})\simeq \frosp_{1|2}$ over
$\FF$, and hence,  $\cR_{\Gamma^{\mathsf s}}\bigl(\Mor_{\cT_{\Gamma^{\mathsf s}}}([n],[m])\bigr)$ is contained in $\Hom_{\frosp(\VV,\mathsf{b})}(\VV^{\otimes n},\VV^{\otimes m})$.

By the arguments of Section \ref{se:3_dim},  $\gamma_0^{\mathsf s}$ allows us to get rid of circles and $\gamma_2^{\mathsf s}$ of crossings. Also, $\gamma_2^{\mathsf s}$ gives:
\[
\begin{split}
\tik{0.5}{%
\path (1,0) coordinate (a1);
\path (2,0) coordinate (a2);
\path (3,0) coordinate (a3);
\path (1,-2) coordinate (b1);
\path (2,-2) coordinate (b2);
\path (3,-2) coordinate (b3);
\draw[thick] (a2) -- (2,-0.5) (2,-1.5) -- (b2);
\draw[thick] (2,-1) circle (0.5cm);
}
\ &=\ 
\tik{0.5}{%
\path (1,0) coordinate (a1);
\path (2,0) coordinate (a2);
\path (3,0) coordinate (a3);
\path (1,-2) coordinate (b1);
\path (2,-2) coordinate (b2);
\path (3,-2) coordinate (b3);
\draw[thick] (a3) -- (2,-0.5) -- (2,-1.5) -- (b3);
\draw[thick] (2,-0.5) .. controls (a1) and (b1) .. (2,-1.5);
}
\ =\ 
\tik{0.5}{%
\path (1,0) coordinate (a1);
\path (2,0) coordinate (a2);
\path (3,0) coordinate (a3);
\path (1,-2) coordinate (b1);
\path (2,-2) coordinate (b2);
\path (3,-2) coordinate (b3);
\draw[thick] (a3) -- (2,-0.5) (2,-1.5) -- (b3);
\draw[thick] (2,-0.5) .. controls (a1) and (b1) .. (2,-1.5);
}
\ +\ \frac{1}{2}
\tik{0.5}{%
\path (1,0) coordinate (a1);
\path (2,0) coordinate (a2);
\path (3,0) coordinate (a3);
\path (1,-2) coordinate (b1);
\path (2,-2) coordinate (b2);
\path (3,-2) coordinate (b3);
\draw[thick] (1.5,-0.3) .. controls (a1) and (b1) .. (1.5,-1.7);
\draw[thick] (a3) -- (1.5,-1.7);
\draw[cross] (1.5,-0.3) -- (b3);
}
\ +\frac{1}{2}\ 
\tik{0.5}{%
\draw[thick] (0,0) circle (0.5cm);
\draw[thick] (0.8,1) -- (0.8,-1);
}
\\[10pt]
&=\ \left(1+\frac{1}{2}-\frac{1}{2}\right)\ 
\tik{0.5}{%
\draw[thick] (0,1) -- (0,-1);
}
\ =\ 
\tik{0.5}{%
\draw[thick] (0,1) -- (0,-1);
}
\end{split}
\]

Now from $\gamma_2^{\mathsf s}$ and the $3$-tangle obtained rotating it,  we get the following relation in $\cT_{\Gamma^{\mathsf s}}$:
\begin{equation}\label{eq:assoc3}
\tik{0.5}{%
\path (1,0) coordinate (a1);
\path (3,0) coordinate (a2);
\path (1,-2) coordinate (b1);
\path (3,-2) coordinate (b2);
\draw[thick] (a1) -- (2,-0.6) -- (a2);
\draw[thick] (b1) -- (2,-1.4) -- (b2);
\draw[thick] (2,-0.6) -- (2,-1.4);
}
\ -\ 
\tik{0.5}{%
\path (1,0) coordinate (a1);
\path (3,0) coordinate (a2);
\path (1,-2) coordinate (b1);
\path (3,-2) coordinate (b2);
\draw[thick] (a1) -- (1.6,-1) -- (b1);
\draw[thick] (a2) -- (2.4,-1) -- (b2);
\draw[thick] (1.6,-1) -- (2.4,-1);
}\ =\ \frac{1}{2}\left[\ 
\tik{0.5}{%
\path (1,0) coordinate (a1);
\path (3,0) coordinate (a2);
\path (1,-2) coordinate (b1);
\path (3,-2) coordinate (b2);
\draw[thick] (a1) to [bend right=60]  (a2);
\draw[thick] (b1) to [bend left=60] (b2);
}
\ -\ \vrule width 0pt depth 20pt
\tik{0.5}{%
\path (1,0) coordinate (a1);
\path (3,0) coordinate (a2);
\path (1,-2) coordinate (b1);
\path (3,-2) coordinate (b2);
\draw[thick] (a1) to [bend left=60] (b1);
\draw[thick] (a2) to [bend right=60] (b2);
}
\,\right]
\end{equation}
(compare with \eqref{eq:assoc1}), and this allows us, as in Section \ref{se:3_dim}, to replace cycles by linear combinations of $3$-tangles without crossings, circles and cycles, and eventually to express any $3$-tangle in $\cT_{\Gamma^{\mathsf s}}$ as a linear combination of $3$-tangles which are disjoint unions of normalized connected $3$-tangles.

\begin{theorem}\label{th:Kap}
Let $\VV$ be the $3$-dimensional Kaplansky superalgebra  over a field $\FF$ of characteristic $0$, and assume $\mathsf{b}$ is the 
nondegenerate supersymmetric bilinear form on $\VV$. Let $n,m\in\NN$. Assume $\Gamma^{\mathsf s} =
\{\gamma_0^{\mathsf s},\gamma_1^{\mathsf s},\gamma_2^{\mathsf s},\gamma_3^{\mathsf s}\}$, where these $3$-tangles are as in \eqref{eq:gammas} and \eqref{eq:whgamma2s}.
\begin{itemize}
\item[{\rm (a)}]  The classes modulo $\Gamma^{\mathsf s}$ of the normalized $3$-tangles $[n]\rightarrow [m]$ form a basis of $\Mor_{\cT_{\Gamma^{\mathsf s}}}([n],[m])$.
\item[{\rm (b)}]  The functor $\cR_{\Gamma^{\mathsf s}}$ gives a linear isomorphism 
\[
\Mor_{\cT_{\Gamma^{\mathsf s}}}([n],[m])\rightarrow \Hom_{\frosp(\VV,\mathsf{b})}(\VV^{\otimes n},\VV^{\otimes m}).
\]
\item[{\rm (c)}] The normalized $3$-tangles  
$[n]\rightarrow [n]$ \ give a basis of the centralizer algebra $$\End_{\frosp(\VV,\mathsf{b})}\bigl(\VV^{\otimes n}\bigr)\simeq \Mor_{\cT_{\Gamma^{\mathsf s}}}([n],[n]),$$ and $\dimm  \End_{\frosp(\VV,\mathsf{b})}\bigl(\VV^{\otimes n}\bigr)$ equals the number $a(2n)$ of Catalan partitions.  \end{itemize}
\end{theorem}
\begin{proof} Extending scalars, we may assume that $\FF$ is algebraically closed.

Given any unital commutative associative superalgebra $R=R\subo\oplus R\subuno$ over $\FF$, let $\alpha\in R\subo$, $\beta,\gamma\in R\subuno$, and consider the even element $u=\alpha e+\beta p+\gamma q$ in the scalar extension $\VV\otimes_\FF R$. Then $\mathsf{b}( u,u)=\frac{1}{2}\alpha^2+2\beta\gamma$, so
\[
\mathsf{b}( u,u)^3=\frac{1}{8}\Bigl(\alpha^6+12\alpha^4\beta\gamma\Bigr)
\]
because $\beta^2=\gamma^2=0$, and
\[
\begin{split}
\mathsf{b}( u\times u,u)&=\mathsf{b}\bigl((\alpha^2+2\beta\gamma)e+\alpha\beta p+\alpha\gamma q,\alpha e+\beta p+\gamma q\bigr)\\
&=\frac{1}{2}(\alpha^2+2\beta\gamma)\alpha+2\alpha\beta\gamma=\frac{1}{2}(\alpha^3+6\alpha\beta\gamma),
\end{split}
\]
and hence we obtain
\[
\mathsf{b}( u\times u,u)^2=\frac{1}{4}(\alpha^6+12\alpha^4\beta\gamma).
\]
This means that, up to scalars, $\mathsf{b}( u\times u,u)$ is the \emph{super Pfaffian} (see \cite{Ser92}). The First Fundamental Theorem of Invariant Theory for the special orthosymplectic group $\mathsf{SOSp}(\VV,\mathsf{b})$ \cite{LZ15} shows that $\cR_{\Gamma^{\mathsf s}}$ gives a surjection
\[
\Mor_{\cT_{\Gamma^{\mathsf s}}}([n],[m])\rightarrow \Hom_{\mathsf{SOSp}(\VV,\mathsf{b})}(\VV^{\otimes n},\VV^{\otimes m})=\Hom_{\frosp(\VV,\mathsf{b})}(\VV^{\otimes n},\VV^{\otimes m}).
\]
Now  it follows from \cite{RS82} that the dimension of $\Hom_{\frosp(\VV,\mathsf{b})}(\VV^{\otimes n},\VV^{\otimes m})$ coincides with the dimension of $\Hom_{\frsl_2}(\frsl_2^{\otimes n},\frsl_2^{\otimes m})$ for the adjoint module for $\frsl_2\simeq \frso_3$, and this dimension equals the number of Catalan partitions of $n+m$, that is,  the number of normalized $3$-tangles $[n]\rightarrow [m]$ (see the comments before Theorem \ref{th:sl2}). Hence, the result holds.
\end{proof}
 
\begin{example}  Let us consider explicitly the basis of the centralizer algebra $\End_{\frosp(\VV,\mathsf{b})}(\VV^{\otimes 3})$ or, equivalently, of $\Mor_{\cT_{\Gamma^{\mathsf s}}}([3],[3])$. As in Remark \ref{re:det},  this is a homomorphic image of  the Brauer algebra  $\mathsf{B}_3(-1)$
 due to tangle relation $\gamma_0^{\mathsf s}$ (in fact, it is isomorphic to $\mathsf{B}_3(-1)$), but here a basis different from the
standard diagram basis of the Brauer algebra is obtained. According to Theorem \ref{th:Kap}, a basis of $\Mor_{\cT_{\Gamma^{\mathsf s}}}([3],[3])$ is given by the classes modulo $\Gamma^{\mathsf s}$ of the following $3$-tangles:

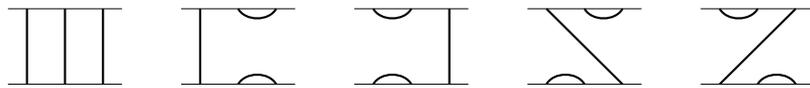
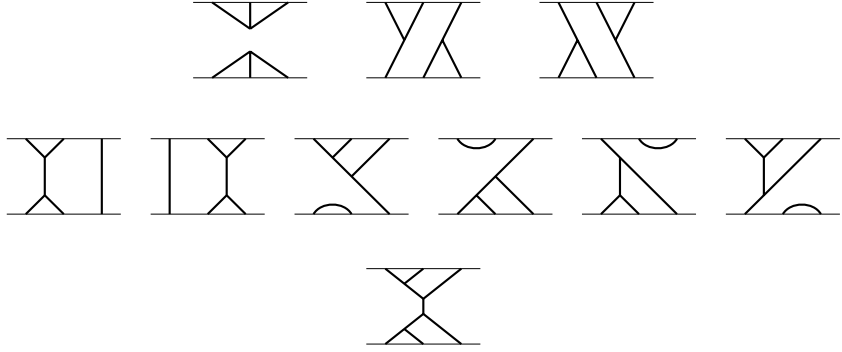
\[
\tik{0.5}{%
\draw[very thin] (0.5,0) to (3.5,0);
\draw[very thin] (0.5,-2) to (3.5,-2);
\path (1,0) coordinate (a1);
\path (2,0) coordinate (a2);
\path (3,0) coordinate (a3);
\path (1,-2) coordinate (b1);
\path (2,-2) coordinate (b2);
\path (3,-2) coordinate (b3);
\draw[thick] (a1) -- (b1) (a2) -- (b2) (a3) -- (b3);
}
\qquad
\tik{0.5}{%
\draw[very thin] (0.5,0) to (3.5,0);
\draw[very thin] (0.5,-2) to (3.5,-2);
\path (1,0) coordinate (a1);
\path (2,0) coordinate (a2);
\path (3,0) coordinate (a3);
\path (1,-2) coordinate (b1);
\path (2,-2) coordinate (b2);
\path (3,-2) coordinate (b3);
\draw[thick] (a1) -- (b1);
\draw[thick]  (a2) to [bend right=60] (a3);
\draw[thick] (b2) to [bend left=60]   (b3);
}
\qquad
\tik{0.5}{%
\draw[very thin] (0.5,0) to (3.5,0);
\draw[very thin] (0.5,-2) to (3.5,-2);
\path (1,0) coordinate (a1);
\path (2,0) coordinate (a2);
\path (3,0) coordinate (a3);
\path (1,-2) coordinate (b1);
\path (2,-2) coordinate (b2);
\path (3,-2) coordinate (b3);
\draw[thick] (a3) -- (b3);
\draw[thick]  (a1) to [bend right=60] (a2);
\draw[thick] (b1) to [bend left=60]   (b2);
}
\qquad
\tik{0.5}{%
\draw[very thin] (0.5,0) to (3.5,0);
\draw[very thin] (0.5,-2) to (3.5,-2);
\path (1,0) coordinate (a1);
\path (2,0) coordinate (a2);
\path (3,0) coordinate (a3);
\path (1,-2) coordinate (b1);
\path (2,-2) coordinate (b2);
\path (3,-2) coordinate (b3);
\draw[thick] (a1) -- (b3);
\draw[thick]  (a2) to [bend right=60] (a3);
\draw[thick] (b1) to [bend left=60]   (b2);
}
\qquad
\tik{0.5}{%
\draw[very thin] (0.5,0) to (3.5,0);
\draw[very thin] (0.5,-2) to (3.5,-2);
\path (1,0) coordinate (a1);
\path (2,0) coordinate (a2);
\path (3,0) coordinate (a3);
\path (1,-2) coordinate (b1);
\path (2,-2) coordinate (b2);
\path (3,-2) coordinate (b3);
\draw[thick] (a3) -- (b1);
\draw[thick]  (a1) to [bend right=60] (a2);
\draw[thick] (b2) to [bend left=60]   (b3);
}
\]

\bigskip

\[
\tik{0.5}{%
\draw[very thin] (0.5,0) to (3.5,0);
\draw[very thin] (0.5,-2) to (3.5,-2);
\path (1,0) coordinate (a1);
\path (2,0) coordinate (a2);
\path (3,0) coordinate (a3);
\path (1,-2) coordinate (b1);
\path (2,-2) coordinate (b2);
\path (3,-2) coordinate (b3);
\path (2,-0.7) coordinate (m1);
\path (2,-1.3) coordinate (m2);
\draw[thick] (a1) -- (m1) (a2) -- (m1) (a3) -- (m1);
\draw[thick] (b1) -- (m2) (b2) -- (m2) (b3) -- (m2);
}
\qquad
\tik{0.5}{%
\draw[very thin] (0.5,0) to (3.5,0);
\draw[very thin] (0.5,-2) to (3.5,-2);
\path (1,0) coordinate (a1);
\path (2,0) coordinate (a2);
\path (3,0) coordinate (a3);
\path (1,-2) coordinate (b1);
\path (2,-2) coordinate (b2);
\path (3,-2) coordinate (b3);
\draw[thick] (a2) -- (b1) (a1) -- (1.5,-1);
\draw[thick] (a3) -- (b2) (b3) -- (2.5,-1);
}
\qquad
\tik{0.5}{%
\draw[very thin] (0.5,0) to (3.5,0);
\draw[very thin] (0.5,-2) to (3.5,-2);
\path (1,0) coordinate (a1);
\path (2,0) coordinate (a2);
\path (3,0) coordinate (a3);
\path (1,-2) coordinate (b1);
\path (2,-2) coordinate (b2);
\path (3,-2) coordinate (b3);
\draw[thick] (a1) -- (b2) (b1) -- (1.5,-1);
\draw[thick] (a2) -- (b3) (a3) -- (2.5,-1);
}
\]

\bigskip

\[
\tik{0.5}{%
\draw[very thin] (0.5,0) to (3.5,0);
\draw[very thin] (0.5,-2) to (3.5,-2);
\path (1,0) coordinate (a1);
\path (2,0) coordinate (a2);
\path (3,0) coordinate (a3);
\path (1,-2) coordinate (b1);
\path (2,-2) coordinate (b2);
\path (3,-2) coordinate (b3);
\draw[thick] (a1) -- (1.5,-0.5) -- (1.5,-1.5) -- (b1);
\draw[thick] (a2) -- (1.5,-0.5) (b2) -- (1.5,-1.5);
\draw[thick] (a3) -- (b3);
}
\quad
\tik{0.5}{%
\draw[very thin] (0.5,0) to (3.5,0);
\draw[very thin] (0.5,-2) to (3.5,-2);
\path (1,0) coordinate (a1);
\path (2,0) coordinate (a2);
\path (3,0) coordinate (a3);
\path (1,-2) coordinate (b1);
\path (2,-2) coordinate (b2);
\path (3,-2) coordinate (b3);
\draw[thick] (a2) -- (2.5,-0.5) -- (2.5,-1.5) -- (b2);
\draw[thick] (a3) -- (2.5,-0.5) (b3) -- (2.5,-1.5);
\draw[thick] (a1) -- (b1);
}
\quad
\tik{0.5}{%
\draw[very thin] (0.5,0) to (3.5,0);
\draw[very thin] (0.5,-2) to (3.5,-2);
\path (1,0) coordinate (a1);
\path (2,0) coordinate (a2);
\path (3,0) coordinate (a3);
\path (1,-2) coordinate (b1);
\path (2,-2) coordinate (b2);
\path (3,-2) coordinate (b3);
\draw[thick] (a1) -- (b3)  (a2) -- (1.5,-0.5) (a3) -- (2,-1);
\draw[thick] (b1) to [bend left=60] (b2);
}
\quad
\tik{0.5}{%
\draw[very thin] (0.5,0) to (3.5,0);
\draw[very thin] (0.5,-2) to (3.5,-2);
\path (1,0) coordinate (a1);
\path (2,0) coordinate (a2);
\path (3,0) coordinate (a3);
\path (1,-2) coordinate (b1);
\path (2,-2) coordinate (b2);
\path (3,-2) coordinate (b3);
\draw[thick] (a3) -- (b1)  (b2) -- (1.5,-1.5) (b3) -- (2,-1);
\draw[thick] (a1) to [bend right=60] (a2);
}
\quad
\tik{0.5}{%
\draw[very thin] (0.5,0) to (3.5,0);
\draw[very thin] (0.5,-2) to (3.5,-2);
\path (1,0) coordinate (a1);
\path (2,0) coordinate (a2);
\path (3,0) coordinate (a3);
\path (1,-2) coordinate (b1);
\path (2,-2) coordinate (b2);
\path (3,-2) coordinate (b3);
\draw[thick] (a1) -- (b3)  (b1) -- (1.5,-1.5) -- (1.5,-0.5);
\draw[thick] (b2) -- (1.5,-1.5);
\draw[thick] (a2) to [bend right=60] (a3);
}
\quad
\tik{0.5}{%
\draw[very thin] (0.5,0) to (3.5,0);
\draw[very thin] (0.5,-2) to (3.5,-2);
\path (1,0) coordinate (a1);
\path (2,0) coordinate (a2);
\path (3,0) coordinate (a3);
\path (1,-2) coordinate (b1);
\path (2,-2) coordinate (b2);
\path (3,-2) coordinate (b3);
\draw[thick] (a3) -- (b1)  (a1) -- (1.5,-0.5) -- (1.5,-1.5);
\draw[thick] (a2) -- (1.5,-0.5);
\draw[thick] (b2) to [bend left=60] (b3);
}
\]

\bigskip

\[
\tik{0.5}{%
\draw[very thin] (0.5,0) to (3.5,0);
\draw[very thin] (0.5,-2) to (3.5,-2);
\path (1,0) coordinate (a1);
\path (2,0) coordinate (a2);
\path (3,0) coordinate (a3);
\path (1,-2) coordinate (b1);
\path (2,-2) coordinate (b2);
\path (3,-2) coordinate (b3);
\draw[thick] (a1) -- (2,-0.8) -- (2,-1.2) -- (b1);
\draw[thick] (a2) -- (1.5,-0.4) (a3) -- (2,-0.8) (b2) -- (1.5,-1.6) (b3) -- (2,-1.2);
}
\]

\bigskip

It is straightforward to translate this basis of normalized $3$-tangles to a basis of the centralizer algebra $\End_{\frosp(\VV,\mathsf{b})}(\VV^{\otimes 3})$. For example, 

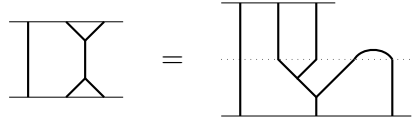
\[
\tik{0.5}{%
\draw[very thin] (0.5,0) to (3.5,0);
\draw[very thin] (0.5,-2) to (3.5,-2);
\path (1,0) coordinate (a1);
\path (2,0) coordinate (a2);
\path (3,0) coordinate (a3);
\path (1,-2) coordinate (b1);
\path (2,-2) coordinate (b2);
\path (3,-2) coordinate (b3);
\draw[thick] (a2) -- (2.5,-0.5) -- (2.5,-1.5) -- (b2);
\draw[thick] (a3) -- (2.5,-0.5) (b3) -- (2.5,-1.5);
\draw[thick] (a1) -- (b1);
}
\quad = \quad
\tik{0.5}{%
\draw[very thin] (0.5,1.5) to (3.5,1.5);
\draw[very thin,dotted] (0.5,0) to (5.5,0);
\draw[very thin] (0.5,-1.5) to (5.5,-1.5);
\path (1,1.5) coordinate (a1);
\path (2,1.5) coordinate (a2);
\path (3,1.5) coordinate (a3);
\path (1,-1.5) coordinate (b1);
\path (3,-1.5) coordinate (b2);
\path (5,-1.5) coordinate (b3);
\draw[thick] (a2) -- (2,0) -- (2.5,-0.5) -- (3,0) -- (a3);
\draw[thick] (2.5,-0.5) -- (3,-1) -- (b2);
\draw[thick] (3,-1) -- (4,0) (5,0) -- (b3);
\draw[thick] (4,0) to [bend left=60] (5,0);
\draw[thick] (a1) -- (b1);
}
\]

\noindent and hence, if $\{x_1,x_2,x_3\}$ and $\{y_1,y_2,y_3\}$ are two dual bases of $\VV$ consisting of homogeneous elements, then using Proposition \ref{pr:tPhiPsi}, which is valid in the super setting, we obtain that the corresponding element of the centralizer is the composition
\[
u\otimes v\otimes w\mapsto \sum_{i=1}^3 u\otimes v\otimes w\otimes y_i\otimes x_i\mapsto
\sum_{i=1}^3 u\otimes\bigl((v\times w)\times y_i\bigr)\otimes x_i.
\]

In the same vein, 

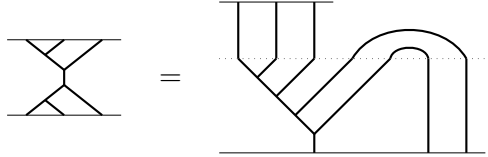
\[
\tik{0.5}{%
\draw[very thin] (0.5,0) to (3.5,0);
\draw[very thin] (0.5,-2) to (3.5,-2);
\path (1,0) coordinate (a1);
\path (2,0) coordinate (a2);
\path (3,0) coordinate (a3);
\path (1,-2) coordinate (b1);
\path (2,-2) coordinate (b2);
\path (3,-2) coordinate (b3);
\draw[thick] (a1) -- (2,-0.8) -- (2,-1.2) -- (b1);
\draw[thick] (a2) -- (1.5,-0.4) (a3) -- (2,-0.8) (b2) -- (1.5,-1.6) (b3) -- (2,-1.2);
}
\quad =\quad
\tik{0.5}{%
\draw[very thin] (0.5,1.5) to (3.5,1.5);
\draw[very thin,dotted] (0.5,0) to (7.5,0);
\draw[very thin] (0.5,-2.5) to (7.5,-2.5);
\path (1,1.5) coordinate (a1);
\path (2,1.5) coordinate (a2);
\path (3,1.5) coordinate (a3);
\path (3,-2.5) coordinate (b1);
\path (6,-2.5) coordinate (b2);
\path (7,-2.5) coordinate (b3);
\draw[thick] (a1) -- (1,0) (a2) -- (2,0) (a3) -- (3,0);
\draw[thick] (1,0) -- (3,-2) -- (b1);
\draw[thick] (2,0) -- (1.5,-0.5) (3,0) -- (2,-1) (4,0) -- (2.5,-1.5) (5,0) -- (3,-2);
\draw[thick] (6,0) -- (b2) (7,0) -- (b3);
\draw[thick] (4,0) to [bend left=60] (7,0);
\draw[thick] (5,0) to [bend left=80] (6,0);
}
\]

\noindent which corresponds to the map
\[
u\otimes v\otimes w\mapsto \sum_{i,j=1}^3\bigl((((u\times v)\times w)\times y_i)\times y_j\bigr)\otimes x_j\otimes x_i.
\]
\end{example}

\end{document}